\documentclass[11pt]{amsart}
\usepackage[dvipdfmx]{graphicx, color}
\usepackage{amssymb, amsmath,amsthm}
\usepackage{epstopdf}
\newtheorem{thm}{{\bf Theorem}}[section]
\newtheorem{lem}[thm]{{\bf Lemma}}
\newtheorem{cor}[thm]{{\bf Corollary}}
\newtheorem{prop}[thm]{{\bf Proposition}}

\newtheorem{rem}[thm]{Remark}

\newtheorem{ques}[thm]{Question}

\numberwithin{equation}{section}

\begin{document} 

\title[]{
	The asymptotic behavior of the minimal pseudo-Anosov dilatations in the hyperelliptic handlebody groups}
	\dedicatory{Dedicated to Professors Taizo Kanenobu, Yasutaka Nakanishi and Makoto Sakuma
for their sixtieth birthdays}

\author[S. Hirose]{%
Susumu Hirose
}
\address{%
Department of Mathematics,  
Faculty of Science and Technology, 
Tokyo University of Science, 
Noda, Chiba, 278-8510, Japan}
\email{%
hirose\_susumu@ma.noda.tus.ac.jp
}

\author[E. Kin]{%
    Eiko Kin
}
\address{%
       Department of Mathematics, Graduate School of Science, Osaka University Toyonaka, Osaka 560-0043, JAPAN
}
\email{%
        kin@math.sci.osaka-u.ac.jp
}



\date{%
\today
}

\thanks{%
The first author is supported by 
Grant-in-Aid for
Scientific Research (C) (No. 16K05156),
Japan Society for the Promotion of Science. 
The second author is supported by 
Grant-in-Aid for
Scientific Research (C) (No. 15K04875), 
Japan Society for the Promotion of Science. 
	} 
	
\begin{abstract} 
We consider 
the hyperelliptic handlebody group on a closed surface of genus $g$. 
This is the subgroup of the mapping class group on a closed surface of genus $g$ 
consisting of isotopy classes of homeomorphisms on the surface 
that commute with some fixed hyperelliptic involution 
and that extend to homeomorphisms on the handlebody. 
We prove that the logarithm of the minimal dilatation (i.e, the minimal entropy) of 
all pseudo-Anosov elements in the  hyperelliptic handlebody group of genus $g$ 
is comparable to $1/g$. 
This means that the asymptotic behavior of the minimal pseudo-Anosov dilatation of 
the subgroup of genus $g$ in question 
 is the same as that of the ambient mapping class group of genus $g$. 
 We also determine finite presentations of the hyperelliptic handlebody groups. 
\end{abstract}
\maketitle

\section{Introduction}
\label{section_Introduction}

Let $\varSigma_g$ be a closed, orientable surface of genus $g$, and 
let $\mathrm{Mod}(\varSigma_g)$ be the mapping class group on $\varSigma_g$. 
The {\it hyperelliptic mapping class group} $\mathcal{H}(\varSigma_g)$ is the subgroup of $\mathrm{Mod}(\varSigma_g)$ 
consisting of isotopy classes of orientation preserving homeomorphisms on $\varSigma_g$ 
that commute with some fixed hyperelliptic involution $\mathcal{S}: \varSigma_g \rightarrow \varSigma_g$. 
If $g \ge 3$, then $\mathcal{H}(\varSigma_g)$ is of infinite index in $\mathrm{Mod}(\varSigma_g)$, and 
it is a particular subgroup in some sense. 
Despite such a property, 
$\mathcal{H}(\varSigma_g)$ plays a significant role to study the mapping class group $\mathrm{Mod}(\varSigma_g)$. 
Especially, elements of $\mathcal{H}(\varSigma_g)$ have a handy description via the spherical braid group $SB_{2g+2}$ 
with $2g+2$ strings,  
which is proved by Birman-Hilden: 
$$ \mathcal{H} (\varSigma_g)/ \langle \iota \rangle 
\simeq SB_{2g+2}/ \langle \Delta^2 \rangle,$$
where 
$\iota = [\mathcal{S}] \in  \mathcal{H} (\varSigma_g)$ is the mapping class of $\mathcal{S}$, and 
$\Delta \in SB_{2g+2}$ is a half twist braid. 
Here $ \langle \iota \rangle$ and $\langle \Delta^2 \rangle$ are the subgroups generated by 
$\iota$ and $\Delta^2$ respectively. 
There exists a natural surjective homomorphism 
from $SB_{2g+2}$ to the mapping class group $\mathrm{Mod}(\varSigma_{0, 2g+2})$ 
on a sphere with $2g+2$ punctures:  
$$\Gamma: SB_{2g+2} \rightarrow \mathrm{Mod}(\varSigma_{0, 2g+2})$$ 
with the kernel generated by $\Delta^2$.

Let $G$ be a subgroup of $\mathrm{Mod}(\varSigma_g)$. 
Whenever  $G \cap \mathcal{H}(\varSigma_g)$ contains a non-trivial element, 
it is worthwhile to consider the subgroup 
$G \cap \mathcal{H}(\varSigma_g)$  of $\mathrm{Mod}(\varSigma_g)$. 
The group $G \cap \mathcal{H}(\varSigma_g)$ would be an intriguing one  in its own right. 
Also we may have a chance to find new examples or phenomena on $G$ 
by using a handy braid description of  $G \cap \mathcal{H}(\varSigma_g)$. 
In the case $G$ is the Torelli group $\mathcal{I}(\varSigma_g)$ 
consisting of elements of $\mathrm{Mod}(\varSigma_g)$ 
which act trivially on $H_1(\varSigma_g; {\Bbb Z})$, 
the hyperelliptic Torelli group $\mathcal{I}(\varSigma_g) \cap \mathcal{H}(\varSigma_g)$ 
is studied by Brendle-Margalit, see \cite{BM} and references therein.  
In this paper, we consider the {\it handlebody group} $\mathrm{Mod}({\Bbb H}_g)$ 
as $G$. 
This is the subgroup of $\mathrm{Mod}(\varSigma_g)$ 
consisting of isotopy classes of orientation preserving homeomorphisms on $\varSigma_g$ 
that extend to homeomorphisms on the  handlebody ${\Bbb H}_g$ of genus $g$. 
The main subgroup of $\mathrm{Mod}(\varSigma_g)$  in this paper 
is the {\it hyperelliptic handlebody group} 
$$\mathcal{H}({\Bbb H}_g) = \mathrm{Mod}({\Bbb H}_g) \cap \mathcal{H} (\varSigma_g).$$ 
We prove a version of Birman-Hilden's theorem about $\mathcal{H}({\Bbb H}_g)$, 
and identify the subgroup of $SB_{2g+2}$ 
corresponding to $\mathcal{H}({\Bbb H}_g)$. 
More precisely, we prove  in Theorem~\ref{thm_HiKi} that 
$$\mathcal{H}({\Bbb H}_g)/ \langle \iota \rangle \simeq SW_{2g+2}/\langle \Delta^2 \rangle,$$
where $SW_{2g+2}$ is so called the {\it wicket group}. 
(See Section~\ref{subsubsection_relation}.) 
Hilden introduced a subgroup $SH_{2g+2}$ of $\mathrm{Mod}(\varSigma_{0, 2g+2})$ 
in \cite{Hilden}, which is now called the {\it (spherical) Hilden group}. 
The group $SH_{2g+2}$ is isomorphic to the image $\Gamma(SW_{2g+2})$ 
under $\Gamma: SB_{2g+2} \rightarrow \mathrm{Mod}(\varSigma_{0, 2g+2})$  
(Theorem~\ref{thm_HiWi}).  
As an application of the above relation between $\mathcal{H}({\Bbb H}_g)$ and $SW_{2g+2}$, 
we determine a finite presentation of $\mathcal{H}({\Bbb H}_g)$ in Appendix~\ref{section_appendix}, 
see Theorem~\ref{thm:presentaion-HHg}

We are interested in the asymptotic behavior of the minimal dilatations of all pseudo-Anosov elements in 
$\mathcal{H}({\Bbb H}_g)$ varying $g$. 
To state our results, we need some setup. 
Let $\varSigma$ be an orientable, connected surface possibly with punctures. 
A homeomorphism $\Phi: \varSigma \rightarrow \varSigma$ is {\it pseudo-Anosov} 
if there exist a pair of transverse measured foliations 
$(\mathcal{F}^u, \mu^u)$ and 
$(\mathcal{F}^s, \mu^s)$ 
and  a constant $\lambda = \lambda(\Phi)>1$ such that 
$$\Phi(\mathcal{F}^u, \mu^u) = (\mathcal{F}^u, \lambda \mu^u) \  \hspace{2mm}\mbox{and}\ \hspace{2mm}
\Phi(\mathcal{F}^s, \mu^s) = (\mathcal{F}^s, \lambda^{-1} \mu^s).$$
Then 
$\mathcal{F}^u$ and $\mathcal{F}^s$ are called the {\it unstable} and {\it stable foliations},  and 
$\lambda$ is called the {\it dilatation} or {\it stretch factor} of $\Phi$. 
The topological entropy $\mathrm{ent}(\Phi)$ is precisely equal to $\log \lambda(\Phi)$. 
A significant property of pseudo-Anosov homeomorphisms is that 
$\mathrm{ent}(\Phi)$ attains the minimal entropy among all homeomorphisms on $\varSigma$ 
which are isotopic to $\Phi$, 
see \cite[Expos\'{e} 10]{FLP}. 
An element $\phi$ of the mapping class group $\mathrm{Mod}(\varSigma)$ of $\varSigma$ 
is called {\it pseudo-Anosov} if $\phi$ contains a pseudo-Anosov homeomorphism $\Phi: \varSigma \rightarrow \varSigma$ 
as a representative. 
In this case, we let $\lambda(\phi)= \lambda(\Phi)$ and $\mathrm{ent}(\phi) = \mathrm{ent}(\Phi)$, 
and we call them the {\it dilatation} and {\it entropy} of $\phi$ respectively.  
We call 
$$\mathrm{Ent}(\phi)= |\chi(\varSigma)|\mathrm{ent}(\phi)$$ 
the {\it normalized entropy} of $\phi$, 
where $\chi(\varSigma)$ is the Euler characteristic of $\varSigma$.

Let  $f: \varSigma \rightarrow \varSigma$ be a representative of 
a given mapping class $\phi \in \mathrm{Mod}(\varSigma)$. 
The mapping torus ${\Bbb T}_{\phi} = {\Bbb T}_{[f]}$  is defined by 
$${\Bbb T}_{\phi} = \varSigma \times {\Bbb R}/ \sim,$$ 
where  $\sim$ identifies $(x,t+1)$ with $(f(x),t)$ for $x \in \varSigma$ and $t \in {\Bbb R}$. 
We call $\phi$ the {\it monodromy} of ${\Bbb T}_{\phi}$. 
We sometimes call the representative $f \in \phi$  the monodromy of ${\Bbb T}_{\phi}$. 
The {\it suspension flow} $f^t$ on ${\Bbb T}_{\phi}$ is a flow induced by 
the vector field $\tfrac{\partial}{\partial t}$. 
The hyperbolization theorem by Thurston \cite{Thurston3} states that 
when a $3$-manifold $M$ is a surface bundle over the circle, that is 
$M \simeq {\Bbb T}_{\phi}$ for some mapping class $\phi$, 
$M$ admits a hyperbolic structure 
if and only if $\phi$ is pseudo-Anosov.

We fix a surface $\varSigma$, and consider the set of dilatations of all pseudo-Anosov elements on $\varSigma$: 
$$ \mathrm{dil}(\varSigma)= \{\lambda(\phi)\ |\ \phi \in \mathrm{Mod}(\varSigma)\  \mbox{is pseudo-Anosov}\}. 
$$
This is a closed, discrete subset of ${\Bbb R}$, see \cite{Ivanov} for example. 
In particular, given a subgroup $G$ of $\mathrm{Mod}(\varSigma)$ which contains pseudo-Anosov elements,  
there exists a minimum $\delta(G)>1$ among dilatations of all pseudo-Anosov elements in $G$. 
Clearly we have $\delta(G) \ge \delta(\mathrm{Mod}(\varSigma))$. 
Let $\varSigma_{g,n}$ be a closed, orientable surface of genus $g$ removed $n$ punctures. 
We denote by $\delta_g$ and $\delta_{g,n}$, 
the minimal dilatations 
$ \delta(\mathrm{Mod}(\varSigma_g))$ and 
$\delta(\mathrm{Mod}(\varSigma_{g,n}))$ respectively.

By pioneering work of Penner \cite{Penner}, 
the asymptotic equality 
$$\log \delta_g  \asymp \tfrac{1}{g}$$ holds. 
Here 
$A \asymp B$ means that 
there exists a universal constant $c>0$ so that 
$\tfrac{A}{c} < B< c A$. 
In this case, we say that 
$A$ is comparable to $B$. 
Penner proves this claim by using his lower bound 
$\log \delta_{g,n} \ge \tfrac{\log 2}{12g + 4n-12}$ (\cite{Penner}). 
After work of Penner,  
one can ask the following.

\begin{ques}
\label{ques_Penner}
Which sequence of subgroups $G_{(g)}$'s of  $\mathrm{Mod}(\varSigma_{g})$ 
satisfies  $\log \delta(G_{(g)}) \asymp \tfrac{1}{g}$?  
\end{ques}

\noindent
Hironaka also studied Question~\ref{ques_Penner} in \cite{Hironaka2014}. 
To prove  $\log \delta(G_{(g)}) \asymp \tfrac{1}{g}$, 
thanks to the Penner's lower bound 
$\log \delta_g \ge \tfrac{\log 2}{12g-12}$, 
it suffices to construct a sequence of pseudo-Anosov elements $\phi_{(g)} \in G_{(g)}$  for $g \ge 2$ 
whose normalized entropies $\mathrm{Ent}(\phi_{(g)}) = (2g-2) \mathrm{ent}(\phi_{(g)})$ are uniformly bounded from above.

It is a result by Farb-Leininger-Margalit that 
the dilatation of any pseudo-Anosov element in
the Torelli group $\mathcal{I}(\varSigma_g) $ 
has a uniform lower bound (\cite[Theorem~1.1]{FLM}). 
See also  Agol-Leininger-Margalit~\cite{ALM}.  
On the other hand, the two subgroups $\mathcal{H}(\varSigma_g)$ and $\mathrm{Mod}({\Bbb H}_g)$ 
are examples of answers to Question~\ref{ques_Penner}. 
In fact, 
Hironaka-Kin prove in \cite[Theorem~1.1]{HK}, 
$$g \log \delta (\mathcal{H}(\varSigma_g)) <  \log (2+ \sqrt{3})  \approx 1.3169 
\hspace{2mm}\mbox{for}\  g \ge 2.$$ 
Hironaka proves in \cite[Section~3.1]{Hironaka2014}, 
\begin{equation}
\label{equation_Hir_bound}
\limsup_{g \to \infty} g \log \delta(\mathrm{Mod}({\Bbb H}_g)) \le  \log ( 33+ 8 \sqrt{17}) \approx 4.1894. 
\end{equation}
The main result of this paper is to prove that 
$\log \delta(\mathcal{H}({\Bbb H}_g))$  is still comparable to $1/g$.

\begin{thm}
\label{thm_asymp} 
We have 
$\log \delta(\mathcal{H}({\Bbb H}_g)) \asymp \frac{1}{g}$ 
 and $\log \delta(SH_{2n}) \asymp \frac{1}{n}$. 
\end{thm}

\begin{prop}
\label{prop_main}
There exists a sequence of pseudo-Anosov braids $w_{2n} \in SW_{2n}$ $(n \ge 3)$  
such that 
$$\displaystyle\lim_{n \to \infty} n \log (\lambda(w_{2n})) = 2 \log \kappa,$$ 
where 
 $\kappa = \tfrac{1+ \sqrt{5}}{2}+ \tfrac{\sqrt{2+ 2 \sqrt{5}}}{2} \approx 2.89005$ equals the largest root of 
$$t^4-2t^3 -2t^2-2t+1 = (t^2 - (1+ \sqrt{5})t+1) (t^2 - (1- \sqrt{5})t+1).$$ 
\end{prop}
\noindent
The braids $w_{2n}$'s are written  by the standard generators of the spherical braid groups concretely 
(Section~\ref{section_proofs}). 
Theorem~\ref{thm_asymp} follows from Proposition~\ref{prop_main} 
as we explain now. 
We say that  a braid $b \in SB_{2g+2}$ is pseudo-Anosov 
if $\Gamma(b) \in \mathrm{Mod}(\varSigma_{0, 2g+2})$ is a pseudo-Anosov mapping class. 
In this case, the dilatation $\lambda(b)$ is defined by the dilatation $\lambda(\Gamma(b))$ 
of the pseudo-Anosov element $\Gamma(b)$. 
On the other hand, there exists a surjective homomorphism 
$Q: \mathcal{H}({\Bbb H}_g) \rightarrow SH_{2g+2}$ 
with the kernel $\langle \iota \rangle$ (Theorem~\ref{thm_HiKi}). 
If  $\phi \in \mathcal{H}({\Bbb H}_g)$ is pseudo-Anosov,  
then $Q(\phi) \in SH_{2g+2}$ is also pseudo-Anosov.   
If $\Phi: \varSigma_{0, 2g+2} \rightarrow \varSigma_{0, 2g+2}$ is a pseudo-Anosov homeomorphism 
which represents $Q(\phi)$, 
then one can take a pseudo-Anosov homeomoprhism $\tilde{\Phi}: \varSigma_g \rightarrow \varSigma_g$ 
which is a lift of $\Phi$ such that $\phi = [\tilde{\Phi}]$. 
Two pseudo-Anosov homeomorphisms $\Phi$ and $\tilde{\Phi}$ have the same dilatation, 
since their local dynamics are the same. 
Hence we have $\lambda(\phi)= \lambda(Q(\phi))$. 
In particular we have 
$\delta(\mathcal{H}({\Bbb H}_g)) = \delta(SH_{2g+2})$ for $g \ge 2$ 
(Lemma~\ref{lem_NTtype2}). 
Proposition~\ref{prop_main} says that there exists a sequence of pseudo-Anosov elements 
$\Gamma(w_{2n}) \in SH_{2n}$ 
whose normalized entropies $\mathrm{Ent}(\Gamma(w_{2n}))$ are uniformly bounded from above. 
Thus the same thing occurs in  $\mathcal{H}({\Bbb H}_g)$. 
See Section~\ref{subsection_HypHand}.

By Proposition~\ref{prop_main} together with 
$\delta(\mathcal{H}({\Bbb H}_g)) = \delta(SH_{2g+2})$, the following holds. 

\begin{thm}
\label{thm_SH}
We have 
$\displaystyle\limsup_{g \to \infty} g \log \delta(\mathcal{H}({\Bbb H}_g)) \le 2 \log \kappa \approx 2.12255$. 
\end{thm}

\noindent
Since $\mathcal{H}({\Bbb H}_g)$ is the subgroup of $ \mathrm{Mod}({\Bbb H}_g)$, 
we have $\delta( \mathrm{Mod}({\Bbb H}_g)) \le \delta(\mathcal{H}({\Bbb H}_g))$. 
Comparing  Theorem~\ref{thm_SH} with  (\ref{equation_Hir_bound}), 
we find that Theorem~\ref{thm_SH} improves the previous upper bound of $\delta(\mathrm{Mod}({\Bbb H}_g))$ 
by Hironaka.  
In this sense, 
the sequence of pseudo-Anosov elements of $\mathcal{H}({\Bbb H}_g)$ 
used for the proof of Theorem ~\ref{thm_SH} 
is a new example for $ \mathrm{Mod}({\Bbb H}_g)$ 
whose normalized entropies are uniformly bounded from above.

Let us mention a property of the sequence of pseudo-Anosov braids $w_{2n}$'s in Proposition~\ref{prop_main} 
and give an outline of its proof. 
(See Section~\ref{section_proofs} for more details.) 
Let $L_0$ be a link with $3$ components as in Figure~\ref{fig_L10n95}. 
The mapping torus of $\Gamma(w_6) \in \mathrm{Mod}(\varSigma_{0,6})$ 
is homeomorphic to $S^3 \setminus L_0$, that is the complement of $L_0$ in a $3$-sphere $S^3$. 
Thus once we prove that $w_6$ is a pseudo-Anosov braid, 
it follows that $S^3 \setminus L_0$ is a hyperbolic fibered $3$-manifold.
The sequence $w_8, w_{10}, w_{12}, \cdots$ has a property such that 
if $k= 4n+8$, then the mapping torus of $\Gamma(w_{k})$ is homeomorphic to $S^3 \setminus L_0$, 
and if $k= 4n+6$, then the fibration of the mapping torus of $\Gamma(w_{k})$ comes from 
a fibration of $S^3 \setminus L_0$ by Dehn filling cusps along the boundary slopes of a fiber 
(which depends on $k$). 
A technique about {\it disk twists} (see Section~\ref{subsection_DiskTwists}) 
provides a method of constructing 
sequences of mapping classes on punctured spheres 
whose mapping tori are homeomorphic to each other. 
We use this technique for the construction of the sequence 
$w_8, w_{12},  \cdots, w_{4n+8},\cdots$ from the mapping torus of $\Gamma(w_6)$.  
We conclude that the braids 
$w_8, w_{12},  \cdots, w_{4n+8}, \cdots$ are pseudo-Anosov 
from the fact that $S^3 \setminus L_0$ is  hyperbolic. 
We point out that our method by using disk twists quite suit to construct elements in the Hilden groups 
whose mapping tori are homeomorphic to each other.  
Now let $\Phi= \Phi_6: \varSigma_{0,6} \rightarrow \varSigma_{0,6}$ 
be the pseudo-Anosov homeomorphism which represents $\Gamma(w_6)$, 
and let $\tau_6$ and $\mathfrak{p}_6: \tau_6 \rightarrow \tau_6$ 
be the invariant train track and the train track representative for $\Gamma(w_6)$ respectively. 
We find that $\lambda(w_6)$ is equal to the constant $\kappa$ in Proposition~\ref{prop_main}. 
An analysis by using the suspension flow $\Phi^t$ on $S^3 \setminus L_0$ 
and the train track representative $\mathfrak{p}_6: \tau_6 \rightarrow \tau_6$ 
tells us the dynamics of the pseudo-Anosov homeomorphism which represents $\Gamma(w_{4n+8})$ 
for each $n \ge 0$. 
In particular one can construct the train track representative 
$ \mathfrak{p}_{4n+8}: \tau_{4n+8} \rightarrow \tau_{4n+8}$ for $\Gamma(w_{4n+8})$ concretely. 
From the `shape' of the invariant train track $\tau_{4n+8}$, 
we see that $w_{4n+6}$ is a pseudo-Anosov braid with the same dilatation as $w_{4n+8}$. 
By a study of a particular fibered face for the exterior of the link $L_0$, 
we see that the normalized entropy of $\Gamma(w_{4n+8})$ 
converges to the one of $\Gamma(w_6)$, 
which implies that  Proposition~\ref{prop_main} holds.

From view point of fibered faces of fibered $3$-manifolds,  
 the sequence of mapping classes $\Gamma(w_{4n+8})$'s 
are obtained from a certain deformation of 
the monodromy $\Gamma(w_6)$ on  the $\varSigma_{0,6}$-fiber 
of the fibration on $S^3 \setminus L_0$. 
See also Hironaka~\cite{Hironaka2} and Valdivia~\cite{Valdivia} for other constructions 
in which fibered faces on hyperbolic $3$-manifolds are used crucially.

\begin{center}
\begin{figure}
\includegraphics[width=1.5in]{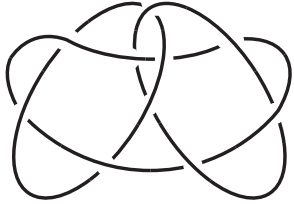}
\caption{Link $L_0$ which gets the name  $L10n95$ in the Thistlethwaite link table, 
see the Knot Atlas \cite{KnotAtlas}.}
\label{fig_L10n95}
\end{figure}
\end{center}

By using Penner's lower bound $\log \delta_{0,n} \ge \tfrac{\log 2}{4n-12}$, 
Hironaka-Kin prove that $\log \delta_{0,n} \asymp \tfrac{1}{n}$ (\cite{HK}). 
In fact, it is shown in \cite{HK} that  
the subgroup $\Gamma(SB_{(n-1)})$ of $\mathrm{Mod}(\varSigma_{0,n})$ 
which consists of  all mapping classes on an $(n-1)$-punctured disk $D_{n-1}$ 
satisfies $\log \delta( \Gamma(SB_{(n-1)})) \asymp \tfrac{1}{n}$. 
(See Section~\ref{subsection_Sbraid} for the definition of $SB_{(n-1)}$.)   
By Theorem~\ref{thm_asymp}, we have another example, namely the Hilden group $SH_{2n}$, 
with the same property, 
that is the asymptotic behavior of the minimal dilatation of $SH_{2n}$ 
is the same as that of the ambient group $\mathrm{Mod}(\varSigma_{0,2n})$. 
On the other hand, it is proved by Song that 
the dilatation of any pseudo-Anosov element of the pure braid groups 
has a uniform lower bound (\cite{Song}). 
We  ask the following. 

\begin{ques}
\label{ques_braid}
Which  sequence of subgroups $G_{(n)}$'s of  $\mathrm{Mod}(\varSigma_{0,n})$ 
satisfies  $\log \delta(G_{(n)}) \asymp \tfrac{1}{n}$?  
\end{ques}

The organization of this paper is as follows. 
In Section~\ref{section_Preliminaries}, 
we review basic facts on the Thurston norm and fibered faces on hyperbolic fibered $3$-manifolds. 
We recall the connection between the spherical braid groups and the mapping class groups 
on punctured spheres. 
Then we recall the definitions of  the Hilden groups and the wicket groups, 
and we describe a connection between them. 
We also introduce the hyperelliptic handlebody groups and give a relation 
between the hyperelliptic handlebody groups and the wicket groups. 
Lastly, we introduce the disk twists. 
In Section~\ref{section_proofs}, 
we prove Proposition~\ref{prop_main}. 
In Apppendix~\ref{section_appendix}, we prove some claims given in 
Sections~\ref{subsection_Hilden} and \ref{subsection_HypHand}, and 
we determine  a finite presentation of  $\mathcal{H}({\Bbb H}_g)$. 

\medskip

\noindent
{\bf Acknowledgements}: 
The authors thank 
Tara E.~Brendle and Masatoshi Sato 
for useful comments.

\section{Preliminaries} 
\label{section_Preliminaries}

\subsection{Mapping class groups} 
Let $\varSigma$ be a compact, connected, orientable surface removed the set of finitely many points $P$  
in its interior. 
The mapping class group $\mathrm{Mod}(\varSigma)$ 
is the group of isotopy classes of homeomorphisms on $\varSigma$ 
which fix both $P$ and the boundary $\partial \varSigma$ as {\it sets}. 
We apply elements of $\mathrm{Mod}(\varSigma)$  from right to left.

\subsection{Thurston norm, fibered faces and entropy functions}

Let $M$ be an oriented hyperbolic $3$-manifold possibly with boundary. 
We recall some properties of the Thurston norm  
$\|\cdot\|: H_2(M, \partial M; {\Bbb R}) \rightarrow {\Bbb R}$. 
For more details, see  \cite{Thurston1} by Thurston. 
See also  \cite[Sections~5.2, 5.3]{Calegari} by Calegari. 

Let $F$ be a finite union of oriented, connected surfaces. 
We define $\chi_-{(F)}$ to be  
$$\chi_-(F)= \sum_{F_i \subset F} \max \{0, -\chi (F_i)\},$$ 
where $F_i$'s are the connected components of $F$. 
The Thurston norm $\|\cdot\|$ is defined for an integral class $a \in H_2(M, \partial M; {\Bbb Z})$ by 
$$\|a\|= \min_{F} \{\chi_-(F)\ | \ a= [F] \},$$ 
where the minimum ranges over all oriented surfaces $F$ embedded in $M$. 
A surface $F$ which realizes the minimum is called a 
{\it minimal representative} or {\it norm-minimizing} of $a$.  
Then $\|\cdot\|$ defined on all integral classes admits a unique continuous extension 
$\|\cdot\|: H_2(M, \partial M; {\Bbb R}) \rightarrow {\Bbb R}$ 
which is linear on rays through the origin. 
A significant property of $\|\cdot\|$ is that 
the unit ball $U_M  $ with respect to $\|\cdot\|$ is a finite-sided polyhedron. 

We take a top dimensional face $\Omega$ on the boundary $\partial U_M$. 
Let $C_{\Omega}$ be the cone over $\Omega$ with the origin, and 
let $int(C_{\Omega})$ be its open cone, that is 
the interior of $C_{\Omega}$. 
When $M$ is a hyperbolic fibered $3$-manifold, 
the Thurston norm provides deep information about fibrations on $M$.

\begin{thm}[Thurston~\cite{Thurston1}]
\label{thm_Thurston}
Suppose that $M$ fibers over the circle $S^1$ with fiber $F$. 
Then there exists a top dimensional face $\Omega$ on  $\partial U_M$ so that 
$[F] \in int(C_\Omega)$. 
Moreover given any integral class $a \in int(C_{\Omega})$, 
its minimal representative $F_a$  becomes a fiber of a fibration on $M$. 
\end{thm}

\noindent
Such a face $\Omega$ and such an open cone $int(C_{\Omega})$ are 
called a {\it fibered face} and {\it fibered cone} respectively, 
and an integral class $a \in int(C_{\Omega})$ 
is called a {\it fibered class}. 

Now we take any primitive fibered class $a \in int(C_{\Omega})$. 
The minimal representative $F_a$ is a connected fiber of the fibration associated to $a$. 
If we let $\Phi_a: F_a \rightarrow F_a$ be the monodromy of this fibration, 
then the mapping class $\phi_a= [\Phi_a]$  is necessarily pseudo-Anosov, 
since $M$ is hyperbolic. 
One can define the dilatation $\lambda(a)$ and entropy $\mathrm{ent}(a)$ 
to be the dilatation and entropy of pseudo-Anosov $\phi_a$. 
The entropy function defined on primitive fibered classes $a$'s  can be extended to 
the entropy function on rational classes 
by homogeneity. 
An important property of such entropies, studied by Fried, Matsumoto and McMullen 
is that the function $a \mapsto \mathrm{ent}(a)$ defined for rational classes $a \in int(C_{\Omega})$ 
extends to a real analytic convex function on the fibered cone $int(C_{\Omega})$, 
see \cite{McMullen2} for example. 
Moreover 
 the {\it normalized entropy function} 
$$\mathrm{Ent}= \|\cdot\| \mathrm{ent}: int(C_{\Omega}) \rightarrow {\Bbb R}$$ 
is constant on each ray in $int(C_{\Omega})$ through the origin.

Since $M$ fibers over $S^1$ with fiber $F$, 
$M$ is homeomorphic to a mapping torus  ${\Bbb T}_{[\Phi]}$, 
where $\Phi: F \rightarrow F$ is the monodromy of the fibration associated to $[F] \in int(C_{\Omega})$. 
We may assume that $\Phi: F \rightarrow F$ is a pseudo-Anosov homeomorphism 
with the stable and unstable foliations $\mathcal{F}^s$ and $\mathcal{F}^u$.  
A surface $F'$ is called a {\it cross-section} to the suspension flow $\Phi^t$ on $M$ 
if $F'$ is transverse to $\Phi^t$ and $F'$ intersects every flow line. 

Let $J_1$ and $J_2$ be embedded arcs in $M$ 
which are transverse to $\Phi^t$. 
We say that $J_1$ {\it is connected to} $J_2$ if there exists a positive continuous function 
$\mathfrak{g}: J_1 \rightarrow {\Bbb R}$ which satisfies the following. 
For any $x \in J_1$, we have 
$\Phi^{\mathfrak{g}(x)} \in J_2$ and 
$\Phi^{t}(x) \notin J_2$ for $0 < t < \mathfrak{g}(x)$. 
Moreover the map $: J_1 \rightarrow J_2$ given by $x \mapsto \Phi^{\mathfrak{g}(x)}(x)$ is a homeomorphism. 
In this case, we let 
$$[J_1, J_2]= \{\Phi^{t}(x)\ |\ x \in J_1, 0 \le t \le \mathfrak{g}(x)\},$$ 
and we call $[J_1, J_2]$ a {\it flowband}. 
We use flowbands in the proof of Proposition~\ref{prop_main}.

\begin{thm}[Fried~\cite{Fried} for (1)(2), Thurston~\cite{Thurston1} for (3)] 
\label{thm_Fried}
Let $\Phi: F \rightarrow F$, $M \simeq {\Bbb T}_{[\Phi]}$ and $\Omega$ be as above. 
Let $\widehat{\mathcal{F}^s}$ and $\widehat{\mathcal{F}^u}$ 
denote the suspensions of $\mathcal{F}^s$ and $\mathcal{F}^u$ by $\Phi$ 
in $M \simeq {\Bbb T}_{[\Phi]}$. 
For any minimal representative $F_a$ of any fibered class $a \in int(C_{\Omega})$, 
we can modify $F_a$ by an isotopy which satisfies the following. 
\begin{enumerate}
\item[(1)] 
$F_a$ is transverse to $\Phi^t$, 
and the first return map $:F_a \rightarrow F_a$ is precisely the pseudo-Anosov monodromy $\Phi_a: F_a \rightarrow F_a$ 
of the fibration on $M$ associated to $a$. 
Moreover $F_a$ is unique up to isotopy along flow lines. 

\item[(2)] 
The stable and unstable foliations of the pseudo-Anosov homeomorphism  $\Phi_a$ are given by 
$\widehat{\mathcal{F}^s} \cap F_a$ and $\widehat{\mathcal{F}^u} \cap F_a$ respectively.  

\item[(3)] 
If $a' \in H_2(M, \partial M; {\Bbb R})$ is represented by some cross-section  to $\Phi^t$, 
then $a' \in int(C_{\Omega})$. 
\end{enumerate}
\end{thm}

\subsection{Spherical braid groups}
\label{subsection_Sbraid}

Let $SB_m$ be the spherical braid group with $m$ strings. 
We depict braids vertically in this paper. 
We define the product of braids as follows. 
Given $b, b' \in SB_m$, 
we stuck $b$ on $b'$, and concatenate the bottom $i$th endpoint of $b$ 
with the top $i$th endpoint of $b'$ for each $1 \le i \le m$. 
Then we get $m$ strings, and 
the product $bb' \in SB_m$ is the resulting braid (after rescaling such $m$ strings), see Figure~\ref{fig_product}. 
We often label the numbers $1, \cdots, m$ (from left to right)  at the bottom of a given braid. 
Let $\sigma_i$ denote a braid of $SB_m$ 
obtained by crossing the $i$th string under the $(i+1)$st string, 
see Figure~\ref{fig_halftwist}(1).  
(Here the $i$th string means  the string labeled $i$ at the bottom.) 
It is well-known that  $SB_m$ is a group generated by 
$\sigma_1, \sigma_2, \cdots, \sigma_{m-1}$, 
and its relations are given by 
\begin{enumerate}
\item[(1)] 
$\sigma_i \sigma_j = \sigma_j \sigma_i$  \hspace{2mm}if $|i - j| \ge 2$, 

\item [(2)] 
$ \sigma_i \sigma_{i+1} \sigma_i = \sigma_{i+1} \sigma_i \sigma_{i+1}$ \hspace{2mm} for $i = 1, \cdots, m-2$, 

\item[(3)] 
$ \sigma_1 \sigma_2 \cdots \cdots \sigma_{m-2} \sigma_{m-1}^2 \sigma_{m-2} \cdots \sigma_2 \sigma_1 = 1.$
\end{enumerate}

\begin{center}
\begin{figure}
\includegraphics[width=3in]{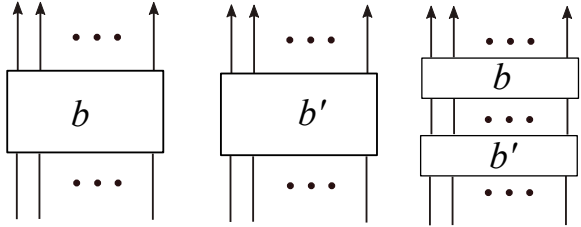}
\caption{Braids $b$, $b'$ and $bb'$.} 
\label{fig_product}
\end{figure}
\end{center}

We recall a connection between $SB_m$ and $\mathrm{Mod}(\varSigma_{0,m})$. 
Let $c_1, \cdots, c_m$ be the punctures of $\varSigma_{0,m}$. 
Let $h_i$ be the left-handed half twist 
about the arc between the $i$th and $(i+1)$st punctures $c_i$ and $c_{i+1}$, 
see Figure~\ref{fig_halftwist}(2). 
We define a homomorphism 
$$\Gamma: SB_m \rightarrow \mathrm{Mod}(\varSigma_{0,m})$$
which sends $\sigma_i$ to $h_i$ 
for $ i \in \{1, \cdots, m-1\}$. 
Since $\mathrm{Mod}(\varSigma_{0,m})$ is generated by $h_1, \cdots, h_{m-1}$, 
$\Gamma$ is surjective. 
If we let 
$$\Delta= \Delta_m= (\sigma_1 \sigma_2 \cdots \sigma_{m-1}) 
(\sigma_1 \sigma_2 \cdots \sigma_{m-2}) \cdots (\sigma_1 \sigma_2) \sigma_1$$
which is a half twist braid, then 
the kernel of $\Gamma$ is isomorphic to ${\Bbb Z}/{2{\Bbb Z}}$ 
which is generated by a full twist braid $\Delta^2$. 
Thus 
$$SB_m / \langle \Delta^2 \rangle \simeq \mathrm{Mod}(\varSigma_{0,m}).$$
Given a braid $b \in SB_m$, 
the mapping torus ${\Bbb T}_{\Gamma(b)}$ of $\Gamma(b)$ 
is denoted by  ${\Bbb T}_{b}$ for simplicity.

\begin{rem}
\label{rem_braid_convention} 
Each $m$-braid as in Figure~\ref{fig_halftwist}(1) with the orientation 
from the bottom of strings to the top 
induces the motion of $m$ points on the sphere. 
This gives rise to the above  homomorphism $\Gamma$, 
which maps $\sigma_i$ to $h_i$. 
In this paper, we denote the braid in Figure~\ref{fig_halftwist}(1)  by $\sigma_i$. 
\end{rem}

We say that a braid $b \in SB_{m}$ is {\it pseudo-Anosov} 
if $\Gamma(b)$ is a pseudo-Anosov mapping class. 
In this case, we define the dilatation $\lambda(b)$ of $b$ to be the dilatation 
$\lambda(\Gamma(b))$. 
Also, we let $\Phi_{b}: \varSigma_{0,m} \rightarrow \varSigma_{0,m}$ be 
the pseudo-Anosov homeomorphism which represents $\Gamma(b)$, 
and let $\mathcal{F}_b$ be the unstable foliation for $\Phi_b$.

Let $SB_{(m-1)}$ be the subgroup of $SB_m$ which is generated by $\sigma_1, \cdots, \sigma_{m-2}$. 
(Hence a braid $b \in SB_{(m-1)}$ is represented by a word without $ \sigma_{m-1}^{\pm 1}$.)
As we will see in Section~\ref{subsection_relation}, 
 $SB_{(m-1)}$  is closely related to the $(m-1)$-braid group $B_{m-1}$.

\begin{center}
\begin{figure}
\includegraphics[width=3.3in]{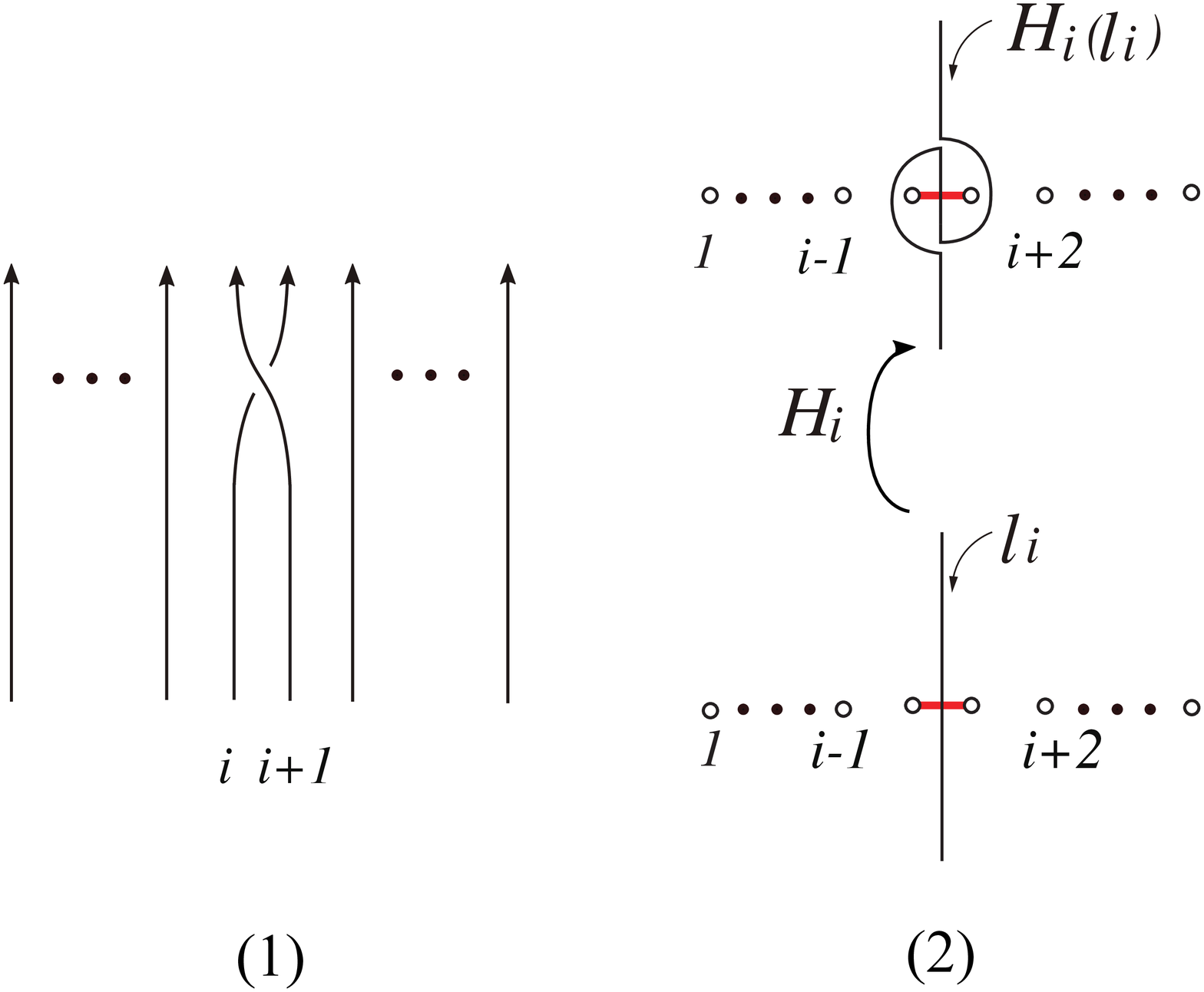}
\caption{(1) $\sigma_i $. 
(2) Action of a representative $H_i \in h_i $ on $\ell_i$, 
where $\ell_i$ is a vertical arc which passes through the horizontal arc between the punctures $c_i$ and $c_{i+1}$, 
see Remark~\ref{rem_convention}.} 
\label{fig_halftwist}
\end{figure}
\end{center}

\begin{center}
\begin{figure}
\includegraphics[width=4in]{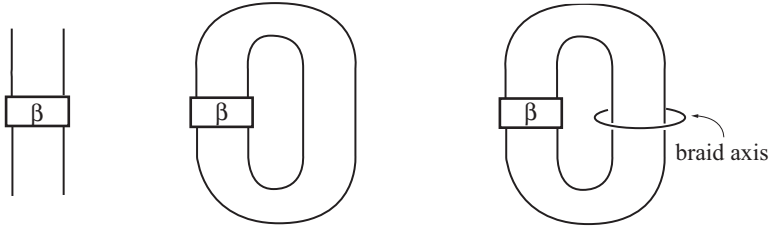}
\caption{Braid $\beta$, closure $\mathrm{cl}(\beta)$, and 
braided link $\mathrm{br}(\beta)$ from left to right.}  
\label{fig_braid_axis}
\end{figure}
\end{center}

\subsection{Braid groups} 
\label{subsection_relation}

We recall a connection between the two groups $(m-1)$-braid group $B_{m-1}$ 
on a disk and the mapping class group 
$\mathrm{Mod}(D_{m-1})$, 
where $D_{m-1}$ is a disk with $m-1$ punctures $c_1, \cdots, c_{m-1}$. 
By abusing notations, 
we denote by $\sigma_i$,  the braid of $B_{m-1}$ 
obtained by crossing the $i$th string under the $(i+1)$st string. 
The braid group $B_{m-1}$ with $m-1$ strings is the group 
generated by $\sigma_1, \cdots, \sigma_{m-2}$ 
having the following relations. 
\begin{enumerate}
\item[(1)] 
$\sigma_i \sigma_j = \sigma_j \sigma_i$  \hspace{2mm}if $|i - j| \ge 2$, 

\item [(2)] 
$ \sigma_i \sigma_{i+1} \sigma_i = \sigma_{i+1} \sigma_i \sigma_{i+1}$ \hspace{2mm} for $i = 1, \cdots, m-3$. 
\end{enumerate}
Abusing notations again, we denote by $h_i$,
 the left-handed half twist about the arc between the $i$th and  $(i+1)$st punctures of $D_{m-1}$. 
We also use  $\Gamma$ for the surjective homomorphism 
$$\Gamma: B_{m-1} \rightarrow \mathrm{Mod}(D_{m-1})$$
which sends $\sigma_i$ to $h_i$ for $i \in \{1, \cdots, m-2\}$. 
In this case, the kernel of $\Gamma$ 
is an infinite cyclic group generated by the full twist braid $\Delta^2 = \Delta_{m-1}^2$.

We have a homomorphism 
\begin{eqnarray*}
c: \mathrm{Mod}(D_{m-1}) &\rightarrow& \mathrm{Mod} (\varSigma_{0, m})
\\
h_i &\mapsto& h_i
\end{eqnarray*}
which is induced by the map 
that sends the boundary of the disk to 
the $m$th puncture  of $\varSigma_{0, m}$. 
Observe that 
$$c(\Gamma(B_{m-1})) = c(\mathrm{Mod}(D_{m-1}))= \Gamma(SB_{(m-1)}).$$
Given $\beta \in B_{m-1}$, 
we denote the mapping torus 
${\Bbb T}_{c( \Gamma(\beta))}$ of $c( \Gamma(\beta))$ 
by  ${\Bbb T}_{\beta}$ for simplicity. 
Let $\mathrm{cl}(\beta)$ be the closure of $\beta$ (or the closed braid of $\beta$). 
We have ${\Bbb T}_{\beta}  \simeq S^3 \setminus \mathrm{br}(\beta)$, 
(that is ${\Bbb T}_{\beta}$ is homeomorphic to $S^3 \setminus \mathrm{br}(\beta)$)  
where $\mathrm{br}(\beta)$ is the {\it braided link} of $\beta$ 
which is a union of $\mathrm{cl}(\beta)$ and  its {\it braid axis}, 
see Figure~\ref{fig_braid_axis}. 

\begin{rem}
\label{rem_convention}
Recall  that we apply elements of the mapping class groups from right to left. 
This convention together with the homomorphism $\Gamma$ from $B_{m-1}$ to 
$\mathrm{Mod}(D_{m-1})$ 
gives rise to an orientation of strings of $\beta$ from the bottom to the top, 
which is compatible with the direction of the suspension flow on 
${\Bbb T}_{\beta}= S^3 \setminus \mathrm{br}(\beta)$. 
\end{rem}

We say that $\beta \in B_{m-1}$ is {\it pseudo-Anosov} 
if $c(\Gamma (\beta))$ is pseudo-Anosov.  
In this case, 
we define the dilatation $\lambda(\beta)$ of $\beta$ to be the dilatation 
$\lambda(c (\Gamma(\beta)))$.

By definition, an $m$-braid $b \in SB_{(m-1)}$ 
is represented by a word without $\sigma_{m-1}^{\pm 1}$. 
Removing the last string of $b$, 
we get an $(m-1)$-braid on a sphere. 
If we regard such a braid as the one on a disk, 
we have an $(m-1)$-braid $\underline{b} $ 
with the same word as $b$. 
By definition of $\underline{b}$, we have 
$$c(\Gamma(\underline{b}))= \Gamma(b).$$ 
Since 
${\Bbb T}_{\underline b} = {\Bbb T}_{c(\Gamma(\underline{b}))} = {\Bbb T}_{\Gamma(b)} = {\Bbb T}_{b}$, 
we have  $ {\Bbb T}_{\underline{b}} = {\Bbb T}_b $. 
We  get the following lemma immediately.

\begin{lem}
\label{lem_bar}
A braid $b \in SB_{(m-1)}$ is pseudo-Anosov if and only if 
$\underline{b} \in B_{m-1}$ is pseudo-Anosov. 
In this case, the equality $\lambda(b) = \lambda(\underline{b})$ holds, and 
${\Bbb T}_b (= {\Bbb T}_{\underline{b}})$ is a hyperbolic fibered $3$-manifold. 
\end{lem}

\subsection{Hilden groups and wicket groups}
\label{subsection_Hilden}

\subsubsection{Relations between Hilden groups and wicket groups} 
\label{subsubsection_relation}

First of all, we define a subgroup of $\mathrm{Mod}(\varSigma_{0,2n})$ 
which was introduced by Hilden~\cite{Hilden}. 
Let $A_1, \cdots, A_{n}$ be $n$ disjoint trivial arcs 
properly embedded in a unit ball $D^3$ 
as in Figure~\ref{fig_action}(1). 
More precisely, each $A_i$ is unknotted and the union 
${\bf A} = {\bf A}_n= A_1 \cup \cdots \cup A_{n}$ is unlinked. 
Such  $A_i$'s are called  {\it wickets}. 
Let $\mathrm{Homeo}_+(D^3, {\bf A})$ be 
the set of orientation preserving homeomorphisms on $D^3$ 
preserving ${\bf A}$ setwise. 
For each $\Psi \in \mathrm{Homeo}_+(D^3, {\bf A})$, 
we have  the restriction 
$$\Psi|_{\partial D^3}: 
(\partial D^3, \partial {\bf A}) \rightarrow 
(\partial D^3, \partial {\bf A})$$
which is an orientation preserving homeomorphism on a $2$-sphere $S^2 = \partial D^3$ 
preserving $2n$ points of $ \partial {\bf A}$ setwise. 
Its isotopy class $[\Psi|_{\partial D^3}]$ gives rise to an element of $\mathrm{Mod}(\varSigma_{0,2n})$. 
We define a homomorphism 
$$\mathrm{Mod}(D^3, {\bf A}) \rightarrow \mathrm{Mod}(\varSigma_{0,2n})$$ 
which sends a mapping class $[\Psi]$ of  $\Psi \in \mathrm{Homeo}_+(D^3, {\bf A})$ 
to the mapping class $[\Psi|_{\partial D^3}]$. 
This homomorphism is injective,  
see for example \cite[p.484]{Brendle-Hatcher} or 
\cite[p.157]{Hilden}. 
We prove this claim  in Appendix~\ref{section_appendix} for the convenience of readers, 
see Proposition~\ref{prop:uniqueness}.

The group $\mathrm{Mod}(D^3, {\bf A})$ or its homomorphic image 
into $ \mathrm{Mod}(\varSigma_{0,2n})$ is called the  {\it (spherical) Hilden group} $SH_{2n}$. 
Let us describe $SH_{2n}$ by using certain subgroup of the spherical braid group 
$SB_{2n}$ of $2n$ strings. 
Given a braid $b \in SB_{2n}$, 
we stuck $b$ on ${\bf A} = A_1 \cup \cdots \cup A_{n}$, 
and concatenate the bottom endpoints of $b$ with the endpoints of ${\bf A}$, see Figure~\ref{fig_action}(2). 
Then 
we obtain $n$ disjoint smooth arcs  $^{b}{\bf A}$ 
properly embedded in $D^3$. 
We may suppose that the arcs $^{b}{\bf A}$ have the same endpoints as ${\bf A}$. 
The {\it (spherical) wicket group} $SW_{2n}$ is the subgroup of $SB_{2n}$ 
generated by braids $b$'s such that 
$^{b}{\bf A}$ is isotopic to ${\bf A}$ relative to $\partial {\bf A}$. 
For example, the following braids are elements of $SW_{2n}$. 
\begin{eqnarray*}
r_i &=& \sigma_{2i} \sigma_{2i+1}\sigma_{2i-1}^{-1} \sigma_{2i}^{-1}\ \hspace{3mm}  (i \in \{ 1, \cdots, n-1\}), 
\\
s_i&=& \sigma_{2i}^{-1} \sigma_{2i+1}^{-1} \sigma_{2i-1}^{-1}  \sigma_{2i}^{-1}\  \hspace{2mm}(i \in \{ 1, \cdots, n-1\}), 
\\
t_j &=& \sigma_{2j-1}^{-1} \ \hspace{23mm}(j \in  \{1, \cdots, n\}), 
\end{eqnarray*}
see Figure~\ref{fig_generator_w}(1)(2)(3). 
Now we recall the homomorphism  $\Gamma: SB_{2n} \rightarrow \mathrm{Mod}(\varSigma_{0,2n})$. 
We claim that $\Gamma(r_i)$, $\Gamma(s_i)$ and $\Gamma(t_j)$ are elements of $SH_{2n}$. 
Indeed,  $\Gamma(r_i)$ (resp. $\Gamma(s_i)$) 
 interchanges the $i$th and $(i+1)$st wickets $A_i$, $A_{i+1}$ 
by passing  $A_i$ through (resp. around) $A_{i+1}$. 
$\Gamma(t_j)$ rotates the $j$th wicket $A_j$
$180$ degrees around its vertical axis of the symmetry, 
see Figure~\ref{fig_generator_w}(4)(5)(6).

\begin{thm}
\label{thm_HiWi}
The Hilden group $SH_{2n}$ is the image of the homomorphism 
$\Gamma|_{SW_{2n}}: SW_{2n} \rightarrow \mathrm{Mod}(\varSigma_{0,2n})$ 
whose kernel is equal to $\langle \Delta^2 \rangle$. 
In particular, 
$$SW_{2n}/\langle \Delta^2 \rangle  \simeq SH_{2n}.$$
\end{thm}

\noindent
We shall prove Theorem~\ref{thm_HiWi} in Appendix~\ref{section_appendix} 
by using a finite generating set of $SW_{2n}$ (resp. $SH_{2n}$) given by 
Brendle-Hatcher~\cite{Brendle-Hatcher} (resp. Hilden~\cite{Hilden}).  
We note that 
the definition of the spherical wicket groups in \cite{Brendle-Hatcher} 
is different from the one in this paper. 
We shall claim in Appendix~\ref{section_appendix} that these two definitions give rise to the same group, 
see Proposition~\ref{prop:SWandPi1}. 

The wicket groups are closely related to the {\it loop braid groups} 
which arise naturally in the different fields of mathematics. 
For more details of loop braid groups, see Damiani~\cite{Damiani}. 

For a finite presentation of the {\it Hilden group} on a {\it plane}, 
see Tawn~\cite{Tawn}.

\begin{center}
\begin{figure}
\includegraphics[width=5.2in]{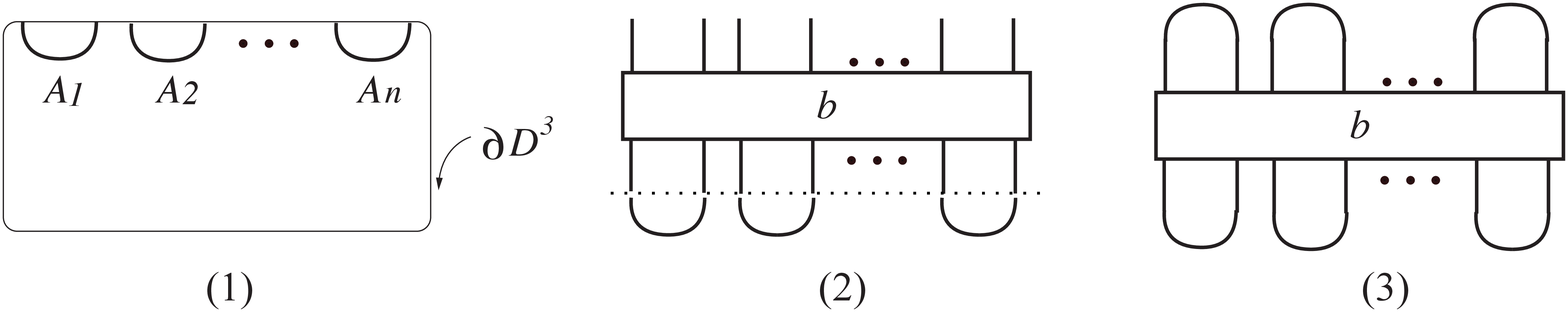}
\caption{(1)  ${\bf A} = A_1 \cup \cdots \cup A_n$. (2) $^{b}{\bf A}$. 
(3) $\mathrm{pl}(b)$.}
\label{fig_action}
\end{figure}
\end{center}

\begin{center}
\begin{figure}
\includegraphics[width=4.2in]{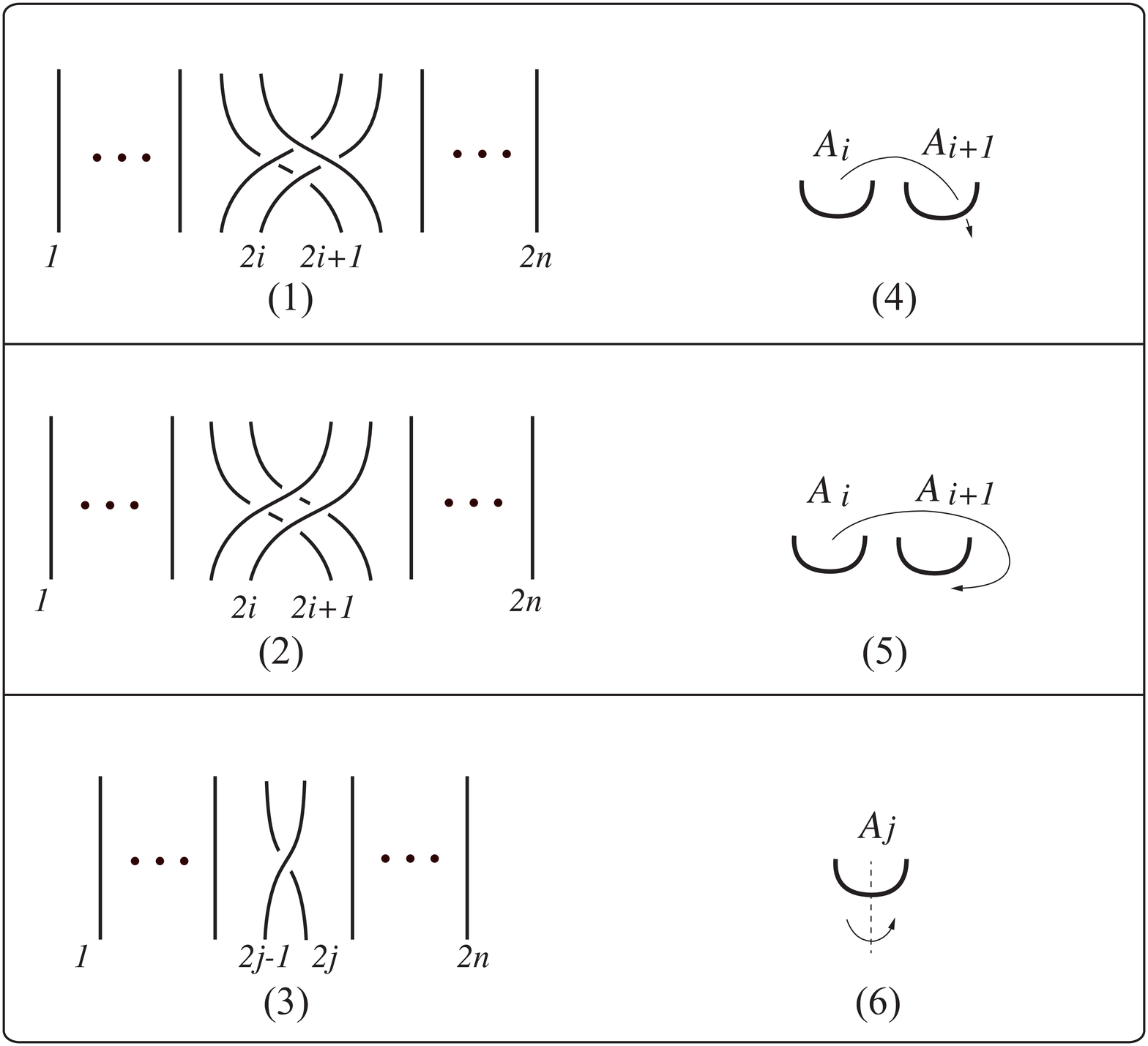}
\caption{(1) $r_i$,  (2) $s_i $ and (3) $t_j\in SW_{2n}$. 
(4) $\Gamma(r_i)$,  (5) $\Gamma(s_i)$ and 
(6) $\Gamma(t_j) \in SH_{2n}$. 
See also Remark~\ref{rem_braid_convention}. 
(cf. \cite[Figure~2]{Brendle-Hatcher}.)} 
\label{fig_generator_w}
\end{figure}
\end{center}

\subsubsection{Plat closures of braids} 
\label{subsubsection_Plat}

In this section, we prove that 
$SW_{2n}$ is of infinite index in $SB_{2n}$ for $n \ge 2$.  
(We do not use this claim in the rest of the paper.) 
To do this, we turn to the plat closures of braids which were introduced by Birman. 
Given $b \in SB_{2n}$, 
the {\it plat closure} of $b$, denoted by $\mathrm{pl}(b)$, is a link in $S^3$ obtained from $b$ 
putting trivial $n$ arcs on $n$ pairs of consecutive, bottom (resp. top) $2n$ endpoints of $b$, 
see Figure~\ref{fig_action}(3). 
Observe that given two braids $w,w' \in SW_{2n}$, 
the plat closures $\mathrm{pl}(b)$ and $\mathrm{pl}(wbw')$ represent the same link. 
Moreover the plat closure of any element $w \in SW_{2n}$, $\mathrm{pl}(w)$,  is a disjoint union of $n$ unknots. 
Every link in $S^3$ can be represented by the plat closure of some braid with even strings 
\cite[Theorem~5.1]{Birman}. 
Birman characterizes two braids with the same strings  
whose plat closures yield the same link 
\cite[Theorem~5.3]{Birman}. 
Fore more discussion on plat closures of braids, see \cite[Chapter~5]{Birman}.

\begin{lem}
\label{lem_SWindex}
$SW_{2n}$ is of infinite index in $SB_{2n}$ for $n \ge 2$.  
\end{lem}

\begin{proof}
We take a braid $b= \sigma_2 \sigma_2 \not\in SW_{2n}$. 
Given $w, w' \in SW_{2n}$, we have 
$\mathrm{pl}(b^k) = \mathrm{pl}(b^k w')  = \mathrm{pl}(w b^k) $ for each integer $k$, and 
the link $\mathrm{pl}(b^k)$ contains the $(2, 2k)$ torus link (as components) 
which is not a disjoint union of  unknots 
for each $k \ne 0$. 
In particular both  $b^k w', w b^k \notin SW_{2n}$. 
This implies that $SW_{2n}$ is of infinite index in $SB_{2n}$ for $n \ge 2$.  
\end{proof}

\noindent
By Lemma~\ref{lem_SWindex}, 
the Hilden group $SH_{2n}$ is of infinite index in $\mathrm{Mod}(\varSigma_{0,2n})$ for $n \ge 2$, 
since $SW_{2n}/{\langle \Delta^2 \rangle} \simeq SH_{2n}$ and 
$SB_{2n}/ {\langle \Delta^2 \rangle} \simeq \mathrm{Mod}(\varSigma_{0,2n})$.

\subsection{Hyperelliptic  handlebody groups} 
\label{subsection_HypHand}

Let ${\Bbb H}_g$ be a handlebody of genus $g$, 
i.e, ${\Bbb H}_g$ is an oriented $3$-manifold 
obtained from a 3-ball attaching $g$ copies of a $1$-handle. 
We take an involution 
$\mathcal{S}: {\Bbb H}_g \rightarrow {\Bbb H}_g$ 
whose quotient space 
 ${\Bbb H}_g/\mathcal{S}$ is a $3$-ball $D^3$ 
 with a union of wickets ${\bf A}= A_1 \cup \cdots  \cup A_{g+1}$ 
 as the image of the fixed point sets of $\mathcal{S}$ under the quotient,  
see Figure~\ref{fig_handle}. 
We call $\mathcal{S}$ the {\it hyperelliptic involution} on ${\Bbb H}_g$. 
The restriction $\mathcal{S}|_{\partial {\Bbb H}_g}: \partial {\Bbb H}_g \rightarrow \partial {\Bbb H}_g$  
defines an involution on $\partial {\Bbb H}_g \simeq \varSigma_g$. 
For simplicity, we denote such an involution $\mathcal{S}|_{\partial {\Bbb H}_g}$ 
by the same notation $\mathcal{S}$, and 
also call it the hyperelliptic involution on $\partial {\Bbb H}_g$. 
The quotient space $\partial {\Bbb H}_g/\mathcal{S}$ is a $2$-sphere 
with $2g+2$ marked points 
that are the image of the fixed points set of 
$\mathcal{S}: \partial {\Bbb H}_g \rightarrow \partial {\Bbb H}_g$ under the quotient.

Let $\mathcal{H}(\varSigma_g)$ be the subgroup of $\mathrm{Mod}(\varSigma_g)$ 
consisting of isotopy classes of orientation preserving homeomorphisms 
on $\varSigma_g$ that commute with $\mathcal{S}:  \partial {\Bbb H}_g \rightarrow \partial {\Bbb H}_g$. 
Such a group $\mathcal{H}(\varSigma_g)$ is called 
the {\it hyperelliptic mapping class group} or {\it symmetric mapping class group}. 
Note that $\mathrm{Mod}(\varSigma_2) = \mathcal{H}(\varSigma_2)$. 
If $g \ge 3$, then 
$\mathcal{H}(\varSigma_g)$ is of infinite index in $\mathrm{Mod}(\varSigma_g)$. 
By the fundamental result by Birman-Hilden~\cite{Birman-Hilden}, 
one has a handy description of $\mathcal{H}(\varSigma_g)$ 
via braids, as we explain now. 
Note that any homeomorphism on $\partial {\Bbb H}_g$ 
that commute with $\mathcal{S}$ fixes the fixed points set of 
$\mathcal{S}:  \partial {\Bbb H}_g \rightarrow \partial {\Bbb H}_g$ 
as a set. 
Hence via the quotient of $\partial {\Bbb H}_g$ by $\mathcal{S}$, 
such a  homeomoprhism on $\partial {\Bbb H}_g$ descends to a homeomorphism 
on a sphere $\partial {\Bbb H}_g/\mathcal{S}$ 
which preserves the $2g+2$ marked points of $\partial {\Bbb H}_g/\mathcal{S}$. 
Thus we have a map 
$$q: \mathcal{H}(\varSigma_g) \rightarrow \mathrm{Mod}(\varSigma_{0,2g+2})$$
by using a representative 
of each mapping class of $\mathcal{H}(\varSigma_g)$ 
which commutes with $\mathcal{S}$. 
Let $\iota \in \mathcal{H}(\varSigma_g)$ 
denote a mapping class of $\mathcal{S}:  \partial {\Bbb H}_g \rightarrow \partial {\Bbb H}_g$ 
which is of order $2$.

\begin{thm}[Birman-Hilden] 
\label{thm_BH}
For $g \ge 2$, 
the map $q: \mathcal{H}(\varSigma_g) \rightarrow \mathrm{Mod}(\varSigma_{0,2g+2})$ 
is well-defined, and it is a surjective homomorphism with the kernel $\langle \iota \rangle$. 
In particular, 
$$\mathcal{H}(\varSigma_g) / \langle \iota \rangle \simeq \mathrm{Mod}(\varSigma_{0,2g+2}) 
\simeq SB_{2g+2}/ \langle \Delta^2 \rangle.$$ 
\end{thm}

Thurston's classification theorem of surface homeomorphisms  states that 
every mapping class $\phi \in \mathrm{Mod}(\varSigma)$ is  one of the  three types: 
periodic, reducible, pseudo-Anosov (\cite{Thurston2}). 
The following well-known lemma says that 
$q$ preserves these types.

\begin{lem}
\label{lem_NTtype}
If $ \phi \in \mathcal{H}(\varSigma_g)$ is pseudo-Anosov 
(resp. periodic, reducible), then 
so is $q(\phi) \in \mathrm{Mod}(\varSigma_{0,2g+2})$, 
i.e, $q(\phi)$ is pseudo-Anosov 
(resp. periodic, reducible). 
When $\phi  \in \mathcal{H}(\varSigma_g)$ is pseudo-Anosov, 
the equality 
$\lambda(\phi ) = \lambda(q(\phi))$ holds. 
\end{lem}

\begin{proof}
It is not hard to see that if $\phi$ is periodic (resp. reducible), 
then $q(\phi)$ is periodic (resp. reducible). 
Suppose that $\phi \in \mathcal{H}(\varSigma_g)$ is pseudo-Anosov. 
Then we see that $q(\phi)$ is pseudo-Anosov. 
If not, then it is periodic or reducible. 
Assume that $q(\phi)$ is periodic. 
(The proof in the reducible case is similar.) 
We take a periodic homeomorphism 
$f: \varSigma_{0, 2g+2} \rightarrow \varSigma_{0, 2g+2}$ 
which represents $q(\phi)$. 
Consider a lift $\tilde{f} : \varSigma_g \rightarrow \varSigma_g$ of $f$. 
Then $\tilde{f}$ is a periodic homeomorphism which represents $\phi$. 
Thus $\phi= [\tilde{f}]$ is a periodic mapping class, which contradicts the assumption that 
$\phi$ is  pseudo-Anosov.

We consider a pseudo-Anosov homeomorphism $\Phi: \varSigma_{0, 2g+2} \rightarrow \varSigma_{0, 2g+2}$  
which represents the pseudo-Anosov mapping class $q(\phi)$. 
 Take a lift $\tilde{\Phi}$ of $\Phi$ which represents $\phi \in \mathcal{H}(\varSigma_g)$. 
 Then $\tilde{\Phi}$ is a pseudo-Anosov homeomorphism whose stable/unstable foliations 
 are lifts of the stable/unstable foliations of $\Phi$. 
 In particular, we have $\lambda(\tilde{\Phi})= \lambda(\Phi)$, 
 since $\Phi$ and $\tilde{\Phi}$ have the same dynamics locally. 
\end{proof}

\noindent
By Theorem~\ref{thm_BH} and Lemma~\ref{lem_NTtype}, 
we  have the following. 

\begin{cor}
We have 
$\delta(\mathcal{H}(\varSigma_g)) = \delta_{0, 2g+2}$ for $g \ge 2$. 
\end{cor}

Let $\mathrm{Mod}({\Bbb H}_g)$ be the group of 
 isotopy classes of orientation preserving homeomorphisms on ${\Bbb H}_g$. 
 We call $ \mathrm{Mod}({\Bbb H}_g)$ the {\it handlebody group}. 
We denote by $\mathrm{SHomeo}_+({\Bbb H}_g)$, 
the group of orientation preserving homeomorphisms on ${\Bbb H}_g$ 
which commute with $\mathcal{S}: {\Bbb H}_g \rightarrow {\Bbb H}_g$. 
Let $\mathcal{H}({\Bbb H}_g)$ be the subgroup of $\mathrm{Mod}({\Bbb H}_g)$ 
consisting of isotopy classes of elements in $\mathrm{SHomeo}_+({\Bbb H}_g)$. 
We call $\mathcal{H}({\Bbb H}_g)$ the {\it hyperelliptic handlebody group}.  
Abusing the notation, we also denote by $\iota \in \mathcal{H}({\Bbb H}_g)$, 
the mapping class of $\mathcal{S}: {\Bbb H}_g \rightarrow {\Bbb H}_g$. 
One can define a homomorphism 
$$ \mathrm{Mod}({\Bbb H}_g) \rightarrow \mathrm{Mod}(\varSigma_g) $$ 
which sends a mapping class $[\Psi]$ of an orientation preserving homeomorphism 
$\Psi: {\Bbb H}_g \rightarrow {\Bbb H}_g$ 
to the mapping class 
$[\Psi|_{\partial {\Bbb H}_g}]$ 
of $\Psi|_{\partial {\Bbb H}_g}: \partial {\Bbb H}_g \rightarrow \partial {\Bbb H}_g$. 
This homomorphism is injective (\cite[Theorem~3.7]{FM}),  
and not surjective (\cite[Section~3.12]{Suzuki}). 
We also call the  homomorphic image of $\mathrm{Mod}({\Bbb H}_g)$ in $\mathrm{Mod}(\varSigma_g)$ 
the  handlebody group, and 
also call the homomorphic image of $\mathcal{H}({\Bbb H}_g) $ in $\mathrm{Mod}(\varSigma_g)$, 
the  hyperelliptic handlebody group.  
As subgroups of $\mathrm{Mod}(\varSigma_g)$, 
we have 
$$\mathcal{H}({\Bbb H}_g) = \mathrm{Mod}({\Bbb H}_g) \cap \mathcal{H} (\varSigma_g).$$ 
We have 
$\mathrm{Mod}({\Bbb H}_2)= \mathcal{H}({\Bbb H}_2)$ 
since $\mathrm{Mod}(\varSigma_2) = \mathcal{H}(\varSigma_2)$ holds. 
If $g \ge 2$, then $\mathrm{Mod}({\Bbb H}_g)$ is of infinite index in $\mathrm{Mod}(\varSigma_g)$; 
If $g \ge 3$, then $\mathcal{H} ({\Bbb H}_g)$ is of infinite index in $\mathrm{Mod}({\Bbb H}_g)$, 
see Remark~\ref{rem_index} in Appendix~\ref{section_appendix}.

In the end of this section, 
we give a description of $\mathcal{H}({\Bbb H}_g)$ 
via $SW_{2g+2}$. 
Any element of $\mathrm{SHomeo}_+({\Bbb H}_g)$ 
fixes the fixed points set of $\mathcal{S}: {\Bbb H}_g \rightarrow {\Bbb H}_g$ 
as a set, 
and hence such an element descends to a homeomorphism on ${\Bbb H}_g/\mathcal{S} \simeq D^3$ 
which preserves ${\bf A}$ as a set. 
Thus a map 
$$Q: \mathcal{H}({\Bbb H}_g) \rightarrow SH_{2g+2}$$ 
is obtained.  
When we think $\mathcal{H}({\Bbb H}_g)$ as the subgroup of $\mathrm{Mod}(\varSigma_g)$ 
(resp. $SH_{2g+2}$ as the subgroup of $\mathrm{Mod}(\varSigma_{0, 2g+2})$),  
we have the restriction of the homomorphism $q$ in Theorem~\ref{thm_BH}
\begin{equation}
\label{equation_BH}
Q = q|_{\mathcal{H}({\Bbb H}_g)}: \mathcal{H}({\Bbb H}_g) \rightarrow SH_{2g+2}. 
\end{equation}
The following theorem, which is a version of Birman-Hilden's theorem~\ref{thm_BH},   
is useful.

\begin{thm}
\label{thm_HiKi}
For $g \ge 2$, 
the map $Q: \mathcal{H}({\Bbb H}_g) \rightarrow SH_{2g+2}$ 
is well-defined, and it is a surjective homomorphism with the kernel $\langle \iota  \rangle$. 
In particular, 
$$\mathcal{H}({\Bbb H}_g) / \langle \iota  \rangle \simeq 
SH_{2g+2}.$$  
\end{thm}

\noindent
For a proof of Theorem~\ref{thm_HiKi}, see Appendix~\ref{section_appendix}. 
By Theorems~\ref{thm_HiWi} and \ref{thm_HiKi}, 
we have 
$$\mathcal{H}({\Bbb H}_g) / \langle \iota \rangle \simeq SW_{2g+2}/  \langle \Delta^2 \rangle 
\simeq SH_{2g+2} .$$

\noindent
Lemma~\ref{lem_NTtype} and Theorem~\ref{thm_HiKi} imply the following. 

\begin{lem}
\label{lem_NTtype2}
We have $\delta(\mathcal{H}({\Bbb H}_g)) = \delta(SH_{2g+2})$ for $g \ge 2$. 
\end{lem}

\begin{center}
\begin{figure}
\includegraphics[width=5.2in]{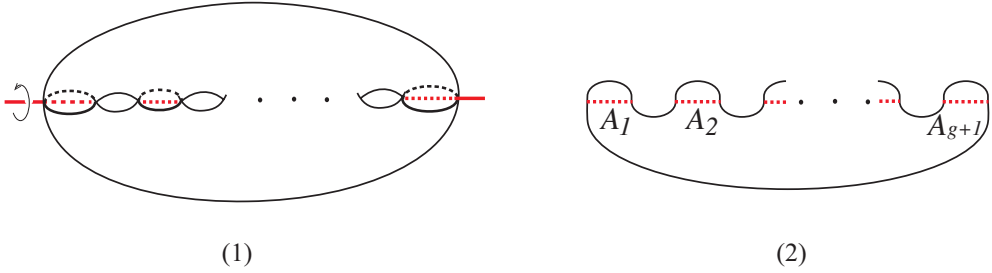}
\caption{(1) Hyperelliptic involution 
$\mathcal{S}: {\Bbb H}_g \rightarrow {\Bbb H}_g$ (on a handlebody ${\Bbb H}_g$) 
which is a rotation by $180$ degrees about the indicated axis. 
The fix points set of $\mathcal{S}$ is illustrated by dotted segments. 
The restriction 
$\mathcal{S}= \mathcal{S}|_{\partial {\Bbb H}_g}: \partial {\Bbb H}_g \rightarrow \partial {\Bbb H}_g$  
defines an involution on $\partial {\Bbb H}_g \simeq \varSigma_g$. 
(2) ${\Bbb H}_g/ \mathcal{S} \simeq D^3$ with ${\bf A}= A_1 \cup A_2 \cup \cdots \cup A_{g+1}$.}
\label{fig_handle}
\end{figure}
\end{center}

\subsection{Disk twists}
\label{subsection_DiskTwists}

We will discuss a method of constructing links in $S^3$ 
whose complements are the same. 
Let $L$ be a link in $S^3$. 
We denote a  tubular neighborhood  of $L$ by $\mathcal{N}(L)$, 
and  the exterior of $L$, that is $S^3 \setminus \mathrm{int}(\mathcal{N}(L))$ by $\mathcal{E}(L)$. 
Suppose that $L$ contains an unknot  $K \subset L$. 
Then $\mathcal{E}(K)$ (resp. $\partial \mathcal{E}(K)$)
is homeomorphic to a solid torus (resp. torus). 
We denote the link $ L\setminus K$ by $L_K$. 
We take a disk $D$ bounded by the longitude of $ \mathcal{N}(K)$. 
By using $D$, 
we define two homeomorphisms 
$$T= T_D:  \mathcal{E}(K) \rightarrow   \mathcal{E}(K)$$
called the {\it (left-handed) disk twist about} $D$ and 
$$H= H_D:   \mathcal{E}(L) 
(=  \mathcal{E}(K \cup L_K)) \rightarrow \mathcal{E}(K \cup T_D(L_K))$$ 
as follows. 
We cut $ \mathcal{E}(K)$ along $D$. 
 We have  resulting two sides obtained from $D$. 
Then we reglue the two sides by rotating either of the sides  $360$ degrees 
so that the mapping class of the restriction 
$T|_{\partial  \mathcal{E}(K)}: \partial \mathcal{E}(K)\rightarrow 
\partial \mathcal{E}(K)$ defines the left-handed Dehn twist about $\partial D$, 
see Figure~\ref{fig_left-handed}(1). 
Such an operation defines the former homeomorphism 
$T_D:  \mathcal{E}(K) \rightarrow   \mathcal{E}(K)$. 
If $m$ segments of $L_K$ pass through $D$, then 
$T(L_K)$ is obtained from $L_K$ by adding a full twist braid $\Delta_m^2$ near $D$. 
In the case  $m=2$, see Figure~\ref{fig_left-handed}(2). 
Notice that 
$T_D:  \mathcal{E}(K) \rightarrow   \mathcal{E}(K)$ determines the latter  homeomorphism 
$$H= H_D:  \mathcal{E}(L) 
(=  \mathcal{E}(K \cup L_K)) \rightarrow \mathcal{E}(K \cup T(L_K)).$$

For any integer $\ell \ne 0$, we have a homeomorphism of 
the $\ell$th power $T^{\ell} = T_D^{\ell}:  \mathcal{E}(K) \rightarrow   \mathcal{E}(K)$ 
so that 
$T^{\ell}|_{\partial  \mathcal{E}(K)}: \partial  \mathcal{E}(K) \rightarrow \partial  \mathcal{E}(K)$
is the $\ell$th power of the left-handed Dehn twist about $\partial D$. 
Observe that  $T^{\ell} = T_D^{\ell}$  converts $L = K \cup L_K$ into a link 
 $K \cup T^{\ell}(L_K)$ in $S^3$ 
 such that $S^3 \setminus L$ is homeomorphic to $S^3 \setminus (K \cup T^{\ell}(L_K))$. 
 We denote by $H_D^{\ell}$, a homeomorphism: $\mathcal{E}(L)(= \mathcal{E}(K \cup L_K)) \rightarrow \mathcal{E}(K \cup T^{\ell}(L_K))$.

\begin{center}
\begin{figure}
\includegraphics[width=5in]{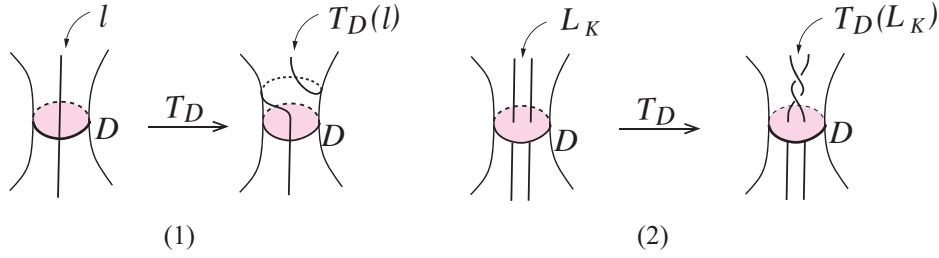}
\caption{(1) Action of $T_D$ on $\ell$, 
where $\ell$ is an arc on $\partial \mathcal{E}(K)$ 
which passes through $\partial D$. 
(2) Local picture of $L_K$ and its image $T_D(L_K)$. }
\label{fig_left-handed}
\end{figure}
\end{center}

The following remark is used 
in the proof of Proposition~\ref{prop_main}. 

\begin{rem}
\label{rem_LocalMove} 
Let $L$ be a link in $S^3$. 
Suppose that $L$ contains 
two unknotted components $K$ and $K'$ such that 
$K \cup K'$ is the Hopf link. 
Let $D$ be a disk bounded by the longitude of $\mathcal{N}(K)$. 
We assume that parallel $m \ge 1$ segments of $L_K \setminus K'$ 
pass through $D$, 
see Figure~\ref{fig_LocalMove}(1) in the case $m=2$. 
($L_K \setminus K'$ may intersect with the disk bounded by the longitude of $\mathcal{N}(K')$.) 
Pushing $D$  along the meridian of $\mathcal{N}(K)$,  
one can put the resulting disk $D$ as in Figure~\ref{fig_LocalMove}(2). 
The small circles in Figure~\ref{fig_LocalMove}(2) indicate the intersection 
between $L_K$ and $D$. 
Now we consider the disk twist $T$ about $D$, that is, 
we cut $\mathcal{E}(K)$ along $D$ and we reglue the two sides obtained from $D$ 
by rotating one of the sides by 360 degrees. 
In this case, one can choose the intersection point $D \cap K'$ as an origin of the rotation of $D$.  
As a result, we get a local diagram of the link $T(L_K)$ shown in Figure~\ref{fig_LocalMove}(3) 
so that  $T= T_D$ fixes $K'$. 
(See $K'$ and $T_D(K')$ in Figure~\ref{fig_LocalMove}(2)(3).) 
\end{rem}

\begin{center}
\begin{figure}
\includegraphics[width=5.3in]{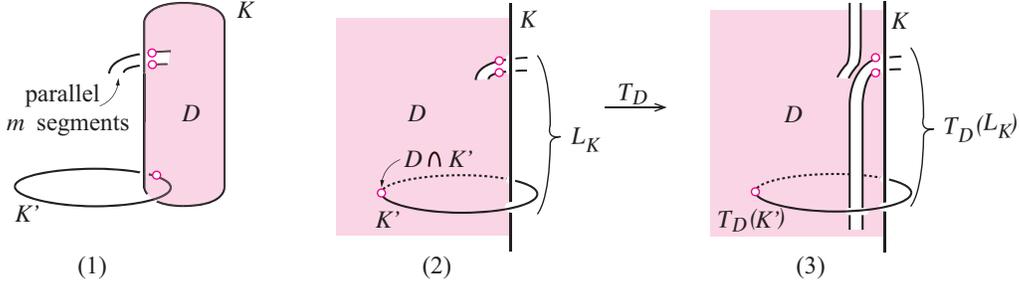} 
\caption{Small circles in (1), (2) (resp. (3)) 
indicate the intersection between $L_K$ and $D$ (resp. between $T_D(L_K)$ and $D$).}
\label{fig_LocalMove}
\end{figure}
\end{center}

\section{Proof of Proposition~\ref{prop_main}} 
\label{section_proofs}

 We introduce a sequence of braids $w_{2n} \in SW_{2n}$. 
Let 
$$w_6 
=\sigma_{2}^{-1} \sigma_1^{-1} \sigma_3 \sigma_2 \sigma_4 \sigma_3^2 \sigma_4 \sigma_3  
= \sigma_{2}^{-1} \sigma_1^{-1} \sigma_3 \sigma_2 \sigma_4 \sigma_3 \sigma_4 \sigma_3 \sigma_4  
\in SB_{(5)},$$
see Figure~\ref{fig_wicket6}(1). 
(For the definition of the subgroup $SB_{(m-1)}$ of $SB_m$, see Section~\ref{subsection_Sbraid}.) 
Since $^{w_6}{\bf A}$ is isotopic to ${\bf A}$ relative to $\partial {\bf A}$, 
we have $w_6  \in SW_6$. 
In order to define a sequence of braids $w_8, w_{10}, \cdots$, 
we introduce $x_{4n+8}, y_{4n+8} \in SB_{(4n+7)}$ 
for each  $n \ge 0$ as follows. 
\begin{eqnarray*}
x_{4n+8} &=& 
\sigma_5 \sigma_2^{-1} \sigma_1^{-1} 
(\sigma_3 \sigma_4 \cdots \sigma_{4n+5})
(\sigma_2 \sigma_3 \cdots \sigma_{4n+4}) \sigma_{4n+6} (\sigma_{4n+5})^2 \sigma_{4n+6}, 
\\
 y_{4n+8} &=& 
 (\sigma_1 \sigma_2 \cdots \sigma_{4n+5})^4 
 \sigma_{4n+6}  \sigma_{4n+5} \sigma_{4n+4} (\sigma_{4n+3})^2  \sigma_{4n+4}  \sigma_{4n+5} \sigma_{4n+6}, 
\end{eqnarray*}
see Figure~\ref{fig_z_braid}(1)(2). 
It is straightforward to check that 
$^{b}{\bf A}$ is isotopic to ${\bf A}$ relative to $\partial {\bf A}$ 
when $b= x_{4n+8}$, $y_{4n+8}$. 
Thus $x_{4n+8}$, $y_{4n+8} \in SW_{4n+8}$. 
We let 
$$w_{4n+8} = x_{4n+8} (y_{4n+8})^n \in SB_{(4n+7)} \cap SW_{4n+8},$$
where $(y_8)^0 = 1 \in SB_{(7)}$.   
For example, in the case of $n=0$, 
$$w_8= x_8 (y_8)^0 = \sigma_5 \sigma_2^{-1} \sigma_1^{-1} 
\sigma_3 \sigma_4 \sigma_{5}
\sigma_2 \sigma_3  \sigma_{4} \sigma_{6} \sigma_{5}^2 \sigma_{6},$$ 
see Figure~\ref{fig_wicket6}(2). 
Notice that the last two strings 
($(4n+7)$th and $(4n+8)$th strings) of both $x_{4n+8}$ and $y_{4n+8}$ 
define the identity $1 \in SB_2$, 
see Figure~\ref{fig_z_braid}(1)(2). 
Thus we obtain the $(4n+6)$-spherical braid 
by removing the last two strings from $w_{4n+8}$. 
In the case $n \ge 1$, we denote by $w_{4n+6}$, the resulting $(4n+6)$-braid. 
Said differently if we let $x_{4n+6}$ (resp. $y_{4n+6}$) be the $(4n+6)$-spherical braid 
obtained from $x_{4n+8}$ (resp. $y_{4n+8}$) 
by removing the last two strings from $x_{4n+8}$ (resp. $y_{4n+8}$), 
then $w_{4n+6}$ is given by 
$$w_{4n+6}= x_{4n+6}(y_{4n+6})^n,$$ 
see Figure~\ref{fig_z_braid}(3)(4).  
Clearly $w_{4n+6} \in SW_{4n+6}$, 
since $w_{4n+8} \in SW_{4n+8}$. 

\begin{rem}
The braid $w_6 \in SW_6$ is not the same as 
the braid  which is obtained from $w_8$ as above. 
The latter braid is not used in the rest of the paper.  
\end{rem}

\begin{center}
\begin{figure}
\includegraphics[width=3in]{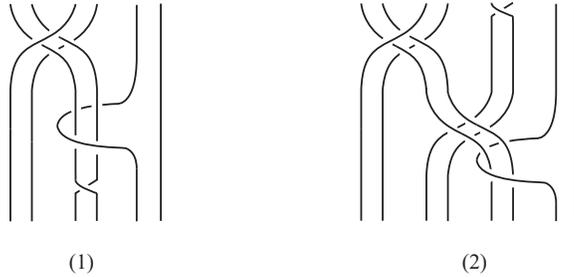}
\caption{(1) $w_6 \in SB_{(5)} \cap SW_6$. 
(2) $w_8 \in SB_{(7)} \cap SW_8$.}
\label{fig_wicket6}
\end{figure}
\end{center}

\begin{center}
\begin{figure}
\includegraphics[width=5.5in]{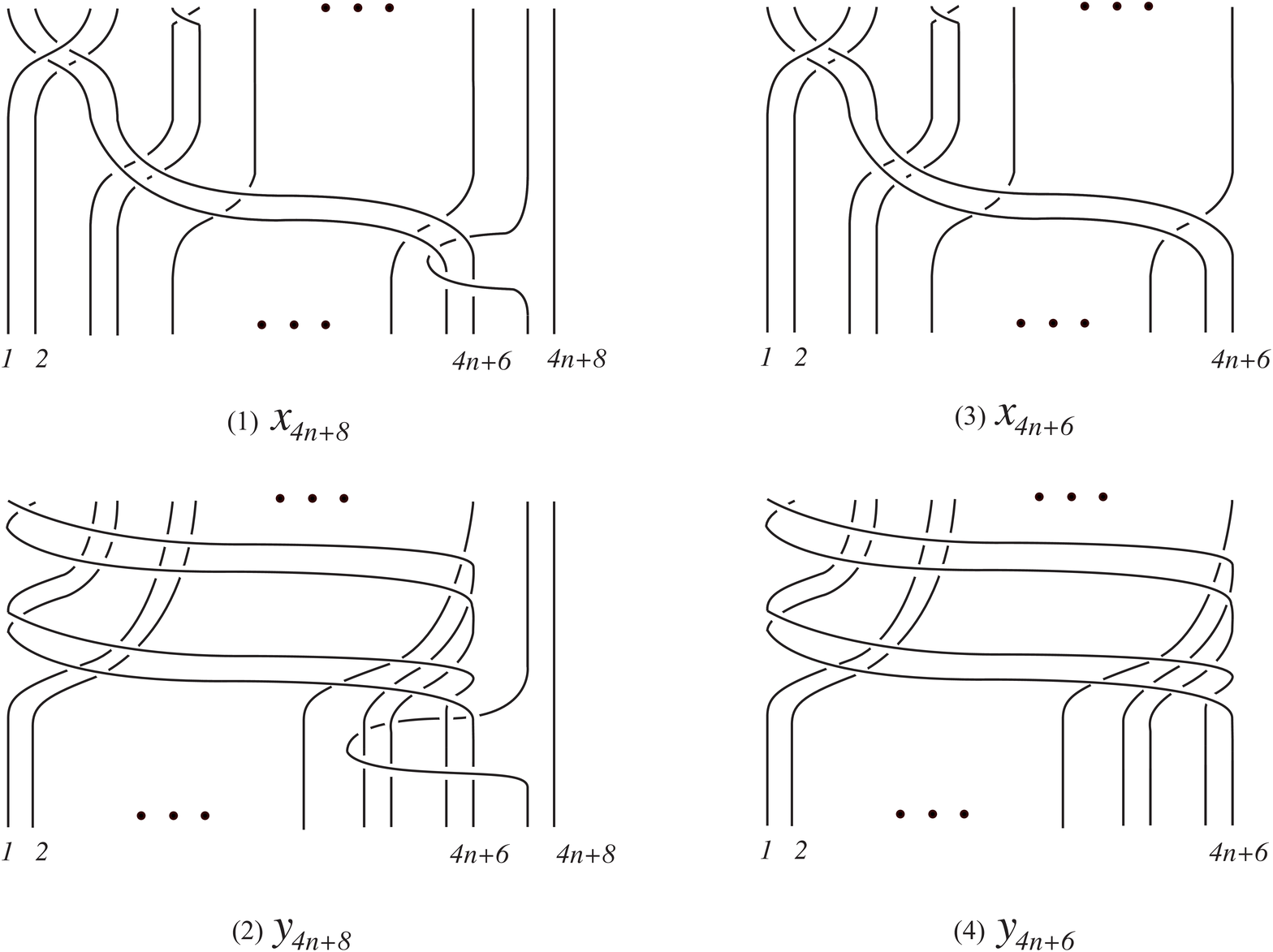}
\caption{(1) $x_{4n+8}$ and (2) $y_{4n+8} \in SB_{(4n+7)} \cap SW_{4n+8}$. 
(3) $x_{4n+6}$ and (4) $y_{4n+6} \in SW_{4n+6}$. 
(In (1)--(4), dots indicate parallel strings.)} 
\label{fig_z_braid}
\end{figure}
\end{center}

In the proof of the next lemma, 
we use some basic facts on {\it train tracks} of pseudo-Anosov homeomorphisms. 
See \cite{BH,PaPe} for more details. 
For a quick review about train tracks, 
see \cite[Section~2.1]{Kin} 
which contains terms and basic facts needed in this paper.

\begin{lem}
\label{lem_6braid}
The braid $\underline{w_6} \in B_5 $  is pseudo-Anosov, 
and $\lambda(\underline{w_6}) $ equals  $\kappa$, 
where $\kappa$ is the constant given in Proposition~\ref{prop_main}. 
\end{lem}

\begin{proof}
We choose a train track $\tau \subset D_5$ 
with non-loop edges $p_1, \cdots, p_6$ 
as in Figure~\ref{fig_train_suspend}(1), 
where $c_1, \cdots, c_5$ are punctures of $D_5$. 
Each component of $D_5 \setminus \tau$ 
is either a $1$-gon with one puncture, a $3$-gon (without punctures), or 
a $1$-gon containing the boundary of the disk 
(see the illustration of $\tau \subset D_5 \times \{0\}$ on the bottom of Figure~\ref{fig_train_suspend}(3)). 
We consider the braid $\underline{w_6}$ with base points $c_1, \cdots, c_5$. 
We push $\tau $ on $ D_5$ 
along the suspension flow on $S^3 \setminus \mathrm{br}(\underline{w_6})$, 
then we get the train track $\tau'$ on $D_5 \times \{1\} $ 
illustrated in Figure~\ref{fig_train_suspend}(2). 
This implies that 
there exists a representative $f \in \Gamma(\underline{w_6})$ such that $\tau'= f(\tau)$. 
Here, the edge $(p_i)$ of $\tau'$ in Figure~\ref{fig_train_suspend}(2) 
denotes the image of $p_i$ under $f$.

We see that  $f(\tau)$ is carried by $\tau$, 
and hence $\tau$ is an invariant train track\footnote{One can use the software Trains \cite{Hall} 
to find  invariant train tracks for pseudo-Anosov braids.} 
for $\Gamma(\underline{w_6})$.
Let $N(\tau)$ be a fibered (tie) neighborhood of $\tau$ 
whose fibers (ties) are segment given by 
a retraction $\mathfrak{R}: \mathcal{N}(\tau) \to \tau$. 
Then we get a train track representative $\mathfrak{p}= \mathfrak{R} \circ f|_{\tau}: \tau \rightarrow \tau$ 
for $\Gamma(\underline{w_6})$. 
It turns out that 
$p_1, \cdots, p_6$ are \footnote{We recall the terminology in Bestvina-Handel~\cite{BH}. 
An edge $e$ of $\tau$ for a train track representative $\mathfrak{p}: \tau \rightarrow \tau$ is called 
{\it infinitesimal} if $e$ is eventually periodic under $\mathfrak{p}$. 
Otherwise $e$ is called {\it real}.}real edges for $\mathfrak{p}$, 
and other loop edges of $\tau$ are periodic under $\mathfrak{p}$, and 
hence they are infinitesimal edges. 
The incident matrix $M_{\mathfrak{p}}$ with respect to real edges 
is given by 
$$M_{\mathfrak{p}} = \left[\begin{array}{cccccc}2 & 0 & 0 & 0 & 0 & 1 \\
2 & 0 & 0 & 2 & 1 & 0 \\
1 & 0 & 1 & 1 & 1 & 0 \\
0 & 0 & 2 & 1 & 2 & 0 \\
1 & 0 & 0 & 0 & 0 & 0 \\
0 & 1 & 0 & 0 & 0 & 0\end{array}\right].$$
For example, 
we get 
$\ ^{t}\!\left[\begin{array}{cccccc}0 & 0 & 1 & 2 & 0 & 0\end{array}\right]$ 
for the  $3$rd column of $M_{\mathfrak{p}} $, 
since $f(p_3)$ passes through $p_3$ once and $p_4$ twice in either direction. 
(See the edge path $(p_3)$ in Figure~\ref{fig_train_suspend}(2).) 
Since the $5$th power $M_{\mathfrak{p}}^5$ is positive,  $M_{\mathfrak{p}}$ is Perron-Frobenius and we conclude that 
$\underline{w_6}$ is  pseudo-Anosov. 
The characteristic polynomial of $M_{\mathfrak{p}}$ equals 
$$ (t-1)^2 (t^4-2t^3 -2t^2-2t+1),$$ 
and the largest root $\kappa$ of  the second factor  gives us 
$ \lambda(\underline{w_6})$. 
\end{proof}

\noindent
The type of singularities of the (un)stable foliation for the pseudo-Anosov homeomorphism 
$\Phi = \Phi_{w_6}: \varSigma_{0,6} \rightarrow \varSigma_{0,6}$ 
can be read from the topological types of components of $\varSigma_{0,6} \setminus \tau$. 
See \cite[Section~3.4]{BH} which describes a construction of invariant measured foliations. 
Notice that two component of $\varSigma_{0,6} \setminus \tau$ are non punctured $3$-gons. 
The other components are  once punctured $1$-gons. 
Thus exactly two points in the interior of $\varSigma_{0,6}$ have $3$ prongs and 
each puncture of $\varSigma_{0,6}$ has a $1$ prong.

\begin{center}
\begin{figure}
\includegraphics[width=4.8in]{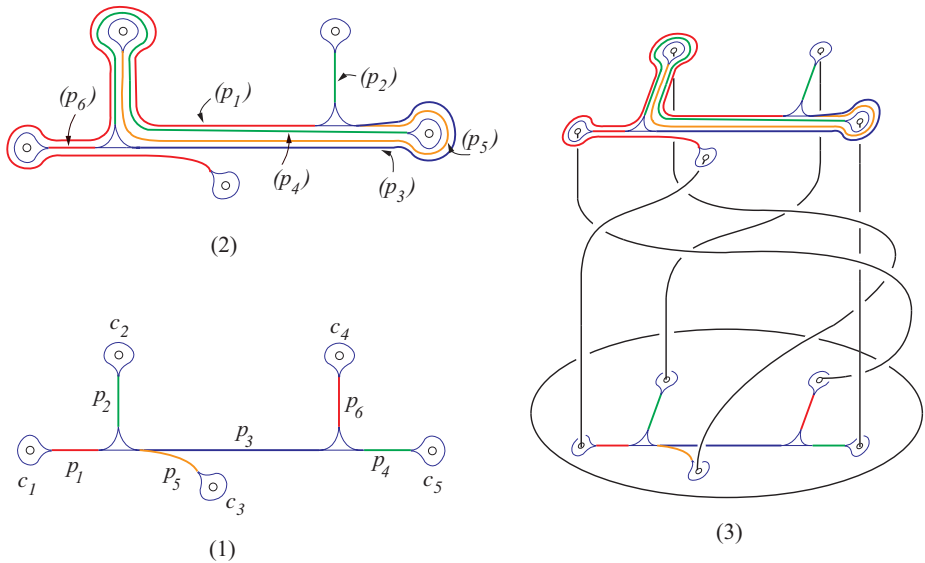}
\caption{(1) 
$\tau \subset D_5$. 
(2) $\tau'  \subset D_5$. 
(3) We get 
$\tau' \subset D_5 \times \{1\}$ by pushing $\tau \subset D_5 \times \{0\}$ 
along the suspension flow on $S^3 \setminus \mathrm{br}(\underline{w_6})$.}
\label{fig_train_suspend}
\end{figure}
\end{center}

Observe that 
$\mathrm{br}(\underline{w_6})$ is the link with $3$ components. 
The following lemma says that complements of both links 
$\mathrm{br}(\underline{w_6})$ and $L_0$ (Figure~\ref{fig_L10n95}) are the same. 
\begin{lem}
\label{lem_L10n95}
${\Bbb T}_{\underline{w_6}} $ 
 is homeomorphic to $S^3 \setminus L_0$. 
 In  particular $S^3 \setminus L_0$ is a hyperbolic fibered $3$-manifold. 
\end{lem}

\begin{proof}
We use another diagram of $L_0$ illustrated in Figure~\ref{fig_L10n95_2}(1). 
The link $L_0$ contains two unknots 
$K$ and $K^0$ so that $K \cup K^0$ is the Hopf link. 
We take the disk $D$ bounded by the longitude of $\mathcal{N}(K)$. 
We may assume that $L_0 \setminus K (= (L_0)_K)$ intersect with $D$ at the three points indicated by 
small circles in  the same figure. 
We apply the argument (in the case $m=2$) 
of Remark~\ref{rem_LocalMove}, 
and consider the disk twist about $D$, 
see Figure~\ref{fig_L10n95_2}(2). 
It turns out that  
$K \cup T_D (L_0 \setminus K)$ is of the form $\mathrm{br}(\underline{w_6})$, 
see Figure~\ref{fig_L10n95_2}(3).  
Thus $S^3 \setminus L_0$ is homeomorphic to 
$S^3 \setminus \mathrm{br}(\underline{w_6}) (= S^3 \setminus (K \cup T_D (L_0 \setminus K)))$. 
Since ${\Bbb T}_{\underline{w_6}}$ is homeomorphic to $ S^3 \setminus \mathrm{br}(\underline{w_6})$ 
and $\underline{w_6}$ is pseudo-Anosov by Lemma~\ref{lem_6braid}, 
we complete the proof. 
\end{proof}

\begin{center}
\begin{figure}
\includegraphics[width=5.5in]{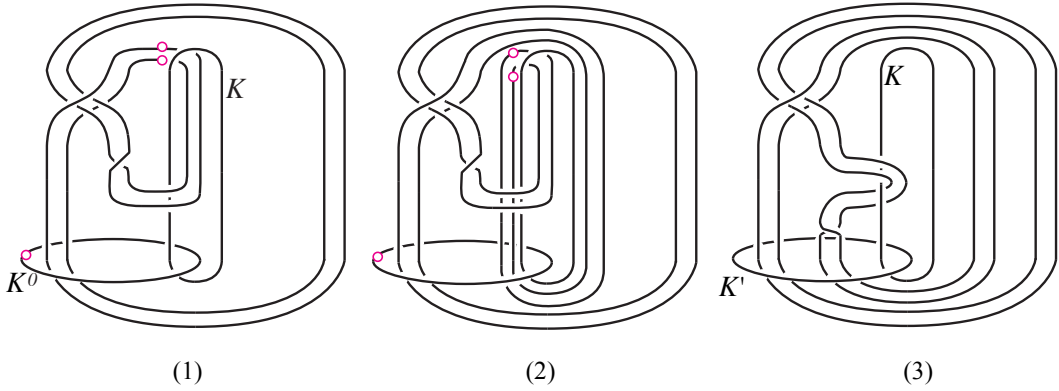}
\caption{(1) A diagram of $L_0$. 
($K \cup K^0$ is the Hopf link.)
(2) $K \cup T_D (L_0 \setminus K)$. 
(3) $\mathrm{br}(\underline{w_6})$.} 
\label{fig_L10n95_2}
\end{figure}
\end{center}

Let $F_{g,n}$ be a compact, connected, orientable surface of genus $g$ with $n$ boundary components. 
Let $M_0$ be the exterior of the link $L_0$. 
By Lemma~\ref{lem_L10n95}, we let 
$a_6 \in H_2(M_0, \partial M_0; {\Bbb Z})$ be the homology class of the $F_{0,6}$-fiber 
of the fibration on $M_0$ whose monodromy is described by $w_6 \in SB_6$.

\begin{lem} 
\label{lem_B_Pro}
${\Bbb T}_{\underline{w_{4n+8}}}$ 
 is homeomorphic to ${\Bbb T}_{\underline{w_6}}$ 
for each $n \ge 0$. 
In particular  $\underline{w_{4n+8}} \in B_{4n+7}$ is pseudo-Anosov 
for each $n \ge 0$. 
\end{lem}

\begin{proof}
We prove that 
$S^3 \setminus \mathrm{br}(\underline{w_{4n+8}})$ is homeomorphic to 
$S^3 \setminus \mathrm{br}(\underline{w_6})$. 
To do this, we use Remark~\ref{rem_LocalMove}  twice. 
The braided link $\mathrm{br}(\underline{w_6})$ contains two unknotted components, 
the braid axis $K'$ and the closure of the last string of $\underline{w_6}$, say $K$ 
so that $K \cup K'$ is the Hopf link. 
Let $D'$ be the disk bounded by the longitude of $\mathcal{N}(K')$. 
Consider the $n$th power of the disk twist $T_{D'}^n$ for $n \ge 0$.  
Following Remark~\ref{rem_LocalMove}, 
we take the point $D' \cap K$ as an origin of the  rotation of $D'$ for the disk twists $T_{D'}^n$. 
Then we have the diagram of 
$K' \cup T_{D'}^n(\mathrm{br}(\underline{w_6}) \setminus K') = K' \cup T_{D'}^n(\mathrm{cl}(\underline{w_6}))$ 
shown in Figure~\ref{fig_twisting}(2). 
Note that $  T_{D'}^n(\mathrm{cl}(\underline{w_6}))$ is isotopic to $\mathrm{cl}(\underline{w_6} \Delta^{2n})$, 
that is the closure of $\underline{w_6} \Delta^{2n} = 
 \underline{w_6} (\sigma_1 \sigma_2 \sigma_3 \sigma_4 )^{5n}$. 
 Thus 
 $$K' \cup T_{D'}^n(\mathrm{br}(\underline{w_6} )\setminus K') = \mathrm{br} (\underline{w_6} \Delta^{2n}),$$ 
 and $H_{D'}^n$ is a homeomorphism from $ \mathcal{E}( \mathrm{br}(\underline{w_6}))$ 
 to $\mathcal{E}(K' \cup T_{D'}^n(\mathrm{br}(\underline{w_6} )\setminus K')) \simeq 
 \mathcal{E}(\mathrm{br} (\underline{w_6} \Delta^{2n}))$. 
(See Section~\ref{subsection_DiskTwists} for  $H_{D'}^n$.) 
The closure of the last string of $\underline{w_6} \Delta^{2n} \in B_5$ 
is an unknot, say $K''$ 
which bounds a disk $D''$. 
(In the case $n=0$, we have $K= K''$, $D= D''$ and $H_{D'}^n$ equals the identity map.)
We apply the argument of Remark~\ref{rem_LocalMove}, 
and consider the disk twist $T_{D''}$ 
taking the point of $D'' \cap K'$ as an origin of the rotation of the disk $D''$ for $T_{D''}$. 
It turns out that 
$$K'' \cup T_{D''} (\mathrm{br} (\underline{w_6} \Delta^{2n}) \setminus K'') = 
\mathrm{br}(\underline{x_{4n+8}} \  \underline{(y_{4n+8})^n}) (= \mathrm{br}(\underline{w_{4n+8}})).$$ 
To see the equality, we first note that 
$\mathrm{br}(\underline{w_6} \Delta^{2n}) \setminus (K' \cup K'')$ 
  intersects with $D''$ at $2+ 4n$ points, see Figure~\ref{fig_twisting}(2). 
  We arrange, by an isotopy, $2+4n$ intersection points in a line 
 which is parallel to $K''$. 
  Then view the image $T_{D''} (\mathrm{br} (\underline{w_6} \Delta^{2n}) \setminus K'')$ 
  following the local move under the disk twist $T_{D''}$, see Figure~\ref{fig_LocalMove}(2)(3). 
  (Here $m$ in Figure~\ref{fig_LocalMove}(1) is equal to $2+4n$.) 
Figure~\ref{fig_abstractTwisting} explains this procedure in the case $n=2$.

The above equality implies that 
$H_{D''}$ is a homeomorphism 
from $ \mathcal{E}(\mathrm{br} (\underline{w_6} \Delta^{2n}))$ to 
$\mathcal{E}(K'' \cup T_{D''} (\mathrm{br} (\underline{w_6} \Delta^{2n}) \setminus K'')) \simeq  \mathcal{E}(\mathrm{br}(\underline{w_{4n+8}}))$. 
The composition of the maps 
 $H_{D''} \circ H_{D'}^n$ sends $ \mathcal{E}(\mathrm{br}(\underline{w_6})) $ 
to $\mathcal{E}(\mathrm{br}(\underline{w_{4n+8}}))$. 
Thus 
$S^3 \setminus \mathrm{br}(\underline{w_6})$ is homeomorphic to 
$S^3 \setminus \mathrm{br}(\underline{w_{4n+8}})$. 
Since ${\Bbb T}_{\underline{w_6}} = S^3 \setminus \mathrm{br}(\underline{w_6})$ is hyperbolic, 
we conclude that by Lemma~\ref{lem_bar}, 
the braid  $\underline{w_{4n+8}}$ is pseudo-Anosov. 
\end{proof}

\begin{center}
\begin{figure}
\includegraphics[width=3in]{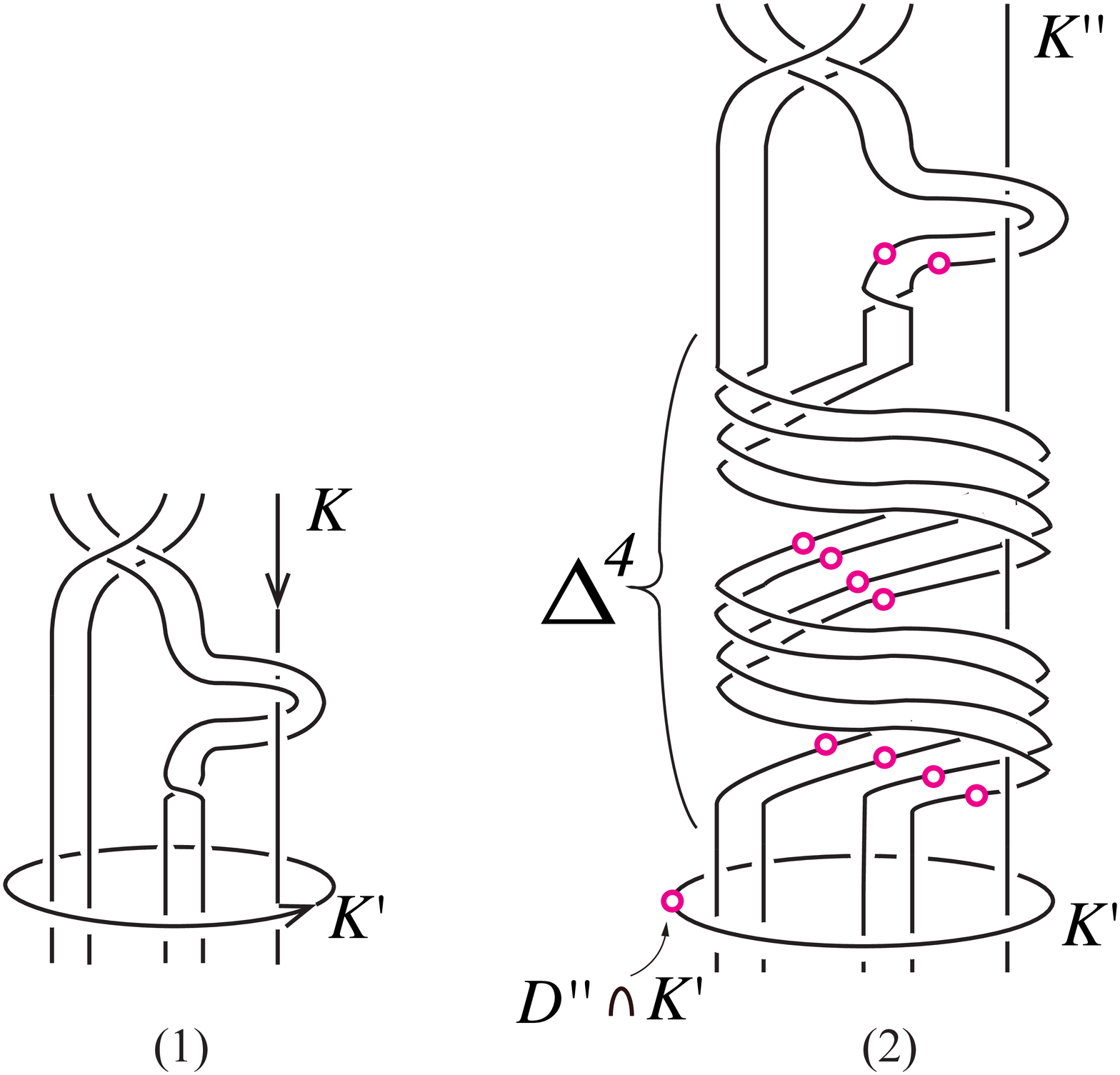}
\caption{Identifying the top and bottom of strings, we get  
(1) $\mathrm{br}(\underline{w_6})$ and 
(2) $\mathrm{br} (\underline{w_6} \Delta^{2n})$, 
where $n=2$ in this figure. 
Figure (1) indicates orientations of $K$ and $K'$.
Figure (2) explains 
 $\mathrm{br}(\underline{w_6} \Delta^{2n}) \setminus K''$ 
  intersects with $D''$ at $3+ 4n$ points (indicated by small circles). 
  Hence $\mathrm{br}(\underline{w_6} \Delta^{2n}) \setminus (K' \cup K'')$ 
  intersects with $D''$ at $2+ 4n$ points.}
\label{fig_twisting}
\end{figure}
\end{center}

\begin{center}
\begin{figure}
\includegraphics[width=4in]{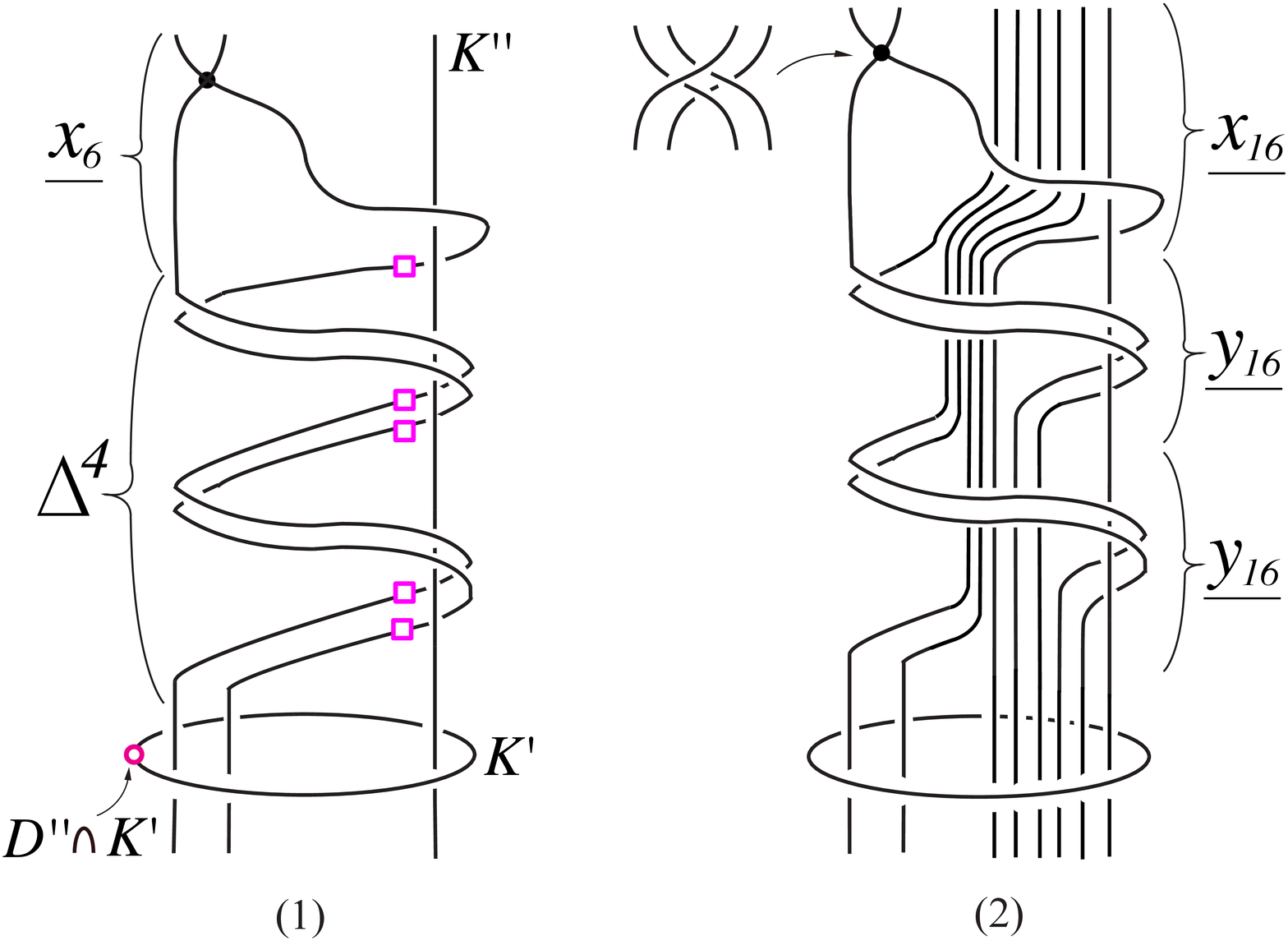}
\caption{
Illustrations of two braided links (1)  $\mathrm{br} (\underline{w_6} \Delta^{4})$ and 
(2)  
$\mathrm{br}(\underline{x_{16}} \  \underline{(y_{16})^2}) (= \mathrm{br}(\underline{w_{16}}))$. 
To get actual two braided links, 
we duplicate each of the braid strings except the last one. 
Each small rectangle $\Box$ in (1) represents some two small circles given in Figure~\ref{fig_twisting}(2). 
The `virtual'  crossings $\bullet$ in both figures (1)(2) mean  $\sigma_2^{-1}\sigma_1^{-1} \sigma_3 \sigma_2 $.} 
\label{fig_abstractTwisting}
\end{figure}
\end{center}

By Lemma~\ref{lem_B_Pro}, 
we let $a_{4n+8} \in H_2(M_0, \partial M_0; {\Bbb Z})$ be the homology class of the $F_{0, 4n+8}$-fiber 
of the fibration on $M_0$ 
whose monodromy is described by $w_{4n+8}$. 
To study properties of such a fibered class, 
we take a  $3$-punctured disk $F \simeq \varSigma_{0,4}$ 
embedded in $S^3 \setminus \mathrm{br}(\underline{w_6})$ 
so that $F$ is bounded by the unknotted component $K \subset \mathrm{br}(\underline{w_6})$, 
i.e,  
$F$ is an interior of the disk $D$ 
removed the set of $3$ points $( \mathrm{br}(\underline{w_6}) \setminus K) \cap D$. 
To choose an orientation of $F$, 
we consider an orientation of the last string of  $\underline{w_6}$ 
from the {\it top} to the {\it bottom}. 
This determines an orientation of $K$ 
(see Figure~\ref{fig_twisting}(1)), and 
we have an  orientation of $F$ induced by $K$. 
Let $\bar{F}$ be the oriented disk with $3$ holes (i.e, sphere with $4$ boundary components) embedded in 
$\mathcal{E}(\mathrm{br}(\underline{w_6}))$, 
which is obtained from $D$ removed the interiors of the $3$ disks 
whose centers are the above $3$ points. 
The fibered class $a_{4n+8}$ can be expressed by using $a_6$ and 
$[\bar{F}]  \in H_2(M_0, \partial M_0; {\Bbb Z})$  as follows.

\begin{lem}
\label{lem_FiberedClass}
We have 
$a_{4n+8}= (n+1) a_{6} + [\bar{F}] \in H_2(M_0, \partial M_0; {\Bbb Z})$ for each $n \ge 0$. 
In particular, 
the ray of $a_{4n+8}$ through the origin goes to the ray of $a_6$ 
as $n$ goes to $\infty$. 
\end{lem}

\begin{proof}
Recall that   $F_{a_6}$ is 
a minimal representative of $a_6 $. 
In other words, $F_{a_6}$ is a $F_{0,6}$-fiber of the fibration on $M_0$ 
associated to $a_6$. 
We consider the {\it oriented sum} 
$n F_{a_6}+ \bar{F}$ which is an oriented surface 
embedded in $M_0$. 
This surface is obtained 
by the cut and paste construction of parallel $n$ copies of $F_{a_6}$ and a copy of  $\bar{F}$. 
(For the construction of the oriented sum, 
see \cite[p104]{Thurston1} or \cite[Section~5.1.1]{Calegari}.) 
We take a surface $F''$ embedded in $M_0$, 
which is a disk with $(3+4n)$ holes as follows. 
Consider the disk $D''$ bounded by the longitude of $\mathcal{N}(K'')$. 
Then remove the interiors of small $(3+4n)$ disks from $D''$ 
whose centers are the $(3+4n)$ intersection points 
$(\mathrm{br}(\underline{w_6} \Delta^{2n}) \setminus K'') \cap D''$, see Figure~\ref{fig_twisting}(2). 
We denote by $F''$, the resulting disk with $(3+4n)$ holes. 
We see that the homeomorphism 
$H_{D'}^n:  \mathcal{E}(\mathrm{br}(\underline{w_6}))  \rightarrow 
\mathcal{E}(\mathrm{br} (\underline{w_6} \Delta^{2n}))$ 
in the proof of Lemma~\ref{lem_B_Pro} 
sends  $nF_{a_6}+ \bar{F}$ to $F''$. 
Hence 
$$[F'']= [n F_{a_6}+ \bar{F}] = n a_6+ [\bar{F}] \in H_2(M_0, \partial M_0; {\Bbb Z}).$$
Obviously 
$[H_{D'}^n(F_{a_6})] = [F_{a_6}] (= a_6)$. 
We now consider the oriented sum $H_{D'}^n(F_{a_6})+ F''$. 
Then 
$H_{D''}:  \mathcal{E}(\mathrm{br} (\underline{w_6} \Delta^{2n})) \rightarrow 
\mathcal{E}(\mathrm{br}(\underline{w_{4n+8}}))$ 
in the proof of Lemma~\ref{lem_B_Pro} 
sends $H_{D'}^n(F_{a_6})+ F''$ to the $F_{0, 4n+8}$-fiber of the fibration on 
$M_0 $ associated to $a_{4n+8}$.  
Putting all things together, 
we have 
$$a_{4n+8} = [H_{D'}^n(F_{a_6})+ F''] =[H_{D'}^n(F_{a_6})]+ [F'']= a_6+ n a_6+ [\bar{F}] = (n+1) a_6 + [\bar{F}].$$ 
Thus 
$$\displaystyle\lim_{n \to \infty} \tfrac{a_{4n+8}}{n+1} 
= \displaystyle\lim_{n \to \infty} (a_6+ \tfrac{ [\bar{F}]}{n+1}) = a_6.$$
This completes the proof. 
\end{proof}

Since the normalized entropy function $\mathrm{Ent}$ is constant on each ray through the origin in the fibered cone, 
Lemmas~\ref{lem_6braid}, \ref{lem_L10n95}, \ref{lem_B_Pro} and 
\ref{lem_FiberedClass} tell us that 
\begin{equation}
\label{equation_asymp}
\lim_{n \to \infty} |\chi (F_{0, 4n+8})| \log (\lambda(w_{4n+8})) =|\chi (F_{0, 6})| \log (\lambda(w_{6})) 
= 4 \log \kappa.
\end{equation}
Since $ |\chi (F_{0, 4n+8})|= 4n+6$  goes to $\infty$ as $n $ does, 
and the right-hand side is constant, 
we conclude that 
\begin{equation}
\label{equation_0}
\displaystyle \lim_{n \to \infty} \log \lambda(w_{4n+8}) = 0.
\end{equation}

Let $\Omega$ be a fibered face of $M_0$ 
such that $a_6 \in int(C_{\Omega})$. 
Lemma~\ref{lem_FiberedClass} implies that 
$a_{4n+8} \in int(C_{\Omega})$ for $n$ large. 
We shall prove in Lemma~\ref{lem_SameFace} 
that the fibered class $a_{4n+8}$ lies in $int(C_{\Omega})$ 
for {\it each} $n \ge 0$. 
Recall that $K$ and $K'$ are the unknotted components of $\mathrm{br} (\underline{w_6})$. 
We choose an orientation of $K'$ as in Figure~\ref{fig_twisting}(1). 
For an embedded surface $\hat{S}$ in $M_0 = \mathcal{E}(\mathrm{br} (\underline{w_6}))$, 
we denote by $\partial_K(\hat{S})$ and $\partial_{K'}(\hat{S})$, 
the components of the boundary $\partial \hat{S}$ of $\hat{S}$ 
which lie on $\partial \mathcal{N}(K)$  and $\partial \mathcal{N}(K')$ respectively. 
Let $\hat{\Phi}: F_{0,6} \rightarrow F_{0,6}$ be a pseudo-Anosov homeomorphism 
whose mapping class $[\hat{\Phi}]$ is described by $w_6$. 
(Thus ${\Bbb T}_{[\hat{\Phi}]}$  is homeomorpphic to $ \mathcal{E}(\mathrm{br} (\underline{w_6})) \simeq M_0$.)

\begin{lem}
\label{lem_SameFace} 
We have 
$a_{4n+8} \in int(C_{\Omega})$ for each $n \ge 0$. 
\end{lem}

\begin{proof}
The minimal representative $F_{a_6}$  is transverse to the suspension flow $\hat{\Phi}^t$ obviously, but 
$\bar{F}$ is not, since 
both  $\partial_K \bar{F}$ and $\partial_{K'} \bar{F}$ are parallel to flow lines of $\hat{\Phi}^t$. 
(See also Figure~\ref{fig_Surface}(3) and its caption.)
We prove that  the oriented sum $(n+1) F_{a_6}+ \bar{F}$ for $n \ge 0$  is transverse to $\hat{\Phi}^t$ 
(up to isotopy) 
and it intersects every flow line. 
This means that  $(n+1) F_{a_6}+ \bar{F}$ is a cross-section to $\hat{\Phi}^t$ 
for $ \mathcal{E}(\mathrm{br} (\underline{w_6}))$. 
By Theorem~\ref{thm_Fried}(3),  
we have 
$a_{4n+8}= [(n+1)F_{a_6}+ \bar{F}] \in  int(C_{\Omega})$.

By the proof of Lemma~\ref{lem_FiberedClass}, 
$(n+1) F_{a_6}+ \bar{F} $ becomes a fiber of the fibration on $M_0$ associated to $a_{4n+8}$. 
Hence we may assume that 
\begin{equation}
\label{equation_fiber4n8}
F_{a_{4n+8}}=(n+1) F_{a_6}+ \bar{F} . 
\end{equation}
We have the meridian and longitude basis $\{ m_K, \ell_K\}$ for  $\partial \mathcal{N}(K)$ 
and $\{ m_{K'}, \ell_{K'}\}$ for $\partial \mathcal{N}(K')$.
It follows that 
\begin{eqnarray*}
{[}\partial_K F_{a_{4n+8}}{]} &=& (n+1) m_K + \ell_K \ne \pm \ell_K \in H_1(\partial \mathcal{N}(K))\  \mbox{and}
\\
{[}\partial_{K'} F_{a_{4n+8}}{]} &=& (n+1)\ell_{K'} + m_{K'} \ne \pm m_{K'} \in H_1(\partial \mathcal{N}(K')).  
\end{eqnarray*}
This implies that  both $\partial_K F_{a_{4n+8}}$ and $\partial_{K'} F_{a_{4n+8}}$
are transverse to  every flow line of $\hat{\Phi}^t$, 
since $[\partial_K \bar{F} ] = \ell_K$ and 
$[\partial_K' \bar{F} ] = m_{K'}$. 
By (\ref{equation_fiber4n8}), 
$F_{a_{4n+8}}$ is an oriented sum obtained from the $(n+1)$ copies of $F_{a_6}$ and the surface $\bar{F}$. 
Hence the shape of the embedded surface $F_{a_{4n+8}}$ in $M_0$  is of a `spiral staircase' which turns round $(n+1)$ times along  $\ell_K$. 
Therefore $F_{a_{4n+8}}$ is transverse to $\hat{\Phi}^t$. 
Moreover $F_{a_{4n+8}}$ intersects every flow line of $\hat{\Phi}^t$ 
(by construction of $(n+1) F_{a_6}+ \bar{F} (= F_{a_{4n+8}})$), 
since so does $F_{a_6}$. 
This completes the proof. 
\end{proof}

\begin{lem}
\label{lem_4n6}
If  $n \ge 1$,  then  $w_{4n+6}$ is pseudo-Anosov and the equality 
$\lambda(w_{4n+6}) = \lambda(w_{4n+8})$ holds. 
In particular ${\Bbb T}_{w_{4n+6}}$ is a hyperbolic fibered $3$-manifold 
obtained from ${\Bbb T}_{w_{4n+8}}$ by Dehn fillings about the two cusps 
along the boundary slopes of the fiber 
associated to  $a_{w_{4n+8}}$. 
\end{lem}

\noindent
We work on the cusped $3$-manifold 
${\Bbb T}_{w_6} = {\Bbb T}_{\underline{w_6}} \simeq S^3 \setminus L_0$ instead of $M_0$ with boundary. 
To prove Lemma~\ref{lem_4n6}, 
we shall construct an invariant train track for $\Gamma(w_{4n+8})$ 
concretely, and study types of singularities of the unstable foliation $\mathcal{F}_{w_{4n+8}}$ of 
the pseudo-Anosov homeomorphism $\Phi_{w_{4n+8}}: \varSigma_{0,4n+8} \rightarrow \varSigma_{0,4n+8}$ 
which represents $\Gamma(w_{4n+8})$. 
The same idea in \cite[Section~3]{Kin} 
for the construction of train tracks  can be used. 
We repeat a similar argument  modifying some claims of \cite{Kin} in a suitable way for the present paper. 
Hereafter we use basic properties on branched surfaces. 
See \cite{Oertel} for more details on the theory of branched surfaces.

By using the pseudo-Anosov homeomorphism 
$\Phi = \Phi_{w_6}: \varSigma_{0,6} \rightarrow \varSigma_{0,6}$, 
we build the mapping torus 
$${\Bbb T}_{w_6}=  \varSigma_{0,6} \times {\Bbb R}/_{(x,t+1)= (\Phi(x),t)}$$ 
for $x \in \varSigma_{0,6}$ and $t \in {\Bbb R}$. 
Given a subset $U \subset \varSigma_{0,6}$, 
we define $U^t \subset {\Bbb T}_{w_6}$ to be the image 
$U \times \{t\}$ under the projection 
$p: \varSigma \times {\Bbb R} \rightarrow {\Bbb T}_{w_6}$. 
We have an orientation preserving homeomorphism 
$$h: S^3 \setminus \mathrm{br} (\underline{w_6}) \to 
{\Bbb T}_{w_6}.$$ 
Recall that $F$ is an oriented  $4$-punctured sphere in $S^3 \setminus \mathrm{br}(\underline{w_6})$. 
Choose an orientation of $\varSigma_{0,6}^v = \varSigma_{0,6} \times \{v\}$ for each $v \in {\Bbb R}$ so that 
its normal direction coincides with the flow $\Phi^t$ direction. 
We shall capture the image $h(F)$ in ${\Bbb T}_{w_6}$. 
To do this, 
let $s$ be a  segment between the punctures $c_5$ and $c_6$. 
Since $\Phi$ fixes $c_5$ and $c_6$ pointwise 
(see the last two strings of $w_6$ in Figure~\ref{fig_wicket6}(1)), 
 $s \cup \Phi (s)$ bounds a $2$-gon. 
 Viewing the image $\Phi(s)$, we see that the $2$-gon 
 contains the punctures $c_1$ and $c_2$, 
see Figure~\ref{fig_Surface}.  
Such a  $2$-gon  removed $c_1$ and $c_2$ is denoted by $S \subset \varSigma_{0,6}$. 
The segment $s^0 = s \times \{0\} $ is connected to $s^1 = s \times \{1\}$, and 
 these segments make a flowband $J= [s^0, s^1]$ 
 which is illustrated in Figure~\ref{fig_Surface}(3). 
Since $s^1 = (\Phi(s))^0$ in ${\Bbb T}_{w_6}$, 
the union 
$$S^0 \cup J (= (S \times \{0\}) \cup J) \subset {\Bbb T}_{w_6}$$ 
defines a $4$-punctured sphere, 
see Figure~\ref{fig_Surface}(3). 
The set of punctures of $F$ maps to the set of punctures of $S^0 \cup J$ under $h$. 
This tells us that $h(F)= S^0 \cup J$ up to isotopy. 
For simplicity, $S^0 \cup J $  in $ {\Bbb T}_{w_6}$ is denoted by  $F$.

We choose  $0 < \epsilon < 2 \epsilon < 1$. 
We push $F  (= S^0 \cup J)\subset {\Bbb T}_{w_6}$  along the flow lines for $\epsilon$ times 
so that the resulting $4$-punctured sphere, denoted by $F^{\epsilon}$, 
satisfies 
$$F^{\epsilon} \cap \varSigma_{0,6}^{\epsilon} = S^{\epsilon} (= S \times \{\epsilon\}),$$
see Figure~\ref{fig_Bsurface}(2) for $S^{\epsilon}$. 
By using  $F^{\epsilon}$ and 
$\varSigma_{0,6}^{2 \epsilon} = \varSigma_{0,6} \times \{2 \epsilon\}$ 
which corresponds to a fiber $F_{a_6}$ of the fibration on $M_0$, 
we set 
$$\widehat{\mathcal{B}}= F^{\epsilon} \cup \varSigma_{0,6}^{2\epsilon}.$$
We get the branched surface $\mathcal{B}$ from $\widehat{\mathcal{B}}$ 
(which agrees with the orientations of $F^{\epsilon} $ and $\varSigma_{0,6}^{2 \epsilon}$) 
after we modify the flowband 
$$[s^{\epsilon}, (\Phi(s))^{\epsilon}] = [s^{\epsilon}, s^1] \cup [(\Phi(s))^0, (\Phi(s))^{\epsilon}]$$
of $F^{\epsilon}$ near the segment  
$s^{2\epsilon} =F^{\epsilon} \cap \varSigma_{0,6}^{2 \epsilon} \in  [s^{\epsilon}, s^1]$. 
(cf. For the illustration of this modification, see \cite[Figure~14]{Kin}.)
By Lemmas~\ref{lem_B_Pro} and \ref{lem_FiberedClass}, 
there exists a $\varSigma_{0, 4n+8}$-fiber of the fibration on ${\Bbb T}_{w_6}$ 
with the monodromy $\Gamma(w_{4n+8})$. 
We denote such a fiber by $\varSigma_{w_{4n+8}}$. 
By (\ref{equation_fiber4n8}), 
we have  
\begin{equation}
\label{equation_sigma}
\varSigma_{w_{4n+8}} = F+ (n+1)\varSigma_{0,6}. 
\end{equation}
By the construction of $\mathcal{B}$, we see that 
$\varSigma_{w_{4n+8}}$ is carried by $\mathcal{B}$.

Let $\widehat{\mathcal{F}} \subset {\Bbb T}_{w_6}$ be the suspension of the unstable foliation $\mathcal{F}$ for $\Phi$. 
We may assume that 
the train track $\tau \subset D_5$ in the proof of Lemma~\ref{lem_6braid} lies on $\varSigma_{0,6}$. 
Then $\tau$ is an invariant train track for $\Gamma(w_6)= [\Phi]$. 
Theorem~\ref{thm_Fried}(1)(2) and Lemma~\ref{lem_SameFace} 
imply the following. 

\begin{lem}
\label{lem_consequence} 
The pseudo-Anosov homeomorphism 
$\Phi_{w_{4n+8}}: \varSigma_{w_{4n+8}} \rightarrow \varSigma_{w_{4n+8}}$ is precisely 
the first return map$: \varSigma_{w_{4n+8}} \rightarrow \varSigma_{w_{4n+8}}$ of  $\Phi^t$. 
Moreover $\mathcal{F}_{w_{4n+8}} =  \widehat{\mathcal{F}} \cap \varSigma_{w_{4n+8}}$. 
\end{lem}

\noindent
We turn to the construction of the branched surface $\mathcal{B}_{\Omega}$ 
which carries $\widehat{\mathcal{F}}$. 
First of all, we note that $\tau$ is obtained from $\Phi(\tau)$ 
by {\it folding} edges (or {\it zipping} edges), see Figure~\ref{fig_zip}. 
We choose  a family of train tracks $\{\tau_t\}_{0 \le t \le 1}$ on $\varSigma_{0,6}$ as follows. 
\begin{enumerate}
\item[(1)] 
$\tau_0 = \Phi(\tau)$. 

\item[(2)] 
$\tau_t$ at $t= \epsilon$ is a train track illustrated in Figure~\ref{fig_zip}(middle in the left column). 

\item[(3)] 
 $\tau_t = \tau$ for $2 \epsilon \le t \le 1$. 

\item[(4)] 
If $0 \le s < t \le 2 \epsilon$, then 
$\tau_t= \tau_s$ or 
$\tau_t$ is obtained from $\tau_s$ by folding edges between a cusp of $\tau_s$. 
\end{enumerate}
We let 
$$\mathcal{B}_{\Omega} = \bigcup_{0 \le t \le 1} \tau_t \times \{t\} \subset {\Bbb T}_{w_6}.$$
Since $\tau_1 \times \{1\}=  \tau_0 \times \{0\}$ in ${\Bbb T}_{w_6}$ 
(see the above conditions (1)(3)),  
it follows that 
$\mathcal{B}_{\Omega}$ is a branched surface. 
Since the invariant train track $\tau$ carries the unstable foliation $\mathcal{F}$, 
we see that $\mathcal{B}_{\Omega}$ carries $\widehat{\mathcal{F}}$. 
It is not hard to see that $\mathcal{B}_{\Omega}$ is transverse to the previous branched surface $\mathcal{B}$ 
(up to isotopy). 
Let 
\begin{equation}
\label{equation_tau}
\tau_{4n+8}= \varSigma_{w_{4n+8}} \cap \mathcal{B}_{\Omega},
\end{equation}
which is a branched $1$-manifold, 
see Figure~\ref{fig_TrainImage}(1). 
Since $\varSigma_{w_{4n+8}}$ is carried by $\mathcal{B}$, 
we may put $(n+1)$ copies of $\varSigma_{0,6}$ 
which is a part of $\varSigma_{w_{4n+8}}$ 
(see (\ref{equation_sigma})) 
into $\varSigma_{0,6} \times (2 \epsilon, 1)$. 
We may also assume that a copy of $F$ which is another part of  $\varSigma_{w_{4n+8}}$ 
satisfies that 
$S^{\epsilon} = F \cap \varSigma_{0,6}^{\epsilon}$. 
Then intersections  $\varSigma_{0,6}^{2 \epsilon} \cap \mathcal{B}_{\Omega}$ and 
$S^{\epsilon} \cap \mathcal{B}_{\Omega}$ 
(see Figure~\ref{fig_Bsurface}) 
together with $s^{2\epsilon} = F^{\epsilon} \cap \varSigma_{0,6}^{2 \epsilon}$ 
determine $\tau_{4n+8}$. 
More concretely, 
$\tau_{4n+8}$ is constructed from a copy of $S^{\epsilon} \cap \mathcal{B}_{\Omega}$ 
and $(n+1)$ copies of $\varSigma_{0,6}^{2 \epsilon} \cap \mathcal{B}_{\Omega}$, 
see (\ref{equation_sigma}), (\ref{equation_tau}). 
We label $q_1$, $q_2$, $q_3$, $p_1^{(j)}, p_2^{(j)}, \cdots, p_6^{(j)}$ ($1 \le j \le n+1$) 
for non-loop edges of $\tau_{4n+8}$. 
Notice that edges of $q_1$, $q_2$, $q_3$ come from the edges of $S^{\epsilon} \cap \mathcal{B}_{\Omega}$ 
and the rest of non-loop edges come from the edges of $\varSigma_{0,6}^{2 \epsilon} \cap \mathcal{B}_{\Omega}$. 
The $n+1$ edges $p_i^{(1)}, \cdots, p_i^{(n+1)}$ for each $1 \le i \le 6$ originate in the edges of 
$\varSigma_{0,6}^{2 \epsilon} \cap \mathcal{B}_{\Omega}$. 
If we fix $i$, then 
the number of the labeling $(j)$ in $p_i^{(j)}$ increases along the flow direction. 
We call $p_1^{(n+1)}, \cdots, p_6^{(n+1)}$ the top edges, 
$q_1, q_2, q_3$ the bottom edges, 
and $p_1^{(1)}, \cdots, p_6^{(1)}$ the second bottom edges etc. 
See Figure~\ref{fig_TrainImage}(1).

\begin{lem}
\label{lem_carry} 
The branched $1$-manifold $\tau_{4n+8}$ is a train track, and 
the unstable foliation $\mathcal{F}_{w_{4n+8}}$ of $\Phi_{w_{4n+8}}$ 
is carried by $\tau_{4n+8}$. 
\end{lem}

\begin{proof}
Since $a_{4n+8}$ lies in the same fibered cone as $a_6$ (Lemma~\ref{lem_SameFace}), 
$\mathcal{F}_{w_{4n+8}}$ is given by $\widehat{\mathcal{F}} \cap \varSigma_{w_{4n+8}}$  
(Lemma~\ref{lem_consequence}) 
and the suspension  $\widehat{\mathcal{F}}_{w_{4n+8}}$ 
of $\mathcal{F}_{w_{4n+8}} $
by $\Phi_{w_{4n+8}}$ is isotopic to $\widehat{\mathcal{F}}$, see \cite[Corollary~3.2]{McMullen2}. 
Since $\widehat{\mathcal{F}}$ is carried by $\mathcal{B}_{\Omega}$, so is $\widehat{\mathcal{F}}_{w_{4n+8}}$. 
Thus $\mathcal{F}_{w_{4n+8}}$ is carried by 
$ \varSigma_{w_{4n+8}} \cap \mathcal{B}_{\Omega} (= \tau_{4n+8})$. 

Observe that 
each component of $\varSigma_{w_{4n+8}} \setminus \tau_{4n+8}$ is either a $1$-gon with one of the punctures 
$c_1, \cdots, c_{4n+6}$, 
an $(n+2)$-gon with  the puncture $c_{4n+7}$, 
an $(n+1)$-gon with the puncture $c_{4n+8}$ or a $3$-gon without punctures. 
(``Vertical" $(n+2)$ edges of $\tau$ in Figure~\ref{fig_TrainImage}(left) bound an $(n+2)$-gon containing $c_{4n+7}$.)
Since no bigon component is  contained in $\varSigma_{w_{4n+8}} \setminus \tau_{4n+8}$, 
we conclude that 
$\tau_{4n+8}$ is a train track  which carries $\mathcal{F}_{w_{4n+8}}$. 
\end{proof}

Since $\Phi_{w_{4n+8}}: \varSigma_{w_{4n+8}} \rightarrow  \varSigma_{w_{4n+8}} $ 
is the first return map  for $\Phi^t$, 
conditions $(1)$--$(4)$ in the family $\{\tau_t\}_{0 \le t \le 1}$ 
ensure  that the image of $\tau_{4n+8}$ under the first return map $\Phi_{w_{4n+8}}$ 
is carried by $\tau_{4n+8}$, that is 
$\tau_{4n+8}$ is invariant under $\Gamma(w_{4n+8}) = [\Phi_{w_{4n+8}}]$. 
Figure~\ref{fig_TrainImage}(2) shows the image of edges of $\tau_{4n+8}$ 
under $\Phi_{w_{4n+8}}$. 
The top edges $p_1^{(n+1)}, \cdots, p_6^{(n+1)}$ map to the edge paths of the bottom and second bottom edges 
under the first return map. 
This is because these edges $p_i^{(n+1)}$'s arrive at $p_i^{(n+1)} \times \{1\} \subset \varSigma_{0,6}^1$ first 
along the flow lines. 
The identity $p_i^{(n+1)} \times \{1\} = \Phi_{w_6}(p_i^{(n+1)}) \times \{0\}$ holds 
in ${\Bbb T}_{w_6}$. 
We get the image of $p_i^{(n+1)}$ under the first return map 
when we push $ \Phi_{w_6}(p_i^{(n+1)}) \times \{0\}$ along the flow $\Phi^t$ until it hits the fiber $\varSigma_{w_{4n+8}}$. 
The rest of non-loop edges except $q_3$ map to the {\it above edge} in ${\Bbb T}_{w_6}$ 
 along the suspension flow 
 (cf. Figure~\ref{fig_train_suspend}(1)(2)). 
For example, $p_1^{(n)}$ maps to $p_1^{(n+1)}$, 
and $q_1$ maps to $p_1^{(1)}$. 
The edge $q_3$ maps to $p_3^{(1)}$ and $p_4^{(1)}$.

Let $\mathfrak{p}_{4n+8}: \tau_{4n+8} \rightarrow \tau_{4n+8}$ be the train track representative under 
$[\Phi_{w_{4n+8}}]$. 
One can check that all non-loop edges of $\tau_{4n+8}$ are real edges for $\mathfrak{p}_{4n+8}$. 
The incident matrix $M_{\mathfrak{p}_{4n+8}}$ with respect to real edges 
must be Perron-Frobenius, 
since $\tau_{4n+8}$ carries the unstable foliation of the pseudo-Anosov homeomorphism $\Phi_{w_{4n+8}}$. 
Thus the largest eigenvalue of $M_{\mathfrak{p}_{4n+8}}$ gives us  $\lambda(w_{4n+8})$. 

\begin{lem}
\label{lem_charaPoly}
For each $n \ge 0$,  $\lambda(w_{4n+8})$ equals the largest root of the polynomial 
$$ t^{6n+9} - 2 t^{5n+8} -2t^{5n+7} + 3t^{4n+6}+ 3t^{2n+3} -2t^{n+2} - 2t^{n+1}+1.$$
\end{lem}

\noindent
The proof of Lemma~\ref{lem_charaPoly} can be done by the computation of 
 the characteristic polynomial of $M_{\mathfrak{p}_{4n+8}}$. 
Alternatively one can compute $\lambda(w_{4n+8})$ 
from the {\it clique polynomial} of the {\it curve complex} $G_{4n+8}$ 
associated to the directed graph $\Gamma_{4n+8}$ 
for  $\mathfrak{p}_{4n+8}: \tau_{4n+8} \rightarrow \tau_{4n+8}$. 
In general, the curve complex $G$ associated to a directed graph $\Gamma$ 
is an undirected graph together with the weight on the set of vertices $V(G)$ of $G$. 
A consequence of results of McMullen in \cite{McMullen} tells us that 
$\tfrac{1}{\lambda(w_{4n+8})}$ equals the smallest positive root of the clique polynomial of $G_{4n+8}$. 
In this case, the topological types of the undirected graph $G_{4n+8}$ (ignoring its weight on the set of vertices) 
do not depend on $n$. 
This makes the computation of the clique polynomial of $G_{4n+8}$  straightforward. 
One can also prove  Lemma~\ref{lem_charaPoly} from the computation of the Teichm\"{u}ler polynomial 
associated to the fibered face $\Omega$ by using the invariant train track for $\Gamma(\underline{w_6})$. 
For Teichm\"{u}ler polynomials, 
see \cite{McMullen2}.  

Lemmas~\ref{lem_6braid}, \ref{lem_4n6}, \ref{lem_charaPoly} allow us to compute $\lambda(w_{2k})$ for $k \ge 3$. 
See Table~\ref{table_SH}.

\begin{center}
\begin{table}[htbp]
\caption{Computation of $\lambda(w_{2k})$ for small $k$.} 
\label{table_SH}
\begin{tabular}{|c|c|}
\hline
$$ & $\lambda(w_6) \approx  2.89005 $  \\ \hline
$$ & $\lambda(w_8) \approx 2.26844$  \\ \hline
\hline
$n \ge 1$ & $\lambda(w_{4n+8})= \lambda(w_{4n+6})$   \\ \hline
$1$ & $\approx1.56362$  \\ \hline
$2$ & $\approx1.36516$  \\ \hline
$3$ & $\approx1.27074$  \\ \hline
$4$ & $\approx 1.21532$  \\ \hline
$5$ & $ \approx 1.17882$  \\ \hline
$6$ & $\approx1.15293$  \\ \hline
$7$ & $\approx 1.13361$  \\ \hline
$8$ & $\approx 1.11863$ \\ \hline
$9$ & $\approx 1.10668$ \\ \hline
$10$ & $\approx 1.09692$ \\ \hline
$11$ & $\approx  1.08879$ \\ \hline
$12$ & $\approx  1.08193$ \\ \hline
$13$ & $ \approx 1.07605$ \\ \hline
$14$ & $ \approx 1.07096 $ \\ \hline
$15$ & $ \approx  1.06651$ \\ \hline
\end{tabular}
\end{table}
\end{center}

Remember that types of singularities of $\mathcal{F}_{w_{2n+8}}$ can be read 
from the shapes of the components of $ \varSigma_{w_{4n+8}} \setminus \tau_{4n+8}$. 
From the proof of Lemma~\ref{lem_carry}, 
we have the following.


\begin{lem}
\label{lem_foliation}
The unstable foliation $\mathcal{F}_{w_{4n+8}}$ of $\Phi_{w_{4n+8}}$ 
has properties such that 
the last puncture $c_{4n+8}$  has $(n+1)$ prongs and 
the second last puncture $c_{4n+7}$ has  $(n+2)$ prongs. 
\end{lem}

\begin{proof}[Proof of Lemma~\ref{lem_4n6}]
If $n \ge 1$, then $\mathcal{F}_{w_{4n+8}}$ 
has the property such that  last two punctures of $\varSigma_{0,4n+8}$ 
have more than $1$ prong (Lemma~\ref{lem_foliation}). 
Thus $\mathcal{F}_{w_{4n+8}}$  extends to 
the unstable foliation on $\varSigma_{0,4n+6}$ by filling last two punctures. 
This means that the pseudo-Anosov homeomorphism $\Phi_{w_{4n+8}}: \varSigma_{0, 4n+8} \rightarrow \varSigma_{0, 4n+8}$ 
extends to the pseudo-Anosov homeomorphism on $\varSigma_{0, 4n+6}$ 
which represents $\Gamma(w_{4n+6})$ 
with the same dilatation as $\Phi_{w_{4n+8}}$. 

The latter statement on  ${\Bbb T}_{w_{4n+6}}$ in  Lemma~\ref{lem_4n6} 
is clear from the definition of the braid $w_{4n+6}$. 
\end{proof}

\begin{proof}[Proof of Proposition~\ref{prop_main}] 
By Lemma~\ref{lem_4n6} together with (\ref{equation_asymp}), (\ref{equation_0}), 
we have 
 $$\lim_{n \to \infty} 2(2n+4) \log (\lambda(w_{4n+8})) = \lim_{n \to \infty} 2(2n+3) \log (\lambda(w_{4n+6}))
 =4  \log \kappa. $$
 Both sides divided by $2$ give us the desired claim. 
\end{proof}

\begin{center}
\begin{figure}
\includegraphics[width=5in]{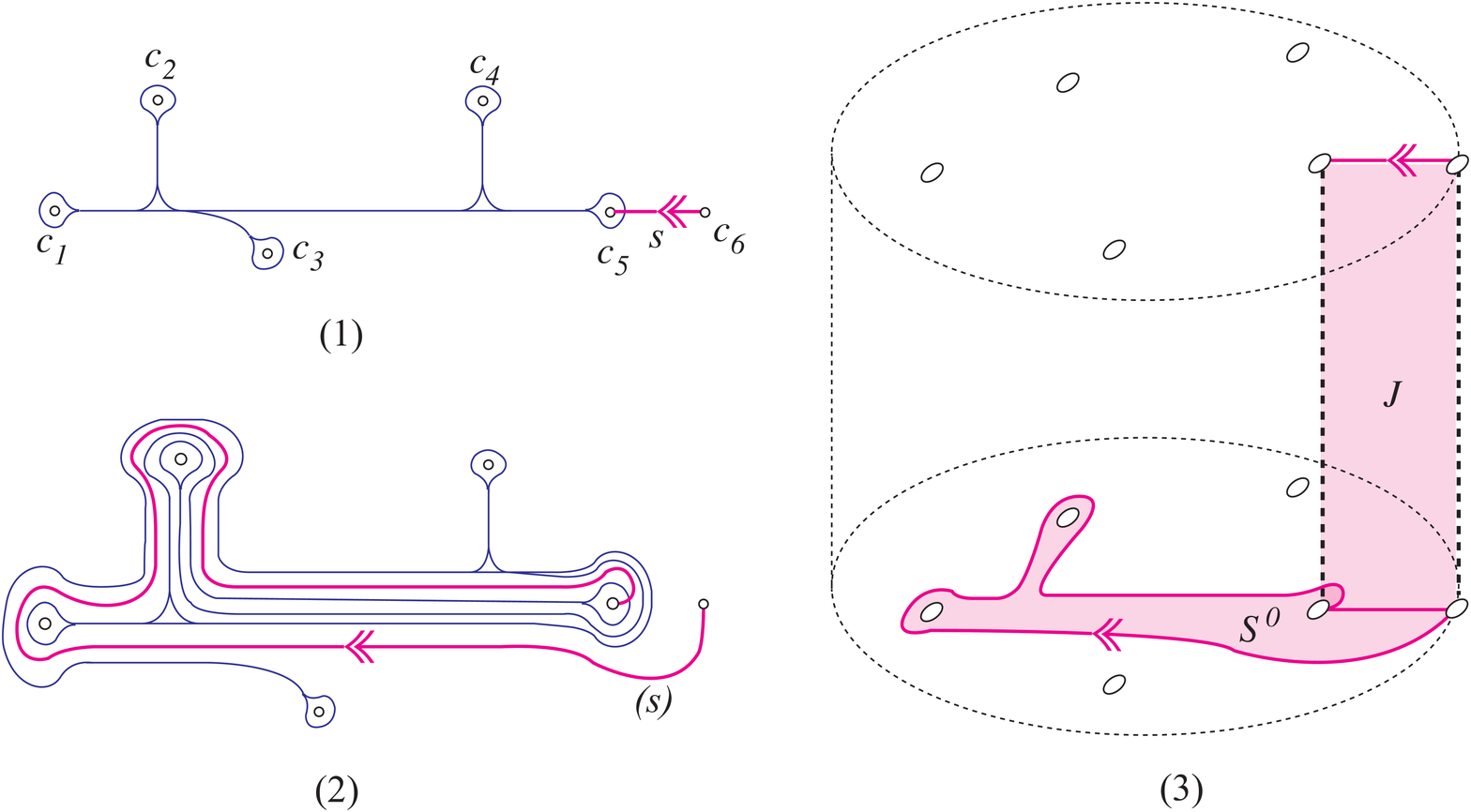}
\caption{
(1) Segment $s$, 
(2) $(s):= \Phi(s)$ up to isotopy relative to endpoints of $ \Phi(s)$. 
See also Figure~\ref{fig_train_suspend}(2). 
(3) Surface 
 $F = S^0 \cup J$ (shaded region) in  ${\Bbb T}_{w_6}$. 
To get ${\Bbb T}_{w_6}$, we glue  
$\varSigma_{0,6} \times \{1\}$ and $\varSigma_{0,6} \times \{0\}$ 
by  $\Phi \in  \Gamma(w_6)$. 
Two ``vertical" dotted lines are the orbits of $c_5$ and $c_6$ for $\Phi^t$.
Dotted two circles  (boundaries of the disks) correspond with 
the last punctures of $\varSigma_{0,6} \times \{1\}$ and $\varSigma_{0,6} \times \{0\}$.}
\label{fig_Surface}
\end{figure}
\end{center}

\begin{center}
\begin{figure}
\includegraphics[width=4.2in]{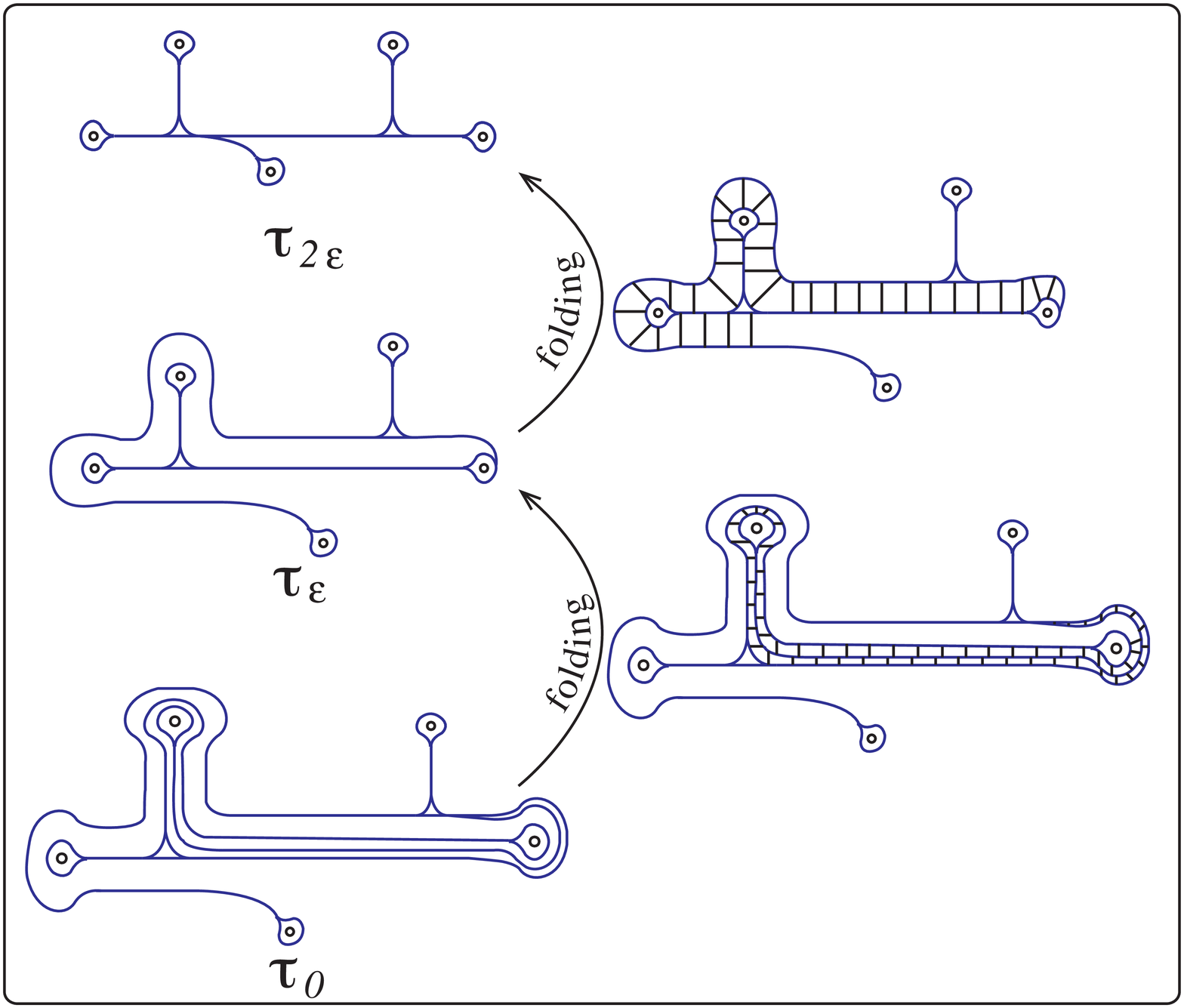}
\caption{
Left column shows train tracks $\tau_0= \Phi(\tau)$ (bottom), $\tau_{\epsilon}$ (middle), 
$\tau_{2 \epsilon} = \tau$ (top). 
$\tau$  is obtained from $\Phi(\tau)$ 
by folding edges between a cusp. 
Right column explains how to fold edges from $\tau_0$ to $\tau_{\epsilon}$ (bottom) and 
from $\tau_{\epsilon}$ to $\tau_{2 \epsilon}$ (top).} 
\label{fig_zip}
\end{figure}
\end{center}

\begin{center}
\begin{figure}
\includegraphics[width=4in]{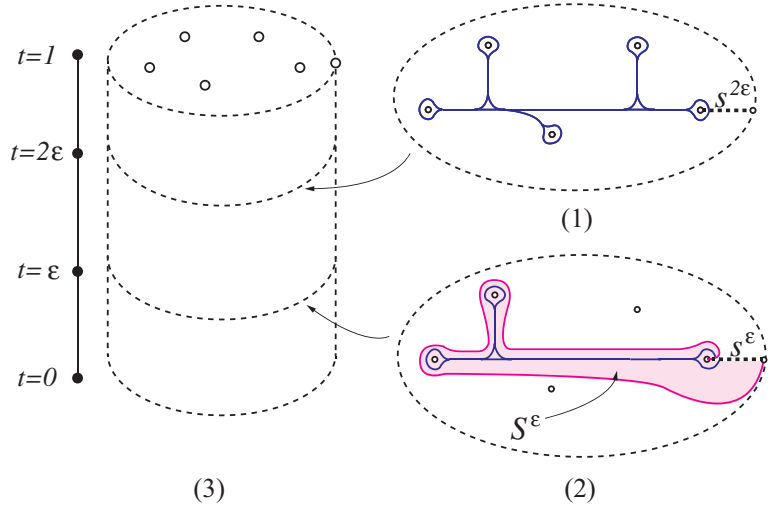}
\caption{(1)  $s^{2 \epsilon}$ (broken line) and 
$\varSigma_{0,6}^{2 \epsilon} \cap \mathcal{B}_{\Omega}$ 
(see also Figure~\ref{fig_zip}(top of left column)). 
(2) $s^{\epsilon}$ (broken line) and 
$S^{\epsilon} \cap \mathcal{B}_{\Omega} \subset S^{\epsilon} $ 
(see also Figure~\ref{fig_zip}(middle of left column). 
(3) $\varSigma_{0,6} \times [0,1] (\supset \displaystyle\bigcup_{0 \le t \le 1} \tau_t \times \{t\})$.} 
\label{fig_Bsurface}
\end{figure}
\end{center}

\begin{center}
\begin{figure}
\includegraphics[width=6.5in]{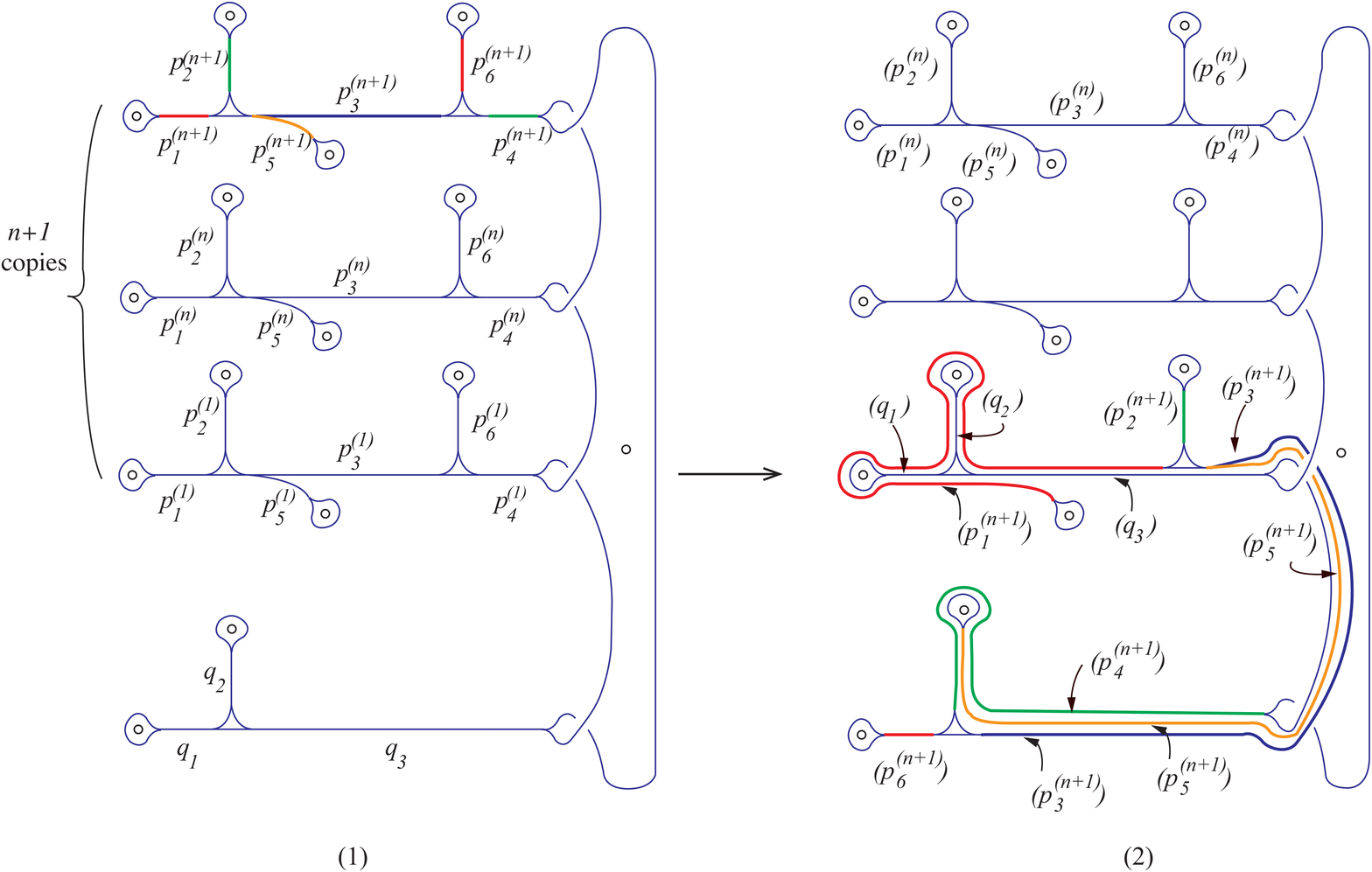}
\caption{(1)  $\tau_{4n+8}$ and  
(2) its image  $\Phi_{w_{4n+8}}(\tau_{4n+8})$  up to isotopy, 
where $n=2$ in this figure.  
Small circles indicate all punctures 
of $\varSigma_{0,4n+8}$ but the last one. 
The last puncture  corresponds to $\partial D$ of a disk $D$ such that 
$\tau_{4n+8} \subset D$.}
\label{fig_TrainImage}
\end{figure}
\end{center}

Finally we ask the following question. 

\begin{ques}[cf. Question~4.2 in \cite{Hironaka2014}]
We know from the proof of 
Proposition~\ref{prop_main} and from Lemma~\ref{lem_6braid} 
that  $4 \log ( \tfrac{1+ \sqrt{5}}{2}+ \tfrac{\sqrt{2+ 2 \sqrt{5}}}{2}) =
 2(2 \log \lambda(\underline{w_6}))$ is an accumulation point of 
 the following set of normalized entropies of pseudo-Anosov elements in $\mathcal{H}({\Bbb H}_g)$: 
$$\{\mathrm{Ent}(\phi) = (2g-2) \log \lambda(\phi)\ |\ \phi \in \mathcal{H}({\Bbb H}_g)\ \mbox{is pseudo-Anosov},\ g \ge 2\}.$$ 
Is the accumulation point $4 \log ( \tfrac{1+ \sqrt{5}}{2}+ \tfrac{\sqrt{2+ 2 \sqrt{5}}}{2})$ the smallest one? 
\end{ques}

\appendix

\section{A finite presentation of $\mathcal{H}({\Bbb H}_g)$}
\label{section_appendix} 

In this appendix, we will prove some claims referred 
in Sections~\ref{subsection_Hilden} and \ref{subsection_HypHand} 
and determine a finite presentation for $\mathcal{H}(\mathbb{H}_g)$ 
(Theorem~\ref{thm:presentaion-HHg}).  

\begin{figure}[ht]
\begin{center}
\includegraphics[height=4cm]{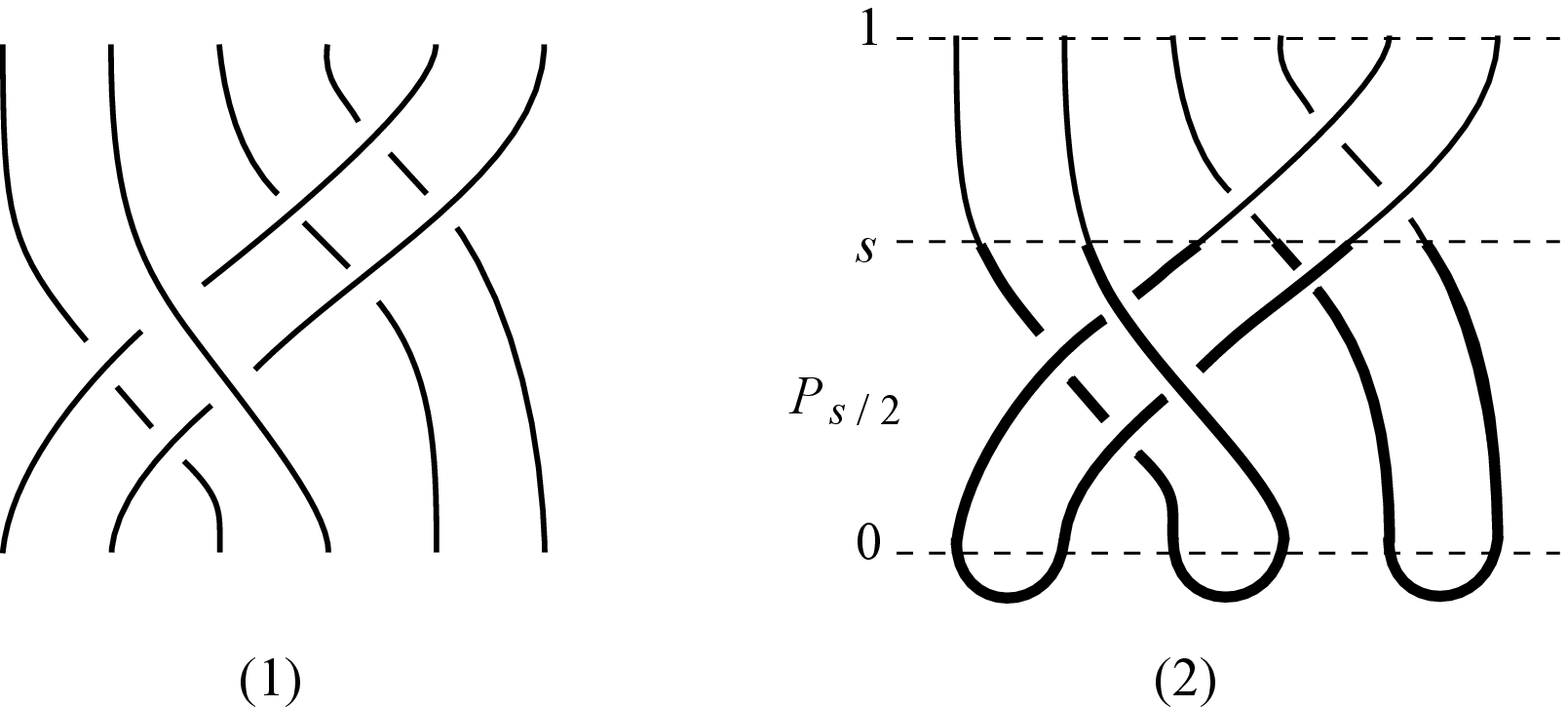}
\caption{(1) $b \in SW_{2n}$. 
(2) A path in $\mathcal{SA}_n$ corresponding to $b$.}
\label{fig:path}
\end{center}
\end{figure}

Here we make some remarks on the spherical wicket group $SW_{2n}$. 
Let $\mathcal{SA}_{n}$ be the space of configurations of $n$ disjoint smooth unknotted 
and unlinked arcs in $D^3$ with endpoints on $\partial D^3$. 
Brendle-Hatcher \cite{Brendle-Hatcher} defined the spherical wicket group to 
be $\pi_1(\mathcal{SA}_{n})$. 
We shall see in Proposition~\ref{prop:SWandPi1} that 
$\pi_1(\mathcal{SA}_{n}) \simeq SW_{2n}$. 
In \cite[p.156--157]{Brendle-Hatcher}, it is shown that the natural homomorphism 
from $\pi_1(\mathcal{SA}_{n})$ to $SB_{2n}$ induced by the map sending a configuration of $n$ arcs 
to the configuration of its endpoints is injective. 
By this injection, we regard $\pi_1(\mathcal{SA}_{n})$ as the subgroup of $SB_{2n}$. 

The wicket group $W_{2n}$ is defined as a subgroup of the braid group $B_{2n}$ 
in the same way as the definition of $SW_{2n}$ 
given in Section~\ref{subsection_Hilden}. 
Let $\mathcal{A}_n$ be the space of configurations of $n$ disjoint smooth unknotted 
and unlinked arcs in 
$\mathbb{R}^3_+ = \{(x,y,z) \in \mathbb{R}^3 \ |\  z \geq 0 \}$ with endpoints on 
$\partial \mathbb{R}^3_+$.
In the same way as $\pi_1(\mathcal{SA}_n)$,  
we regard $\pi_1(\mathcal{A}_{n})$ as a subgroup of  $B_{2n}$. 
Brendle-Hatcher \cite[Propositions 3.2, 3.6]{Brendle-Hatcher} showed that 
$\pi_1 (\mathcal{A}_n)$ is generared by 
$r_i$, $s_i$ $(i \in \{1, \cdots, n-1\})$, 
$t_j$ $(j \in \{1, \cdots, n\})$ 
shown in Figure 5. 
In the beginning of Section~6 in \cite{Brendle-Hatcher},  it is observed that $\pi_1(\mathcal{SA}_n)$ 
is the quotient of $\pi_1(\mathcal{A}_n)$ by 
the normal closure $\langle \langle \vartheta \rangle \rangle$ 
of $\{ \vartheta \}$, where $\vartheta = t_1 s_1 s_2 \cdots s_{n-1} r_{n-1}^{-1} \cdots r_2^{-1}  r_1^{-1} t_1$. 
Especially we see that  $\pi_1(\mathcal{SA}_n)$ is generated by 
$r_i$, $s_i$ and $t_j $ as above.

\begin{prop}\label{prop:SWandPi1}
$SW_{2n} = \pi_1(\mathcal{SA}_{n})$. 
\end{prop}
\begin{proof} 
Recall that $r_i$, $s_i$ $(i \in \{1, \cdots, n-1\})$, 
$t_j$ $(j \in \{1, \cdots, n\})$ 
are elements of $SW_{2n}$. 
Hence 
$\pi_1(\mathcal{SA}_{n}) \subset SW_{2n}$. On the other hand, 
for $b \in SW_{2n}$, we define a closed path $P_t (0 \leq t \leq 1)$ in $\mathcal{SA}_{n}$ 
with a base point corresponding to $n$ trivial arcs in $D^3$ as follows: 
$P_0$ and $P_1$ are $n$ trivial arcs in $D^3$, 
$P_{s/2} (0 \leq s \leq 1)$ is $n$ arcs indicated by the thick arcs in  Figure \ref{fig:path}(2)   
and the path from $P_{1/2} = {}^b\bold{A}$ to $P_1$ is an isotopy between 
$^b \bold{A}$ and $\bold{A}$ fixing end points. 
Then the sequence of endpoints of the path $P_t$ is a closed path in 
the configuration space of $2n$ points in $S^2$ whose homotopy class is the braid $b$. 
This shows $b \in \pi_1( \mathcal{SA}_{n})$. 
Hence $SW_{2n} \subset \pi_1(\mathcal{SA}_{n})$. 
\end{proof} 

In the same way as the proof of Proposition \ref{prop:SWandPi1},  
we see that $W_{2n} = \pi_1 (\mathcal{A}_n)$.  
Under the equivalences $W_{2n} = \pi_1 (\mathcal{A}_n)$ and $SW_{2n} = \pi_1(\mathcal{SA}_{n})$, 
we have the following.

\begin{lem}
\label{lem:plane-sphere}
$SW_{2n} \cong W_{2n} /\langle \langle \vartheta \rangle \rangle .$
\end{lem}

\begin{rem}\ 
\begin{enumerate}

\item[(1)] 
Brendle-Hatcher used notations $W_n$ and $SW_n$ for 
$\pi_1(\mathcal{A}_n)$ and $\pi_1(\mathcal{SA}_n)$ respectively \cite{Brendle-Hatcher}.  
In this paper, we use notations $W_{2n}$ and $SW_{2n}$ rather than $W_n$ and $SW_n$ 
for the same groups, 
because we defined $W_{2n}$ and $SW_{2n}$ as subgroups of $B_{2n}$ and $SB_{2n}$ respectively. \\ 

\item[(2)]
In \cite{Brendle-Hatcher}, 
elements of $\pi_1(\mathcal{SA}_{n})$ are applied from left to right and our convention is opposed to this. 
Hence in our paper, 
we need to take the inverse of their generators and reverse the order of letters in their relations. 
\end{enumerate}
\end{rem}

As promised in Section~\ref{subsubsection_relation}, 
we now prove the following.

\begin{prop}\label{prop:uniqueness}
Let $\psi_1$ and $\psi_2$ be homeomorphisms of $(D^3, \bold{A})$. 
If the restrictions of $\psi_1$ and $\psi_2$ over $S^2 = \partial D^3$ are isotopic as 
homeomorphisms of $(S^2, \partial \bold{A})$ then 
$\psi_1$ and $\psi_2$ are isotopic as homeomorphisms of $(D^3, \bold{A})$. 
\end{prop}
\begin{figure}[ht]
\begin{center}
\includegraphics[height=4cm]{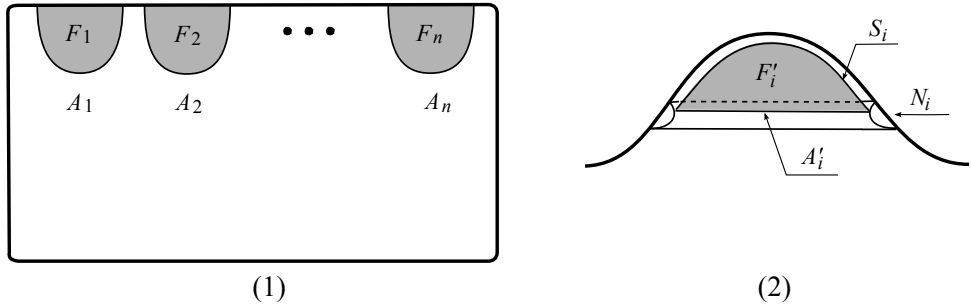}
\caption{(1) $F_i$ is a disk whose boundary is the union of the wicket $A_i$ and an arc on $\partial D^3$. 
(2) $N_i$ is a regular neighborhood of $A_i$, and $F'_i$ is a meridian disk of 
a handlebody $D^3 \setminus (N_1 \cup \cdots \cup N_n)$.  }
\label{fig:disks1}
\end{center}
\end{figure}
\begin{proof}
At first, 
we assume that $\psi_1 = id$, $\psi_2|_{\partial D^3} = id_{\partial D^3}$. 
Since $\psi_2(A_i) = A_i$ and $\psi_2|_{\partial D^3} = id_{\partial D^3}$, 
especially, $\psi_2|_{\partial A_i} = id_{\partial A_i}$, 
we can isotope $\psi_2$ so that $\psi_2|_{A_i} = id_{A_i}$ 
with an isotopy preserving $\bold{A}$ setwise. 
Furthermore, we isotope $\psi_2$ so that $\psi_2(N_i) = N_i$ 
for a regular neighborhood $N_i$ of $A_i$ in $D^3$. 
We remark that $D^3 \setminus (N_1 \cup \cdots \cup N_{n})$ is homeomorphic 
to a handlebody $\mathbb{H}_{n}$. 
The set $\partial D^3 \cap \partial N_i$ consists of two disks $d_{2i-1}$ and $d_{2i}$,  
which are neighborhoods of two points $\partial A_i$. 
The boundary $\partial (D^3 \setminus (N_1 \cup \cdots \cup N_{n}))$ 
is a union of $P = \partial D^3 \setminus (d_1 \cup \cdots \cup d_{2n})$ 
and $U_i = \partial N_i \setminus (d_{2i-1} \cup d_{2i})$.  
We consider the restriction of $\psi_2$ on $\partial (D^3 \setminus (N_1 \cup \cdots \cup N_{n}))$. 
Then $\psi_2|_{P} = id$ and $\psi_2|_{U_i}$ is isotopic to the identity or a product of the Dehn twist 
about the core of $U_i$. 
We will show that $\psi_2|_{U_i}$ is isotopic to the identity. 
Let $F_i$ be a disk in $D^3$ whose boundary is a union of 
$A_i$ and an arc on $\partial D^3$ (see Figure \ref{fig:disks1} (1)). 
Let $F_i' = F_i \setminus N_i$, then this is a meridian disk of $D^3 \setminus (N_1 \cup \cdots \cup N_{n})$ 
and its boundary is a union of two arcs $S_i = F_i' \cap P$, $A_i' = F_i' \cap U_i$ (see Figure \ref{fig:disks1} (2)). 
If we assume that $\psi_2 | _{U_i}$ is not isotopic to the identity, 
then $\psi_2(\partial F_i') = \psi_2(S_i) \cup \psi_2 (A_i') = 
S_i \cup \psi_2 (A_i')$ is not null-homotopic in $D^3 \setminus (N_1 \cup \cdots \cup N_{n})$, 
which contradicts the fact that $\psi_2(\partial F_i')$ bounds a disk $\psi_2(F_i')$ in 
$D^3 \setminus (N_1 \cup \cdots \cup N_{n})$. 
Therefore, $\psi_2|_{U_i}$ is isotopic to the identity. 
Furthermore, we can isotope $\psi_2$ so that $\psi_2|_{N_i} = id_{N_i}$. 
Since the extension of 
a homeomorphism of $\partial (D^3 \setminus (N_1 \cup \cdots \cup N_{n}))$
to the $3$-dimensional handlebody $D^3 \setminus (N_1 \cup \cdots \cup N_{n})$ 
is unique up to isotopy, we have an isotopy between 
$\psi_2|_{D^3 \setminus (N_1 \cup \cdots \cup N_{n})}$ 
and $id_{D^3 \setminus (N_1 \cup \cdots \cup N_{n})}$. 
Hence $\psi_2$ is isotopic to $id_{D^3}$ preserving $\bold{A}$ as a set. 

Next, we assume that $\psi_1|_{\partial D^3} = \psi_2|_{\partial D^3}$. 
Then $\psi_1'=\psi_1^{-1} \circ \psi_1$, $\psi_2' = \psi_1^{-1} \circ \psi_2$ 
satisfy $\psi_1'= id$, $\psi_2'|_{\partial D^3} = id_{\partial D^3}$. 
By applying the argument of the previous paragraph to $\psi_1'$ and $\psi_2'$, we have 
an isotopy $\mathbb{G}'_t \; : \; D^3 \to D^3$ $(0 \leq t \leq 1)$ 
between $\psi_1'$ and $\psi_2'$ in $\mathrm{Homeo}_+(D^3, \bold{A})$. 
Then $\mathbb{G}_t = \psi_1 \circ \mathbb{G}'_t$ is an isotopy between 
$\psi_1$ and $\psi_2$ in $\mathrm{Homeo}_+(D^3, \bold{A})$. 

Finally, we assume that $\psi_1|_{\partial D^3}$ and 
$\psi_2|_{\partial D^3}$ are isotopic in 
$\mathrm{Homeo}_+(\partial D^3, \partial \bold{A})$, 
that is to say, there is an isotopy 
$\mathbb{F}_t : \partial D^3 \to \partial D^3$ 
$(0 \leq t \leq 1)$ fixing $\partial \bold{A}$ such that 
$\mathbb{F}_0 = \psi_1 |_{\partial D^3}$, 
$\mathbb{F}_1 = \psi_2 |_{\partial D^3}$. 
We set the parametrization of the regular neighborhood $N(\partial D^3)$ 
of $\partial D^3$ by $\partial D^3 \times [0,1]$ so that 
$\partial D^3 \times \{ 0 \} = \partial D^3$, 
$\partial D^3 \times \{ 1 \} \subset \mathrm{int}(D^3)$. 
We define an isotopy $\mathbb{I}_t : D^3 \to D^3$ $(0 \leq t \leq 1)$ 
by 
$$
\mathbb{I}_t (x) 
=
\begin{cases}
(\mathbb{F}_{t(1-s)} \circ (\psi_1|_{\partial D^3})^{-1} (p), s) &\text{ if } 
x = (p,s) \in \partial D^3 \times [0,1] = N(\partial D^3), \\
x &\text { if } x \not\in N(\partial D^3). 
\end{cases}
$$
Then $\mathbb{J}_t = \mathbb{I}_t \circ \psi_1$ is an isotopy in 
$\mathrm{Homeo}_+ (D^3, \bold{A})$ so that 
$\mathbb{J}_0 = \psi_1$, 
$\mathbb{J}_1|_{\partial D^3} = \psi_2|_{\partial D^3}$. 
By the argument of the previous paragraph, there is an isotopy $\mathbb{G}_t$ 
between $\mathbb{J}_1$ 
and $\psi_2$ in $\mathrm{Homeo}_+(D^3, \bold{A})$. 
The concatenation of $\mathbb{J}_t$ and $\mathbb{G}_t$ is an isotopy 
between $\psi_1$ and $\psi_2$ in $\mathrm{Homeo}_+ (D^3, \bold{A})$. 
\end{proof}

\begin{figure}[ht]
\begin{center}
\includegraphics[height=7.5cm]{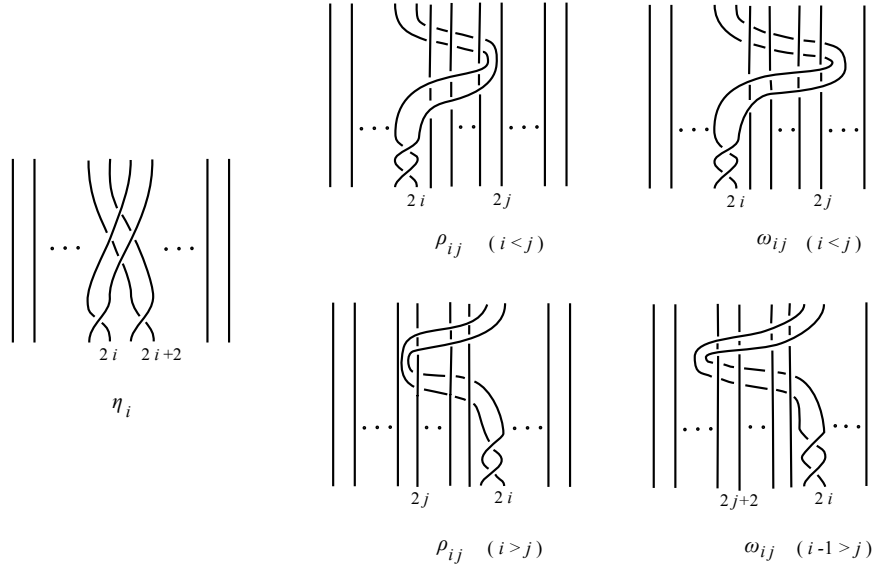}
\caption{$\eta_i$, $\rho_{i,j}$, $\omega_{i,j}$ 
from left to right.  
If $i < j$, then $\rho_{i,j}$ and $\omega_{i,j}$ are the braids on the top. 
If $i>j$ (resp. If $i-1> j$), then $\rho_{i,j}$ (resp. $\omega_{i,j}$) are the braids on the bottom. } 
\label{fig:Hilden}
\end{center}
\end{figure}

We are now ready to prove Theorem~\ref{thm_HiWi} 
as promised  in Section~\ref{subsubsection_relation}.

\begin{proof}[Proof of Theorem~\ref{thm_HiWi}] 
By Proposition \ref{prop:uniqueness}, we regard $\pi_0(\mathrm{Homeo}_+(D^3, \bold{A}))$ 
as a subgroup $SH_{2n}$ of $\mathrm{Mod}(\varSigma_{0,2n})$. 
The following sequence is exact (see \cite[p.245]{Farb-Margalit} for example).   
$$
0 \to \langle \Delta^2 \rangle \to SB_{2n} \overset{\Gamma}{\to} \mathrm{Mod}(\varSigma_{0,2n}) 
\to 1. 
$$
As an immediate consequence of Theorem 5 in \cite{Hilden},
we see that $SH_{2n}$ is generated by
$\Gamma (\sigma_{2i-1})$ $(i = 1, \ldots, n)$, $\Gamma(\eta_i)$ $(i=1,
\ldots, n-1)$,
$\Gamma(\rho_{ij})$ $(i,j=1, \ldots, n, j \not= i)$,
and $\Gamma(\omega_{ij})$ $(i,j=1, \ldots, n, j \not= i-1, i)$,
where
$\eta_i$,  $\rho_{ij}$, $\omega_{ij}$ are as shown in Figure \ref{fig:Hilden}.
We remark that in the case of the braid $\rho_{ij}$,  
the $(2i-1)$st and $2i$th strings pass 
between the $(2j-1)$st and $2j$th strings. 
On the other hand, 
in the case of the braid $\omega_{ij}$, the $(2i-1)$st and $2i$th strings
pass between the $2j$th and $(2j+1)$st strings.
As products of  $r_i$, $s_i$, $t_j$, these braids are expressed as follows,
$$
\begin{aligned}
\eta_i & = s_i t_i t_{i+1}, \\
\rho_{ij} & =
\begin{cases} s_i s_{i+1} \cdots s_{j-2} s_{j-1} r_{j-1} s_{j-2}
\cdots s_{i+1} s_i t_i^2
&\text{ if } i < j, \\
s_{i-1} s_{i-2} \cdots s_{j+1} s_j r_j^{-1} s_{j+1} \cdots s_{i-2}
s_{i-1} t_i^2
&\text{ if } i>j,
\end{cases}
\\
\omega_{ij} & =
\begin{cases} s_i s_{i+1} \cdots s_{j-2} s_{j-1}^2 s_{j-2} \cdots
s_{i+1} s_i t_i^2
&\text{ if } i<j, \\
s_{i-1} s_{i-2} \cdots s_{j+2} s_{j+1}^2 s_{j+2} \cdots s_{i-2} s_{i-1} t_i^2
&\text{ if } i-1>j.
\end{cases}
\end{aligned}
$$
On the other hand, 
Brendle-Hatcher \cite{Brendle-Hatcher} showed that 
$\pi_1(\mathcal{SA}_n)(=SW_{2n})$ is generated by 
$r_i$, $s_i$ 
$(i \in \{1, \cdots, n-1\})$, 
$t_{j}$ $(1 \in \{1, \cdots, n\})$. 
The images of these generators by $\Gamma$ are in $SH_{2n}$,  and 
$\Gamma(\eta_i)$, $\Gamma(\rho_{ij})$, $\Gamma(\omega_{ij})$ are written by  products of 
these images. 
Therefore we see that $\Gamma(SW_{2n}) = SH_{2n} $. 
On the other hand, $\Delta^2 = (s_{n-1} \cdots  s_2  s_1 t_1^2)^{n}$ 
is in $SW_{2n}$, and hence $SW_{2n} = \Gamma^{-1} (SH_{2n})$.  
As a result, Theorem~\ref{thm_HiWi} holds. 
\end{proof}

Let 
$\mathrm{SHomeo}_+(\varSigma_g)$ be the subgroup of  $\mathrm{Homeo}_+(\varSigma_g)$ 
which consists of 
the orientation preserving homeomorphisms on $\varSigma_g \simeq \partial {\Bbb H}_g$ 
that commute with $\mathcal{S}: \partial{\Bbb H}_g \rightarrow \partial{\Bbb H}_g$.
In order to prove Theorem~\ref{thm_BH}, Birman-Hilden showed the following.

\begin{prop}[Theorem~7 in \cite{Birman-Hilden}]
\label{prop:sym-isotopy-BH}
Let $\phi_1$ and $\phi_2 \in \mathrm{SHomeo}_+(\varSigma_g)$ be isotopic in 
$\mathrm{Homeo}_+(\varSigma_g)$. 
Then $\phi_1$ and $\phi_2$ are isotopic in $\mathrm{SHomeo}_+(\varSigma_g)$. 
\end{prop}

\noindent 
By Proposition \ref{prop:sym-isotopy-BH}, 
the natural surjection from $\pi_0(\mathrm{SHomeo}_+(\varSigma_g))$ 
to $\mathcal{H}(\varSigma_g)$ is an isomorphism. 
Therefore, one can define a homomorphism 
$q : \mathcal{H}(\varSigma_g) \to SB_{2g+2}$, see Theorem~\ref{thm_BH}.  

Recall that 
$\mathrm{SHomeo}_+({\Bbb H}_g)$ is the subgroup of $\mathrm{Homeo}_+({\Bbb H}_g)$ 
which consists of  orientation preserving homeomorphisms on ${\Bbb H}_g$ 
that  commute with $\mathcal{S}: {\Bbb H}_g \rightarrow {\Bbb H}_g$. 
We have the following  which is a version of Proposition \ref{prop:sym-isotopy-BH}.

\begin{prop}\label{prop:sym-isotopy}
Let $\phi_1$ and $\phi_2 \in \mathrm{SHomeo}_+(\mathbb{H}_g)$ be isotopic in 
$\mathrm{Homeo}_+(\mathbb{H}_g)$. 
Then $\phi_1$ and $\phi_2$ are isotopic in $\mathrm{SHomeo}_+(\mathbb{H}_g)$.  
\end{prop}

\begin{proof}
For $\phi \in \mathrm{SHomeo}_+(\mathbb{H}_g)$, we define a homeomorphism
$\underline{\phi}$ of $D^3 = \mathbb{H}_g/\iota$ by
$\underline{\phi}([x]) = [\phi(x)]$, where $[x]$ is an element of $D^3
= \mathbb{H}_g/\iota$
represented by $x \in \mathbb{H}_g$.
By Proposition \ref{prop:sym-isotopy-BH},
there is an isotopy in $\mathrm{SHomeo}_+(\varSigma_g)$
between $\phi_1|_{\partial \mathbb{H}_g}$ and $\phi_2|_{\partial \mathbb{H}_g}$.
This isotopy induces an isotopy between
$\underline{\phi_1}|_{\partial D^3}$ and
$\underline{\phi_2}|_{\partial D^3}$
in $\mathrm{Homeo}_+(\partial D^3, \partial \bold{A})$.
By Proposition \ref{prop:uniqueness}, there is an isotopy between
$\underline{\phi_1}$ and $\underline{\phi_2}$ in
$\mathrm{Homeo}_+(D^3, \bf{A})$.
Then the lift of this isotopy is an isotopy in $\mathrm{SHomeo}_+({\Bbb H}_g)$
between
$\phi_1$ and $\phi_2$.
%
\end{proof}

We are now ready to prove Theorem~\ref{thm_HiKi}.

\begin{proof}[Proof of Theorem~\ref{thm_HiKi}] 
By Proposition \ref{prop:sym-isotopy}, 
the natural surjection from $\pi_0(\mathrm{SHomeo}_+(\mathbb{H}_g))$ 
to $\mathcal{H}(\mathbb{H}_g)$ is an isomorphism. 
Therefore, we can define a homomorphism $Q : \mathcal{H}(\mathbb{H}_g) \to SH_{2g+2}$ 
so that $Q= q|_{\mathcal{H}({\Bbb H}_g)}$, 
see (\ref{equation_BH}) in Section~\ref{subsection_HypHand}. 
As a consequence of Theorem~\ref{thm_BH} and the fact that $\iota \in \mathcal{H}(\mathbb{H}_g)$, 
we see that Theorem~\ref{thm_HiKi} holds. 
\end{proof}

\begin{figure}[ht]
\begin{center}
\includegraphics[height=3.5cm]{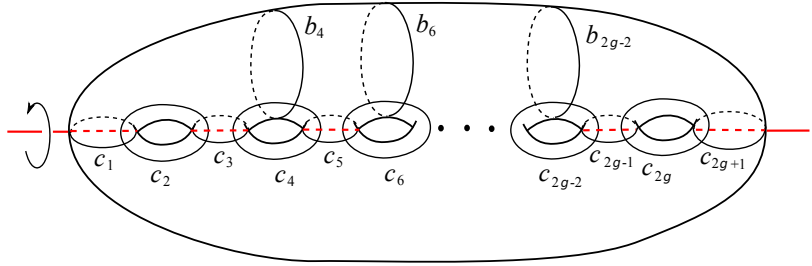}
\caption{Circles on $\partial \mathbb{H}_g$. }
\label{fig:handle1}
\end{center}
\end{figure}

As an application of Theorem~\ref{thm_HiKi}, 
we  determine a finite presentation for $\mathcal{H}({\Bbb H}_g)$ (Theorem~\ref{thm:presentaion-HHg}).  
To do this, we set some circles 
on $\partial \mathbb{H}_g$ as in Figure \ref{fig:handle1}. 
The circle $c_{2j-1}$ $(j \in \{1, \cdots, g+1\})$ 
bounds a disk properly embedded 
in $\mathbb{H}_g$, and $c_{2j-1}$ 
is preserved by the hyperelliptic involution $\mathcal{S}$. 
The circle $b_{2j}$ 
$(j \in \{2, \cdots, g-1\})$ 
also bounds a disk properly 
embedded in $\mathbb{H}_g$, 
but $b_{2j}$ is not preserved by $\mathcal{S}$. 
Let $t_{c_i}$ and $t_{b_{2j}}$ be the left-handed Dehn twist about $c_i$ and $b_{2j}$ 
respectively. 

\begin{rem}
\label{rem_index}
The group $\mathrm{Mod}(\mathbb{H}_g)$ is a subgroup of the mapping class group 
of $\partial \mathbb{H}_g$ of infinite index whenever $g \geq 2$. 
This is because $t_{c_2}$ is not an element of $\mathrm{Mod}(\mathbb{H}_g)$ and has an infinite order. 
The group $\mathcal{H}(\mathbb{H}_g)$ is a subgroup of $\mathrm{Mod}(\mathbb{H}_g)$ of infinite index  
whenever $g \geq 3$. 
In fact, $t_{b_4}$ is not an element of $\mathcal{H}(\mathbb{H}_g)$ but an element of $\mathrm{Mod}(\mathbb{H}_g)$, 
and $t_{b_4}$  has an infinite order. 
\end{rem}

\begin{thm}\label{thm:presentaion-HHg}
$\mathcal{H}({\Bbb H}_g)$ is generated by 
$\frak{r}_i = t_{c_{2i}} t_{c_{2i+1}} t_{c_{2i-1}}^{-1} t_{c_{2i}}^{-1}$, 
$\frak{s}_i = t_{c_{2i}}^{-1} t_{c_{2i+1}}^{-1} t_{c_{2i-1}}^{-1} t_{c_{2i}}^{-1}$ $(i = 1, \ldots, g)$, 
$\frak{t}_j = t_{c_{2j-1}}^{-1}$ $(j=1, \ldots, g, g+1)$ and the relations are as follows. 
\begin{enumerate}
\item[(1)] 
$\frak{r}_i \frak{r}_j = \frak{r}_j \frak{r}_i $ for $|i-j|>1$, 
$\frak{r}_i \frak{r}_{i+1} \frak{r}_i = \frak{r}_{i+1} \frak{r}_i \frak{r}_{i+1}$, 

\item[(2)] 
$\frak{s}_i \frak{s}_j = \frak{s}_j \frak{s}_i$ for $|i-j|>1$, 
$\frak{s}_i \frak{s}_{i+1} \frak{s}_i = 
\frak{s}_{i+1} \frak{s}_i \frak{s}_{i+1}$, 

\item[(3)] 
$\frak{r}_i \frak{s}_j = \frak{s}_j \frak{r}_i$ for $|i-j|>1$, 

\item[(4)] 
$\frak{r}_i \frak{s}_{i+1} \frak{s}_i = \frak{s}_{i+1} \frak{s}_i \frak{r}_{i+1}$, 
$\frak{r}_i \frak{r}_{i+1} \frak{s}_i = \frak{s}_{i+1} \frak{r}_i \frak{r}_{i+1}$, 
$\frak{s}_i \frak{s}_{i+1} \frak{r}_i = \frak{r}_{i+1} \frak{s}_i \frak{s}_{i+1}$, 

\item[(5)] 
$\frak{r}_i \frak{s}_i \frak{t}_i \frak{r}_i = \frak{t}_i \frak{s}_i$, 

\item[(6)] 
$\frak{t}_i \frak{t}_j = \frak{t}_j \frak{t}_i$, 

\item[(7)] 
$\frak{r}_i \frak{t}_j = \frak{t}_j \frak{r}_i$ for $j \not= i, i+1$, 
$\frak{t}_{i+1} \frak{r}_i= \frak{r}_i \frak{t}_i$, 

\item[(8)] 
$\frak{s}_i \frak{t}_j = \frak{t}_j \frak{s}_i$ for $j \not= i, i+1$, 
$\frak{t}_j \frak{s}_i= \frak{s}_i \frak{t}_k$, for $\{ i,i+1 \} = \{ j, k \}$, 

\item[(9)] 
$(\frak{s}_g \cdots  \frak{s}_2  \frak{s}_1 \frak{t}_1^2)^{g+1} = 1$, 

\item[(10)]
$(\frak{t}_1 \frak{s}_1 \frak{s}_2 \cdots \frak{s}_g \frak{r}_g^{-1} 
\cdots \frak{r}_2^{-1}  \frak{r}_1^{-1} \frak{t}_1)^2=1$ 
and  
$ \frak{t}_1 \frak{s}_1 \frak{s}_2 \cdots \frak{s}_g \frak{r}_g^{-1} \cdots 
\frak{r}_2^{-1}  \frak{r}_1^{-1} \frak{t}_1$  commutes with 
$\frak{r}_i, \frak{s}_i, \frak{t}_i$. 
\end{enumerate}
\end{thm}

\begin{proof}
We use Theorems~\ref{thm_HiWi} and \ref{thm_HiKi}. 
Brendle-Hatcher expressed a finite presentation for 
$\pi_1(\mathcal{A}_{g+1})(=W_{2g+2})$ in 
\cite[Propositions~3.2, 3.6]{Brendle-Hatcher}. 
The relations (1)--(4) come from \cite[Proposition~3.2]{Brendle-Hatcher} and 
(5)--(8) come from \cite[Proposition~3.6]{Brendle-Hatcher}.  
The relation (9) means that $\Delta^2$ is trivial in $SH_{2g+2}$. 
In the relation (10), $\frak{t}_1 \frak{s}_1 \frak{s}_2 \cdots \frak{s}_g \frak{r}_g^{-1} 
\cdots \frak{r}_2^{-1}  \frak{r}_1^{-1} \frak{t}_1$ equals $\iota$, and 
the relation means $\iota^2=1$ and any element of $\mathcal{H}({\Bbb H}_g)$ commutes 
with $\iota$. 
\end{proof}

By a straightforward computation together with Theorem~\ref{thm:presentaion-HHg}, 
we have the following. 
\begin{cor}\label{cor:abelianization}
The abelianization 
$\mathcal{H}({\Bbb H}_g) ^{ab} =\mathcal{H}({\Bbb H}_g)/[\mathcal{H}({\Bbb H}_g), \mathcal{H}({\Bbb H}_g)]$ 
is isomorphic to $\mathbb{Z} \oplus \mathbb{Z}_2 \oplus \mathbb{Z}_2$ 
for any $g \geq 2$. 
\end{cor} 

\noindent 
Corollary~\ref{cor:abelianization} is in contrast with the abelianizations of 
other groups which contain $\mathcal{H}({\Bbb H}_g)$ as a subgroup. 
In fact, $\mathrm{Mod}(\varSigma_1)^{ab} =\mathbb{Z}/12\mathbb{Z}$, 
$\mathrm{Mod}(\varSigma_2)^{ab} = \mathbb{Z}/10 \mathbb{Z}$, and 
$\mathrm{Mod}(\varSigma_g)^{ab}$ is trivial when $g \geq 3$ 
(see \cite[\S 5.1]{Farb-Margalit} for example).  
In the case of the hyperelliptic mapping class groups, 
$\mathcal{H}(\varSigma_g)^{ab} = \mathbb{Z}/2(2g+1)\mathbb{Z}$ when $g$ is even and 
$\mathcal{H}(\varSigma_g)^{ab}= \mathbb{Z}/4(2g+1)\mathbb{Z}$ when $g$ is odd. 
They are  proved straightforwardly from the presentation of $\mathcal{H}(\varSigma_g)$ 
by Birman-Hilden \cite[Theorem 8]{Birman-Hilden}. 
For the handlebody groups, 
$\mathrm{Mod}(\mathbb{H}_g)^{ab}$ is a finite abelian group when $g \geq 3$, 
see \cite{Wajnryb,Hirose}.


\begin{thebibliography}{99}


\bibitem{ALM}
I.~Agol, C.~ J. Leininger and D.~Margalit, 
{\it Pseudo-Anosov stretch factors and homology of mapping tori}, 
J. London Math. Soc. 93 Number 3 (2016),  664-682. 

\bibitem{Birman} 
J.~Birman, 
Braids, Links and Mapping Class Groups, 
Annals of Math Studies 82, Princeton University Press (1975).  

\bibitem{Birman-Hilden}
J.~Birman and H.~Hilden, 
{\it On mapping class groups of closed surfaces as covering spaces}, 
Advances in the theory of Riemann surfaces, 
Annals of Math Studies 66, Princeton University Press
(1971), 81-115. 

\bibitem{BH}
M~Bestvina, M~Handel, 
{\it Train--tracks for surface homeomorphisms}, 
Topology 34 (1994) 1909-140. 
 
 \bibitem{Brendle-Hatcher} 
 T.~E.~Brendle and A.~Hatcher, 
 {\it Configuration spaces of rings and wickets}, 
 Commentarii Mathematici Helvetici  88, 1 (2013), 131-162.
 
 \bibitem{BM}
 T.~Brendle, D.~Margalit, 
 {\it Factoring in the hyperelliptic Torelli group}, 
 To appear in Mathematical Proceedings of the Cambridge Philosophical Society, 
 159, 2 (2015) 207-217. 
 
  
 \bibitem{Calegari} 
 D.~Calegari, 
 Foliations and the geometry of $3$-manifolds 
 (Oxford Mathematical Monographs), 
 Oxford University Press (2007). 
 
 
 \bibitem{Damiani}
 C.~Damiani, 
 {\it A journey through loop braid groups}, 
 To appear in  Expositiones Mathematicae. 
 

\bibitem{FLM}
B.~ Farb, C.~ J.~ Leininger and D.~ Margalit, 
{\it The lower central series and pseudo-Anosov dilatations}, 
American Journal of Mathematics  130, Number 3 (2008), 799-827. 


\bibitem{Farb-Margalit} 
B.~Farb and D.~Margalit, 
A primer on mapping class groups, 
Princeton Mathematical Series 49, Princeton University Press, 
Princeton, NJ (2012). 


\bibitem{FLP}
A.~Fathi, F.~Laudenbach and V.~Poenaru, 
Travaux de Thurston sur les surfaces,
Ast\'erisque, 66-67, 
Soci\'et\'e Math\'ematique de France, Paris (1979). 

\bibitem{FM}
A.~T.~Fomenko and S.~V.~Matveev, 
Algorithmic and Computer Methods for Three-Manifolds, 
Kluwer Academic Publishers, Dordrecht  (1997). 

\bibitem{Fried}
D.~Fried, 
{\it Fibrations over $S^1$ with pseudo-Anosov monodromy}, 
Expos\'e~14 in `Travaux de Thurston sur les surfaces' by A.~Fathi, F.~Laudenbach and V.~Poenaru,
Ast\'erisque, 66-67, 
Soci\'et\'e Math\'ematique de France, Paris (1979), 251-266.  



\bibitem{Hall}
T.~Hall, 
The software ``Trains" is available at 
\verb#http://www.liv.ac.uk/~tobyhall/T_Hall.html#


\bibitem{Hilden} 
H.~M.~Hilden, 
{\it Generators for two groups related to the braid group}, 
Pacific Journal of Mathematics  59, Number 2 (1975), 475-486. 




\bibitem{Hironaka2014}
E.~Hironaka, 
{\it Penner sequences and asymptotics of minimum dilatations for subfamilies of the mapping class group}, 
 Topology Proceedings 44 
(2014), 315-324. 



\bibitem{Hironaka2}
E.~Hironaka, 
{\it Quotient families of mapping classes}, preprint (2012), 
arXiv:1212.3197(math.GT)


\bibitem{HK}
E.~Hironaka and E.~Kin, 
{\it A family of pseudo-Anosov braids with small dilatation}, 
Algebraic and Geometric Topology 6 (2006), 699-738. 


\bibitem{Hirose} 
S.~Hirose, 
{\it Abelianization and Nielsen realization problem of the mapping class group of handlebody}, 
Geometriae Dedicata 157 (2012), 217-225. 


\bibitem{Ivanov}
N.~V.~Ivanov,
{\it Stretching factors of pseudo-Anosov homeomorphisms},
Journal of Soviet Mathematics, 52 (1990), 2819--2822, which is translated from
Zap. Nauchu. Sem. Leningrad. Otdel. Mat. Inst. Steklov.
(LOMI), 167 (1988), 111-116. 




\bibitem{Kin}
E.~Kin, 
{\it Dynamics of the monodromies of the fibrations on the magic $3$-manifold},  
New York Journal of Mathematics 21 (2015) 547-599. 



\bibitem{KnotAtlas} 
The Knot Atlas, 
\verb#http://katlas.org/wiki/The_Thistlethwaite_Link_Table#


\bibitem{McMullen} 
C.~McMullen, 
{\it Entropy and the clique polynomial}, 
Journal of  Topology, Number 8 (1) (2015), 184-212. 

\bibitem{McMullen2}
C.~McMullen, 
{\it Polynomial invariants for fibered $3$-manifolds and Teichm\"{u}ler geodesic for foliations}, 
Annales Scientifiques de l'\'{E}cole Normale Sup\'{e}rieure. Quatri\`{e}me S\'{e}rie  33 (2000), 519-560. 


\bibitem{Oertel} 
U.~Oertel, 
{\it Homology branched surfaces: Thurston's norm on $H_2(M^3)$}, 
LMS Lecture Note Series 112, Low-dimensional Topology and Kleinian Groups, Editor D.~B.~A.~ Epstein (1986), 
253-272. 



\bibitem{PaPe}
A.~Papadopoulos and R.~ Penner, 
{\it A characterization of pseudo-Anosov foliations}, 
Pacific Journal of Mathematics 130 (2) (1987), 359-377. 


\bibitem{Penner}
R.~C.~Penner, 
{\it Bounds on least dilatations}, 
Proceedings of the American Mathematical Society 113 (1991), 443-450. 





\bibitem{Song}
W.~T.~Song, 
{\it Upper and lower bounds for the minimal positive entropy of pure braids}, 
The Bulletin of the London Mathematical Society 
 37, Number 2 (2005), 224-229. 
 
 
\bibitem{Suzuki} 
S.~Suzuki, 
{\it On homeomorphisms of a 3-dimensional handlebody}, 
Canadian Journal of Mathematics 
29, Number 1 (1977), 111-124. 
 


\bibitem{Tawn} 
 S.~Tawn, 
{\it  A presentation for Hilden's subgroup of the braid group}. 
Mathematical Research Letters 
  15, Number 6, (2008), 1277-1293. 
{\it   Erratum: A presentation for Hilden's subgroup of the braid group}. 
Mathematical Research Letters 
18, Number 1 (2011), 175-180. 



 
 \bibitem{Thurston1}
W.~Thurston, 
{\it A norm of the homology of $3$-manifolds}, 
Memoirs of the American Mathematical Society 339 (1986), 99-130. 

\bibitem{Thurston2} 
W.~Thurston, 
{\it On the geometry and dynamics of diffeomorphisms of surfaces}, 
Bulletin of the American Mathematical Society 19 (1988),  417-431.

 
 \bibitem{Thurston3} 
W.~Thurston, 
{\it Hyperbolic structures on 3-manifolds II: Surface groups and 
3-manifolds which fiber over the circle}, preprint, 
arXiv:math/9801045




\bibitem{Valdivia}
A.~D.~Valdivia, 
{\it Sequences of pseudo-Anosov mapping classes and their asymptotic behavior}, 
New York Journal of Mathematics 18 (2012), 609-620.


\bibitem{Wajnryb}  
B.~Wajnryb, 
{\it Mapping class group of a handlebody}, 
Fundamenta Mathematicae 158 (1998), 195-228.

\end{thebibliography}
\end{document}